\PassOptionsToPackage{top=2.5cm, bottom=2cm, left=2cm, right=2cm}{geometry}
\documentclass[final,3p,times]{elsarticle}

\makeatletter
\def\ps@pprintTitle{%
   \let\@oddhead\@empty
   \let\@evenhead\@empty
   \def\@oddfoot{\reset@font\hfil\thepage\hfil}
   \let\@evenfoot\@oddfoot
}
\makeatother

\usepackage{amsthm, amssymb, amsfonts, mathrsfs, amsmath}
\usepackage{lmodern}
\usepackage[T5]{fontenc}
\usepackage{amscd}
\usepackage[v2,cmtip]{xy}
\usepackage{graphicx}

\usepackage[x11names]{xcolor}
\usepackage[
  colorlinks,
]{hyperref}

\newcommand\rurl[1]{%
  \href{http://#1}{\nolinkurl{#1}}%
}

\AtBeginDocument{\hypersetup{citecolor=magenta, urlcolor=magenta, linkcolor=blue}}

\theoremstyle{plain}
\newtheorem{thm}{Theorem}[section]
\newtheorem{conj}[thm]{Conjecture}

\newtheorem{corl}[thm]{Corollary}
\theoremstyle{definition}

\theoremstyle{plain}
\newtheorem{therm}{Theorem}[subsection]
\newtheorem{propo}[therm]{Proposition}
\newtheorem{lema}[therm]{Lemma}

\theoremstyle{definition}

\theoremstyle{definition}
\newtheorem{dn}{Definition}[subsection]

\theoremstyle{plain}
\newtheorem{bd}[dn]{Lemma}
\newtheorem{theo}[dn]{Theorem}

\def\leq{\leqslant}
\def\geq{\geqslant}
\def\DD{D\kern-.7em\raise0.4ex\hbox{\char '55}\kern.33em}

\usepackage{sectsty}
\subsectionfont{\fontsize{11.5}{11.5}\selectfont}

\makeatletter
\def\blfootnote{\xdef\@thefnmark{}\@footnotetext}
\makeatother

\begin{document}
\fontsize{11.5pt}{11.5}\selectfont

\begin{frontmatter}

\title{The "hit" problem of five variables in the generic degree \\ and its application}

\author{\DD\d{\u a}ng V\~o Ph\'uc}
\address{{\fontsize{10pt}{10}\selectfont Faculty of Education Studies, University of Khanh Hoa,\\ 01 Nguyen Chanh, Nha Trang, Khanh Hoa, Viet Nam\\[1mm]  \textit{(Dedicated to Professor James F. Peters)}}}
\ead{dangphuc150488@gmail.com, dangvophuc@ukh.edu.vn}

\begin{abstract}
Let $P_s:= \mathbb F_2[x_1,x_2,\ldots ,x_s]$ be the graded polynomial algebra over the prime field of two elements, $\mathbb F_2$, in $s$ variables $x_1, x_2, \ldots , x_s$, each of degree one. This algebra is considered as a graded module over the  mod-2 Steenrod algebra, $\mathscr {A}$. We are interested in the {\it "hit" problem} of finding a minimal set of generators for $\mathscr A$-module $P_s.$ This problem is unresolved for every $s\geqslant 5.$ In this paper, we study the hit problem of five variables in a generic degree, from which we investigate Singer's conjecture [Math. Z. 202 (1989), 493-523] for the transfer homomorphism of rank $5$ in degrees given. This gives an efficient method to study the algebraic transfer and it is different from the ones of Singer.

\end{abstract}



\end{frontmatter}

\tableofcontents

\section{Introduction and main results}\label{s1}
\setcounter{equation}{0}

Let $\mathscr A$ denote the mod-2 Steenrod algebra. This algebra is defined to be the graded algebra over the prime field $\mathbb F_2,$ generated by the Steenrod squares $Sq^i,$ in grading $i\geqslant 0,$ subject to the Adem relations and $Sq^0  = 1$ (see Steenrod and Epstein \cite{SE}). From a topological point of view, the Steenrod algebra is the algebra of stable cohomology operations for ordinary cohomology $H^*$ over $\mathbb F_2.$ Let us denote by $P_s:= \mathbb{F}_2[x_1,x_2,\ldots ,x_s] = \bigoplus_{d\geqslant 0}(P_s)_d$ the polynomial algebra over $\mathbb F_2$, which can be regarded as a graded left module over $\mathscr A$. The grading is by the degree of the homogeneous terms $(P_s)_d$ of degree $d$ in $s$ variables with the degree of each $x_i$ being $1.$ The algebra $P_s$ arises as the cohomology with $\mathbb F_2$-coefficients of the elementary abelian 2-group $\mathbb F_2^{s}$ of rank $s.$ As well known, the Steenrod algebra acts by the composition of linear operators on $P_s$ and the action of the Steenrod squares $Sq^i: (P_s)_d\to (P_s)_{d+i}$  is determined by the Cartan formula and its elementary properties (see also \cite{SE}, \cite{R.W2}).

A homogeneous element $f\in (P_s)_d$ in $\mathscr A$-module $P_s$  is {\it hit} if there is a finite sum $f = \sum_{j > 0}Sq^j(f_j),$  where the homogeneous elements $f_j\in P_s$ have grading strictly less than $d.$ Let us denote by $QP_s:= \mathbb F_2 \otimes_{\mathscr A} P_s$ the quotient of the left $\mathscr A$-module $P_s$ by the hit elements in $\mathscr A^+\cdot P_s,$ where $\mathscr A^+$ denotes the augmentation ideal of $\mathscr A$ and $\mathbb F_2$ is viewed as a right $\mathscr A$-module concentrated in grading $0.$ Then, $QP_s$ is a graded vector space over $\mathbb F_2$ and a basis for $QP_s$ lifts to a minimal generating set for $P_s$ as a module over $\mathscr A.$ It should be noted that this space is also considered as a form modular representation of $GL_s$ over $\mathbb F_2.$ We are interested in the \textit{hit problem} of finding a monomial basis of $QP_s$ in each $s$ and degree $n\geq 0.$ This is the same as the problem of determining of a minimal set of generators for $\mathscr A$-module $P_s.$ Peterson~\cite{F.P}, Wood~\cite{R.W}, Singer~\cite {W.S1}, and Priddy~\cite{S.P} laid the first foundation for the study of the hit problem and pointed out its relationship to several classical problems in the homotopy theory such as cobordism theory of manifolds, modular representation theory of the general linear group, and stable homotopy type of classifying spaces of finite groups. Later, many researchers have been interested in discovering these hit problems (see the works of Boardman~\cite{J.B}, Bruner-H\`a-H\uhorn ng~\cite{B.H.H}, Crabb and Hubbuck \cite{C.H},  H\uhorn ng~\cite{V.H2}, Kameko~\cite{M.K}, Silverman~\cite{J.S}, Silverman and Singer~\cite{S.S}, Walker and Wood \cite{W.W, W.W1, W.W2}, the present author \cite{P.S}-\cite{D.P9}, Sum \cite{N.S1}-\cite{N.S6} and others), but it was a well known unresolved problem for $s\geq 5.$

For a non-negative integer $d,$ denote by $(QP_s)_d$ the subspace of $QP_s$ consisting of all the classes represented by the homogeneous polynomials in $(P_s)_d.$  One of the extremely useful tools for studying the hit problem is Kameko's homomorphism \cite{M.K}: $\widetilde {Sq^0_*}  = (\widetilde {Sq^0_*})_{(s, 2d+s)}: (QP_s)_{2d+s} \to (QP_s)_d$, which is induced by the $\mathbb F_2$-linear map $\psi: P_s \longrightarrow P_s,$ given by
$$ \psi(x_1^{n_1}x_2^{n_2}\ldots x_s^{n_s}) = \begin{cases}{x_1^{\frac{n_1-1}{2}}x_2^{\frac{n_2-1}{2}}\ldots x_s^{\frac{n_s-1}{2}}}&\text{if } n_1, n_2,\ldots, n_s \mbox{ odd},\\
0 &\text{otherwise},
\end{cases}$$
for any monomial $x_1^{n_1}x_2^{n_2}\ldots x_s^{n_s}\in P_s.$ It is straightforward to see that $\psi$ is not an $\mathscr{A}$-homomorphism, but $\psi Sq^{2i}=Sq^{i} \psi$ and $\psi Sq^{2i+1} = 0$ for any $i\geqslant 0.$ The structure of $QP_s$ was systematically studied for the cases $s\leqslant 4$ (see \cite{J.B, M.K, F.P, N.S4}). Specifically, by the works of Wood \cite{R.W} and Sum \cite{N.S6}, it is sufficient to determine $QP_s$ at each degree $d$ of the form:
\begin{equation} \label{ct1.1}
d = r(2^t-1) + 2^tm,
\end{equation}
whenever $r,\, t,\, m$ are non-negative integers satisfying $1\leqslant r = \mu(d)\leqslant s$ and $\mu(m) < r$ (see Subsection \ref{s2.2}). Here, $\mu(n)$ denotes the smallest number $k$ such that $n$ can be written in the form $\sum_{1\leqslant i \leqslant k}(2^{u_i}-1),$ where $u_i > 0.$ For $r = s-1,$ the problem was investigated by Crabb and Hubbuck~\cite{C.H}, Nam~\cite{T.N}, Repka and Selick~\cite{R.S}, Mothebe \cite{M.M1}, Sum \cite{N.S2, N.S4} and by us \cite{P.S, P.S2}. For $r = s = 5,$  it is known by Sum \cite{N.S8, N.S6} with $m = 5, 10$ and by the present writer \cite{D.P1}-\cite{D.P9} with $m = 6, 8, 18, 22, 42.$ So far there is no a completely answer for the case $s = 5$. It should be noted that solving special cases plays an important role in finding general solutions to the problem. By these reasons, our motivation to write up the present paper is to continue studying the hit problem in the 5-variable case. More specifically, we will explicitly describe $QP_5$ at degrees in \eqref{ct1.1} when $r = s - 2 = 3$ and $m= 1.$  

\newpage
We begin our work with the following, which is one of our main results. 

\begin{thm}\label{dlc-1}
Let $d = 3(2^t-1) + 2^t$ with $t$ a positive integer. The dimension of the $\mathbb F_2$-vector space $(QP_5)_{d}$ is determined by the following table:

\centerline{\begin{tabular}{c|cccccc}
$d = 3(2^t-1) + 2^t$  &$t=1$ & $t=2$ & $t=3$ & $t=4$ & $t  =5$ & $t\geqslant 6$\cr
\hline
\ $\dim(QP_5)_d$ & $46$ & $250$ & $645$ &$945$ &$1115$ & $1116$ \cr
\end{tabular}}
\end{thm}

Noting that if $t =0$ then $d = 3(2^t-1) + 2^t = 1,$ and we can easily claim that $(QP_5)_{1}$ has dimension $5.$ The theorem has been proven by Sum \cite{N.S8, N.S6} for $t =1$ and by Moetele-Mothebe \cite{M.M2} for $t = 2.$ However, in this article, we prove the case $t = 2$ in another way (see Subsection \ref{cmdlc-1}). Now, it should be noticed that Kameko's map $\widetilde {Sq^0_*} = (\widetilde {Sq^0_*})_{(5, d)}: (QP_5)_{d} \longrightarrow (QP_5)_{\frac{d-5}{2}}$ is an epimorphism of $\mathbb F_2$-vector spaces. From this and our previous results \cite{P.S, P.S2}, in order to prove the theorem, we only need to determine the kernel of $(\widetilde {Sq^0_*})_{(5, d)}.$ The calculations are quite long and complicated. Our method is based on the Kameko homomorphism and our recent results in \cite{P.S, P.S2}.

Let $GL_s:= GL(s, \mathbb F_2)$ be the general linear group of rank $s$ over $\mathbb{F}_2.$ As it is known, this $GL_s$ acts regularly on $P_s$ by matrix substitution. Further, the two actions of $\mathscr A$ and $GL_s$ upon $P_s$ commute with each other; hence there is an inherited action of $GL_s$ on $QP_s.$ This shows that the problem we are interested in is one of the useful tools for studying the modular representation of $GL_s$ over $\mathbb F_2$. Moreover, it is applied to describe the $\mathbb F_2$-cohomology of the Steenrod algebra, ${\rm Ext}_{\mathscr A}^{*, *}(\mathbb F_2, \mathbb F_2)$ through Singer's transfer homomorphism \cite{W.S1}. Computing explicitly this cohomology is one of the problems of great importance in stable homotopy theory. But despite intensive investigation for nearly half a century, its structure remains elusive. The Singer transfer \cite{W.S1} is a linear transformation $ \varphi_s:\mbox{Tor}^{\mathscr{A}}_{s, s+*}(\mathbb{F}_2, \mathbb{F}_2)\longrightarrow (QP_s)_*^{GL_s}$, from the $s$-th homology group of $\mathscr A$ to the space of $GL_s$-invariants of $(QP_s)_*$.  Dually, let $PH_*(\mathbb F_2^{s})$ be the subspace of $H_*(\mathbb F_2^{s})$  consisting of all elements that are annihilated by all positive degree Steenrod squares, then there is an induced action of $GL_s$ on $PH_*(\mathbb F_2^{s}),$ and we have an $\mathbb F_2$-linear map from the coinvariant elements of $PH_*(\mathbb F_2^{s})$ to mod-2 cohomology group Ext$_{\mathscr{A}}^{s,s+*}(\mathbb F_2,\mathbb F_2)$ of the Steenrod algebra, which is induced over the $E_2$-term of the Adams spectral sequence by the geometrical transfer map $\Sigma^{\infty}(B(\mathbb F_2^{s})_{+})\longrightarrow \Sigma^{\infty}(S^{0})$ in stable homotopy theory (see also Mitchell \cite{S.M}). These transfers can be played a key role in the study of the Kervaire invariant one problem (see Browder \cite{W.B}, Mahowald \cite{Mahowald}, Minami \cite{N.M, N.M2}). The Kervaire invariant was first introduced by Browder's work \cite{W.B}, where he shows that the classes $h_j^{2}\in {\rm Ext}_{\mathscr{A}}^{2,2^{j+1}}(\mathbb F_2,\mathbb F_2)$ are the permanent cycles in the classical Adams spectral sequence at the prime 2, if and only if smooth framed manifolds of Kervaire invariant one exist only in dimensions $2^{j+1}-2.$ The hypothetical element in stable $2^{j+1}-2$-stem, $\pi^{S}_{2^{j+1}-2},$ represented by such a framed manifold is denoted by $\theta_j.$ These $\theta_j$ are known to be exist for $j\leq 5$ (see also Lin-Mahowald \cite{L.M}). Many open issues in Algebraic and Differential Topology depend on knowing whether or not the Kervaire invariant one elements $\theta_j$ for $j\geq 6.$ In 2016, Hill, Hopkins, and Ravenel \cite{Hill} indicated that these elements do not exist for $j\geq 7$. So far the case $j = 6$ is still no answer. Thus, smooth framed manifolds of Kervaire invariant one therefore exist only in dimensions $2, 6, 14, 30, 62,$ and possibly $126.$

Going back to the Singer transfer,  in \cite{W.S1}, Singer claimed that $\varphi_s$ is an isomorphism for $s =1,2$ and  the "total" dual algebraic transfer $\bigoplus_{s\geq 0}\varphi_s^{*}$ is an algebraic homomorphism. Afterwards, Boardman \cite{J.B},  with additional calculations by Kameko \cite{M.K}, showed that $\varphi_3$ is also an isomorphism. Additionally, in higher rank, Singer states that the transfer homomorphism is not a monomorphism in bidegree $(5,\, 14)$ and made the following conjecture.

\begin{conj}\label{gtSinger}
The homological transfer $\varphi_s$ is an epimorphism for any $s > 0.$
\end{conj}
 
This prediction is known to be true for $s\leqslant 3$. Moreover, it can be verified for $s  = 4$ by using the results in Sum \cite{N.S1, N.S4}. However, when $s \geqslant 5,$ it is an open problem. Recently, some authors have been studied the conjecture for $s = 4, 5$ (see Bruner-H\`a-H\uhorn ng \cite{B.H.H}, H\uhorn ng \cite{V.H2}, Ch\ohorn n-H\`a \cite{Chon-Ha1, Chon-Ha2}, H\`a \cite{Ha}, Nam \cite{T.N2}, the present author \cite{D.P1}, \cite{D.P5}-\cite{D.P12}, Sum \cite{N.S8, N.S5, N.S7, N.S6} and others). In the present work, by using techniques of the hit problem of five variables, we investigate Conjecture \ref{gtSinger} in bidegree $(5,\, d+5),$ where $d$ is determined as in Theorem \ref{dlc-1}. This gives an efficient method to study the Singer transfer and it is different from that of Singer \cite{W.S1}. More explicitly, based on this approach, we obtain the following.

\begin{thm}\label{dlc-2}
For each positive integer $t,$ the space of $GL_5$-invariants $(QP_5)^{GL_5}_{3(2^t-1) + 2^t}$ is trivial. 
 \end{thm}

It is known that Kameko's map $ \widetilde {Sq^0_*} = (\widetilde {Sq^0_*})_{(5, 3(2^t-1) + 2^t)}: (QP_5)_{3(2^t-1) + 2^t} \longrightarrow (QP_5)_{2^{t+1} - 4}$ is an epimorphism of $GL_5$-modules. On the other side, following \cite{N.S5}, $(QP_5)^{GL_5}_{2^{t+1} - 4}$ is trivial,  for all $t > 0$.  So, to prove this theorem, we only need to determine the invariant $({\rm Ker}((\widetilde {Sq^0_*})_{(5, 3(2^t-1) + 2^t)}))^{GL_5}.$ The computations are based on a monomial basis of $(QP_5)_{3(2^t-1) + 2^t},$ which are given in the proof of Theorem \ref{dlc-1}.

By using calculations by Chen \cite{T.C}, Lin \cite{W.L} and Tangora \cite{Tangora} and passing to the dual, deduce that ${\rm Tor}^{\mathscr{A}}_{5, 2^{t+2}+2}(\mathbb{F}_2, \mathbb{F}_2) = 0$ for all $t > 0.$ This together with Theorem \ref{dlc-2} imply that the homological transfer $ \varphi_5: {\rm Tor}^{\mathscr{A}}_{5, 2^{t+2}+2}(\mathbb{F}_2, \mathbb{F}_2)\to (QP_5)^{GL_5}_{3(2^t-1) + 2^t}$  is an isomorphism. As a consequence, we have immediately the following.

\begin{corl}
Conjecture \ref{gtSinger} is true for the rank 5 and the degree $3(2^t-1) + 2^t$ with $t$ an arbitrary positive integer.
\end{corl}

The paper contains five sections and is organized as follows: Background is provided in Section \ref{s2}. 
The proofs of Theorem \ref{dlc-1} and \ref{dlc-2} will be presented in Sections \ref{s3} and \ref{s4}, respectively.
All the admissible monomials of degree $d$ in Theorem \ref{dlc-1} are described in the Appendix.

\medskip\noindent
{\bf Acknowledgment}

This research is funded by Vietnam's National Foundation for Science and Technology Development (NAFOSTED) under grant number 101.04-2017.05.

The author appreciates anonymous referees for their careful corrections to, valuable comments on, and helpful suggestions to the last version of this paper.

The author would also like to express my profound gratitude to Professor Nguyen Sum for pointing out some errors in the original paper and his help in this work.

\section{Background}\label{s2}

In this section, we provide some known basis concepts, which might be useful in the next sections.

\subsection{The weight vector of a monomial}\label{s2.1}

Let us denote by $\alpha_j(d)$ the $j$-th coefficients in dyadic expansion of a non-negative integer $d.$ Then, we have $d = \alpha_0(d)2^0 + \alpha_1(d)2^1 + \cdots + \alpha_j(d)2^j+ \cdots,$ for $\alpha_j(d)\in \{0, 1\}$ with $j\geqslant 0.$ For a monomial $u = x_1^{a_1}x_2^{a_2}\ldots x_s^{a_s}\in P_s,$ we define two sequences associated with $u$ by
\begin{eqnarray}
\omega(u) &=&(\omega_1(u), \omega_2(u),\ldots,\omega_j(u),\ldots),\nonumber\\
\sigma(u) &=&(a_1, a_2, \ldots, a_s),\nonumber
\end{eqnarray}
where $\omega_j(u) :=\sum_{1\leqslant i\leqslant s}\alpha_{j-1}(a_i) \leqslant s,$ for all $j.$ The sequences $\omega(u)$ and $\sigma(u)$ are called the {\it weight vector} and the \textit{exponent vector} of the monomial $u,$ respectively. 

Let $\omega = (\omega_1, \omega_2, \ldots, \omega_i,\ldots)$ be a sequence of non-negative integers. The sequence $\omega$ is called a weight vector if $\omega_j = 0$ for $j\gg 0.$ By convention, $\deg(\omega) = \sum_{j\geq 1}2^{j-1}\omega_j,$ and the sets of all the weight vectors and the exponent vectors are given the left lexicographical order.

\subsection{The admissible monomial}\label{s2.2}

For a weight vector $\omega,$ we define $\deg \omega = \sum_{i\geqslant 1}2^{i-1}\omega_i.$ Let us denote by $P_s(\omega)$ the subspace of $P_s$ spanned by all monomials $x\in P_s$ such that $\deg(x) = \deg(\omega),\ \omega(x)\leqslant \omega,$ and by $P_s^-(\omega)$ the subspace of $P_s$ spanned by all monomials $x\in P_s$ such that $\deg(x) = \deg(\omega),\ \omega(x)< \omega.$ 

\begin{dn}[see \cite{M.K}, \cite{N.S4}]\label{dnKS}
Let $\omega$ be a weight vector and let $f, g$ be two homogeneous polynomials of the same degree in $P_s.$ 
\begin{enumerate}
\item [(i)] $f \equiv g,$ if and only if $(f  + g)\in \mathscr{A}^+\cdot P_s.$ When $f \equiv 0,$ then $f\in \mathscr{A}^+\cdot P_s.$ 
\item[(ii)] $f \equiv_{\omega} g$ if and only if $f, g\in P_s(\omega)$ and $(f  + g)\in ((\mathscr{A}^+\cdot P_s \cap P_s(\omega)) + P_s^-(\omega)).$ Specifically, $f$ is said to be {\it $\omega$-hit}, if $f \equiv_{\omega} 0.$ 
\end{enumerate}
\end{dn}
The Definition \ref{dnKS}(ii) is a modification of a definition in Sum \cite[Definition 2.3]{N.S4}. It can be easily seen that the binary relations  $\equiv$ and $\equiv_{\omega}$ on $P_s$ are equivalence ones.

\begin{dn}[see \cite{M.K}]\label{dnqhtt}
Let $x$ and $y$ be two monomials of the same degree in $P_s.$ We say that $x < y$ if and only if one of the following holds:
\begin{enumerate}
\item[(i)] $\omega(x) < \omega(y);$
\item[(ii)] $\omega(x) = \omega(y)$ and $\sigma(x) < \sigma(y).$
\end{enumerate}
\end{dn}

Our key instrument here is the following notion.

 \begin{dn}[see \cite{M.K}, \cite{N.S4}]
We say that a monomial $x$ in $P_s$ is \textit{inadmissible}, if there exist monomials $y_1, y_2,\ldots, y_m$ such that $y_t < x$ for $1\leqslant t\leqslant m$ and $x \equiv \sum_{1\leqslant t\leqslant m}y_t.$ Then, $x$ is said to be \textit{admissible}, if it is not inadmissible.
\end{dn}

Thus, the above definitions indicated that the set of all the admissible monomials of degree $d$ in $P_s$ is a minimal set of $\mathscr{A}$-generators for $P_s$ in degree $d.$ Therefore, $(QP_s)_d$ is an $\mathbb F_2$-vector space with a basis consisting of all the classes represent by the admissible monomials of degree $d$ in $P_s.$

\begin{dn}[see \cite{M.K}, \cite{N.S4}]\label{dniad}
A monomial $x$ in $P_s$ is said to be strictly inadmissible if and only if there exists monomials $y_1, y_2,\ldots, y_m$ such that $y_t < x$ for $1\leqslant t\leqslant m$ and 
$$ x = \sum\limits_{1\leqslant t\leqslant m}y_t + \sum\limits_{1\leqslant j\leqslant 2^r - 1}Sq^{j}(h_j),$$
where $r = {\rm max}\{i\in\mathbb Z: \omega_i(x) > 0\}$ and  suitable polynomials $h_j\in P_s.$
\end{dn}

Noticing that if $x$ is strictly inadmissible, then it is inadmissible. However, in general, the converse is not true. For instance, the monomial $x = x_1x_2^2x_3^2x_4^2x_5^2x_6$ in $P_6$ is inadmissible, but it is not strictly inadmissible.

\begin{theo}[see \cite{M.K}, \cite{N.S4}]\label{dlKS}
Let $x, y$ and $u$ be monomials in $P_s$ such that $\omega_i(x) = 0$ for $i > r >0, \omega_t(u)\neq 0$ and $\omega_i(u)=0$ for $i  > t > 0.$ 
\begin{enumerate}
\item[(i)] If $u$ is inadmissible, then $xu^{2^r}$ is also inadmissible.
\item[(ii)] If $u$ is strictly inadmissible, then $uy^{2^t}$ is also strictly inadmissible.
\end{enumerate}
\end{theo}

Form now on, we use the following notations.

\begin{eqnarray}
P_s^0 &=& \langle\{x= x_1^{a_1}x_2^{a_2}\ldots x_s^{a_s}\in P_s\;|\;a_1a_2\ldots a_s = 0\}\rangle, \nonumber\\
P_s^+&=&\langle\{x= x_1^{a_1}x_2^{a_2}\ldots x_s^{a_s}\in P_s\;|\; a_1a_2\ldots a_s > 0\}\rangle.\nonumber
\end{eqnarray}
Then $P_s^0$ and $P_s^+$ are the $\mathscr{A}$-submodules of $P_s.$ Furthermore, we have a direct summand decomposition of the $\mathbb{F}_2$-vector spaces: $ QP_s = QP_s^0\,\bigoplus\, QP_s^+,$ where $QP_s^0:= \mathbb{F}_2\otimes_{\mathscr A} P_s^0$ and $QP_s^+:= \mathbb{F}_2\otimes_{\mathscr A} P_s^+.$

Let $QP_s(\omega)$ be the quotient of $P_s(\omega)$ by the equivalence relation $\equiv_\omega.$ According to Walker and Wood \cite{W.W2}, we have a filtration of $QP_s$:
$$ \{0\}\subseteq \cdots \subseteq P_s^{-}(\omega)/((\mathscr A^{+}\cdot P_s) \cap P_s^{-}(\omega)) \subseteq P_s(\omega)/((\mathscr A^{+}\cdot P_s) \cap P_s(\omega)) \subseteq \cdots \subseteq P_s/(\mathscr A^{+}\cdot P_s) = QP_s.$$
It should be noticed that this is not only a filtration of $QP_s$ as a vector space, but also as a $GL_s$-module. The inclusion of $P_s^{-}(\omega)$ into $P_s(\omega)$ induces the monomorphism $$P_s^{-}(\omega)/((\mathscr A^{+}\cdot P_s) \cap P_s^{-}(\omega)) \to  P_s(\omega)/((\mathscr A^{+}\cdot P_s) \cap P_s(\omega))$$ and the sequence is short exact:
$$ 0\to P_s^{-}(\omega)/((\mathscr A^{+}\cdot P_s) \cap P_s^{-}(\omega))  \to P_s(\omega)/((\mathscr A^{+}\cdot P_s) \cap P_s(\omega)) \to P_s(\omega)/(((\mathscr{A}^+\cdot P_s)\cap P_s(\omega)) + P_s^-(\omega))\to 0.$$
Combining this and Definition \ref{dnKS}(ii), one gets $$QP_s(\omega) = P_s(\omega)/((\mathscr{A}^+\cdot P_s\cap P_s(\omega)) + P_s^-(\omega)) \cong (P_s(\omega)/((\mathscr A^{+}\cdot P_s) \cap P_s(\omega)))/(P_s^{-}(\omega)/((\mathscr A^{+}\cdot P_s) \cap P_s^{-}(\omega))).$$
Following \cite{N.S6}, $QP_s(\omega)$ is also an $GL_s$-module, further this isomorphism is also an isomorphism of $GL_s$-modules. Combining this and the above filtration of $QP_s$, we have immediately $$\dim((QP_s)_d) = \sum_{\deg(\omega) = d}\dim(QP_s(\omega)),\ \ \dim((QP_s)^{GL_s}_d)\leq \sum_{\deg(\omega) = d}\dim(QP_s(\omega)^{GL_s}).$$

Throughout this paper, we will always denote the set of all admissible monomials of degree $d$ in  $P_s$ by $\mathscr{B}_s(d).$  We put
$$ \mathscr{B}_s^0(d) := \mathscr{B}_s(d)\cap (P_s^0)_d,\,\,\mathscr{B}_s^+(d) := \mathscr{B}_s(d)\cap (P_s^+)_d.$$
Then, if $\omega$ is a weight vector of degree $d,$ then we put
$$ \mathscr{B}_s(\omega) := \mathscr{B}_s(d)\cap P_s(\omega),\, \mathscr{B}_s^0(\omega) := \mathscr{B}_s(\omega)\cap (P^0_s)_d,\, \mathscr{B}_s^+(\omega) := \mathscr{B}_s(\omega)\cap (P^+_s)_d.$$ 
Next, for a polynomial $f\in P_s,$ we denote by $[f]$ the classes in $QP_s$ represented by $f.$ If $\omega$ is a weight vector and $f\in P_s(\omega),$ then denote by $[f]_\omega$ the classes in $QP_s(\omega)$ represented by $f.$ For a subset $\mathscr{B}\subset P_{s},$ we denote $[\mathscr B] = \{[f]\, :\, f\in \mathscr B\}.$ If $\mathscr B\subset P_s(\omega),$ then we set $[\mathscr B]_{\omega} = \{[f]_{\omega}\, :\, f\in \mathscr B\}.$ Thus, we see that $[\mathscr{B}_s(\omega)]_\omega,\, [\mathscr{B}^0_s(\omega)]_\omega$ and $[\mathscr{B}_s^+(\omega)]_\omega$ are respectively the bases of the $\mathbb{F}_2$-vector spaces $QP_s(\omega),\ QP_s^0(\omega):= QP_s(\omega)\cap (QP^0_s)_d$ and $QP_s^+(\omega) := QP_s(\omega)\cap (QP^+_s)_d.$

For a natural number $d,$ let $ \mu(d) = \mbox{min}\big\{r\in \mathbb N:\ d = \sum_{1\leqslant i\leqslant r}(2^{u_i}-1),\ u_i > 0\big\}.$ So, by Wood \cite{R.W}, $(QP_s)_d = 0$ if $\mu(d) > s,$ and by Sum \cite{N.S6}, $\mu(d) = r\leqslant s$ if and only if there exists uniquely a sequence of integers $u_1 > u_2 > \cdots > u_{r-1} \geqslant u_r > 0$ such that $d = \sum_{1\leqslant i\leqslant r}(2^{u_i} - 1).$ From these, we only need to study $(QP_s)_d$ in degrees $d$ satisfying $1\leqslant \mu(d) = r\leqslant s$; moreover $d$ is of the form \mbox{$r(2^t-1) + 2^t.m,$} whenever $t,\, m$ are non-negative integers such that $\mu(m) < r.$ Indeed, if $\mu(d) = r = 1,$ then $d$ has the form $d = 2^{u} - 1 = 1.(2^{u} - 1) + 2^{u}. 0 = r.(2^{u} - 1) + 2^{u}.m,$ where $u > 0$ and $\mu(m = 0) = 0 < 1.$ If $\mu(d) = r \geqslant 2,$ then by Sum \cite{N.S6}, $d$ has the form $\sum_{1\leqslant i\leqslant r}(2^{u_i} - 1),$ where $u_1 > u_2 > \cdots > u_{r-1} \geqslant u_r > 0.$ Consider $m = \sum_{1\leqslant j\leqslant r-1}(2^{u_j-u_r}-1).$ According to Sum \cite{N.S6}, we see that $\mu(m) = r-2 < r$ if $u_{r-1} = u_r$ and $\mu(m) = r-1 < r$ if $u_{r-1} > u_r.$ Further, $d = r(2^{t} - 1) + 2^{t}.m$ with $t = u_r.$  

The following homomorphisms will be used throughout the article. For $1\leqslant i\leqslant s,$ define the homomorphism $\rho_i: P_{s-1}\rightarrow P_s$ of algebras by substituting 
$$ \rho_i(x_j) = \left\{ \begin{array}{ll}
{x_j}&\text{if }\;1\leqslant j < i, \\
x_{j+1}& \text{if}\; i\leqslant j < s.
\end{array} \right.$$
The maps $\rho_i$ are also the monomorphism of $\mathscr A$-modules. We consider the set
$$ \mathcal{N}_s := \{(i; I)\;|\; I = (i_1,i_2,\ldots,i_r), 1\leqslant i < i_1<i_2 < \ldots < i_r\leqslant s, 0\leqslant r < s\},$$  
where by convention, $I = \emptyset,$ if $r = 0.$ Let $r = \ell(I)$ be the length of $I.$ 

\begin{dn}
Let $(i; I)\in\mathcal N_s, r = \ell(I),$ and let $u$ be an integer with $1\leqslant u\leqslant r.$ A monomial $x = x_1^{a_1}x_2^{a_2}\ldots x_{s-1}^{a_{s-1}}\in P_{s-1}$ is said to be $u$-compatible with $(i; I)$ if all of the following hold:
\begin{enumerate}
\item [(i)] $a_{i_1 - 1} =a_{i_2 - 1} = \ldots = a_{i_{(u-1)} - 1} = 2^{r} - 1,$
\item[(ii)] $a_{i_{u} - 1} > 2^{r} - 1,$
\item [(iii)] $\alpha_{r-t}(a_{i_{u} - 1}) = 1,\;\forall t,\ 1\leqslant t\leqslant u,$
\item[(iv)] $\alpha_{r-t}(a_{i_{t}-1}) = 1,\;\forall t,\ u < t \leqslant r.$
\end{enumerate}
\end{dn}
Clearly, a monomial $x$ can be $u$-compatible with a given $(i; I)\in \mathcal N_s$ for at most one value of $u.$ By convention, $x$ is $1$-compatible with $(i; \emptyset).$

\begin{dn}
Consider $(i; I)\in\mathcal{N}_s, 0 < r < s$ and put $ x_{(I, u)} = x_{i_u}^{2^{r-1} + 2^{r-2} +\, \cdots\, + 2^{r-u}}\prod\limits_{u < t\leqslant r}x_{i_t}^{2^{r-t}},$ for $1\leqslant u\leqslant r,\; x_{(\emptyset, 1)} = 1.$ For a monomial $x\in P_{s-1},$ we define the monomial $\phi_{(i; I)}(x)$ in $P_s$ by setting
 $$ \phi_{(i; I)}(x) = \left\{ \begin{array}{ll}
\dfrac{x_i^{2^{r} - 1}\rho_i(x)}{x_{(I, u)}}&\text{if there exist $u$ such that $x$ is $u$-compatible with $(i; I)$}, \\
0&\text{otherwise}.
\end{array} \right.$$
Then we have a linear transformation $\phi_{(i; I)}: P_{s-1}\rightarrow P_s.$ It is easily seen that  $\phi_{(i; \emptyset)} = \rho_i$ and that $\phi_{(i; I)}$ is not an  $\mathscr A$-homomorphism. Moreover, if  $\phi_{(i; I)}(x)\neq  0,$ then $\omega( \phi_{(i; I)}(x)) = \omega(x).$
\end{dn}

For a subset $U\subset P_{s-1},$ we denote by
$$ \begin{array}{ll}
\Phi^0(U) &= \bigcup\limits_{1\leqslant i \leqslant s}\phi_{(i; \emptyset)}(U) = \bigcup\limits_{1\leqslant i \leqslant s}\rho_i(U),\\
\Phi^+(U) &= \bigcup\limits_{(i; I)\in\mathcal{N}_s,\;0 < \ell(I) < s}\phi_{(i; I)}(U)\setminus P_s^0.
\end{array}$$
Then, it is easy to see that if $U$ is a minimal set of generators for $\mathscr A$-module $P_{s-1}$ in degree $d$, then $\Phi^0(U) $ is also a minimal set of generators for $\mathscr A$-module $P_s^0$ in degree $d.$

\begin{dn}\label{dndc}
For any $(i; I)\in\mathcal{N}_s,$ we define the homomorphism $p_{(i; I)}: P_s\rightarrow P_{s-1}$ of $\mathbb F_2$-algebras by substituting
$$p_{(i; I)}(x_j) = \left\{ \begin{array}{ll}
{x_j}&\text{if }\;1\leqslant j < i, \\
\sum\limits_{k\in I}x_{k-1}& \text{if}\; j = i,\\
x_{j-1}&\text{if}\; i < j \leqslant s.
\end{array} \right.$$
It should be noticed that $p_{(i; I)}$ is also a homomorphism of $\mathscr{A}$-modules. In particular, we have $p_{(i; \emptyset)}(x_i) = 0$ for $1\leqslant i\leqslant s$ and $p_{(i; I)}(\rho_i(u)) = u$ for any $u\in P_{s-1}.$ Moreover, the next technicality is important.
\end{dn}

\begin{bd}[see \cite{P.S}]\label{bdPS}
If $x$ is a monomial in $P_s,$ then $p_{(i; I)}(x)\in P_{s-1}(\omega(x)).$
\end{bd}
The lemma indicates that if $\omega$ is a weight vector and $x\in P_s(\omega),$ then $p_{(i; I)}(x)\in P_{s-1}(\omega).$ Furthermore, $p_{(i; I)}$ passes to a homomorphism from $QP_s(\omega)$ to $QP_{s-1}(\omega).$

In Section \ref{s3}, we use Lemma \ref{bdPS} and the results in Sum \cite{N.S4} to prove a certain subset of  $QP_s$ is linearly independent. More precisely,  let $\mathscr B$ be a finite subset of $P_s$ consisting of some monomials of degree $d.$ Denote by $|\mathscr B|$ the cardinal of $\mathscr B.$ To prove the set $[\mathscr B]$ is linearly independent in $(QP_s)_d,$ we denote the elements of $\mathscr B$ by $b_{d, \, k},\ 1\leqslant k \leqslant n = |\mathscr B|$ and assume that there is a linear relation \mbox{$ \mathcal S = \sum_{k = 1}^n\gamma_kb_{d,\, k}\equiv 0,$} with $\gamma_k\in \mathbb F_2$ for all $k,\, 1\leqslant k\leqslant n.$ For $(i; I)\in \mathcal N_s,$ we explicitly compute $p_{(i; I)}(\mathcal S)$ in terms of the admissible monomials in $P_{s-1}\, ({\rm mod}(\mathscr A^+\cdot P_{s-1})).$ Computing from some relations $p_{(i; I)}(\mathcal S)\equiv 0$ with $(i; I)\in \mathcal N_s,$ we obtain $\gamma_k = 0$ for all $k.$

\section{\boldmath\mbox{$\mathscr A$}-generators for \boldmath\mbox{$P_5$} in degree \boldmath\mbox{$3(2^t-1) + 2^t$}}\label{s3}

The goal of this section is to prove Theorem \ref{dlc-1}.  More precisely, we explicitly determine all admissible monomials of degree $3(2^t-1) + 2^t$ in $P_5.$ We first recall a result of Singer \cite{W.S2} on the hit monomials in $P_s.$ 

\subsection{Singer's criterion on the hit monomials}

\begin{dn}\label{spi}
A monomial $z = x_1^{b_1}x_2^{b_2}\ldots x_s^{b_s}$ in $P_s$ is called a spike if $b_i = 2^{t_i} - 1$ for $t_i$ a non-negative integer and $i = 1, 2, \ldots, s.$ If $z$ is a spike with $t_1 > t_2 > \ldots > t_{r-1}\geqslant t_r > 0$ and $t_j = 0$ for $j  > r,$ then it is called a minimal spike.
\end{dn}

\begin{bd} All the spikes in $P_s$ are admissible and their weight vectors are weakly decreasing. Furthermore, if a weight vector $\omega = (\omega_1, \omega_2, \ldots)$ is weakly decreasing and $\omega_1\leqslant s,$ then there is a spike $\mathscr Z $ in $P_s$ such that $\omega(\mathscr Z ) = \omega.$
\end{bd}

\begin{proof}
Let $\mathscr Z = \prod\limits_{1\leq i\leq s}x_i^{2^{d_i}-1}$ be a spike monomial in $P_s.$ By way of contradiction, suppose that $\mathscr Z$ is inadmissible. Then, there exist monomials $y_1, y_2, \ldots, y_m\in P_s$ such that $y_t < \mathscr Z$ for all $t$ and 
$ \mathscr Z = \sum_{1\leq t\leq m}y_t + \sum\limits_{j > 0}Sq^j(w_j),$
with suitable monomial $w_j\in P_s.$ This implies that there is a monomial $u\in P_s$ with $\deg u < \deg \mathscr Z$ such that $Sq^n(u) = \mathscr Z + f,$ where $n > 0$ and $f = \sum_{i}X_i\in P_s,\ \deg X_i = \deg \mathscr Z.$ Suppose $u = \prod\limits_{1\leqslant \ell\leqslant k}x_{\ell}^{a_{\ell}}.$ Since $\deg u <\deg \mathscr Z,$ ther is a number $t,\, 1\leq t\leq s$ such that $a_t < 2^{d_t} - 1.$ According to Cartan's formula, we have
$$ \mathscr Z + f = Sq^n(u) = \sum\limits_{i\geq 0}Sq^i(x_t^{a_t})Sq^{n-i}(x_1^{a_1}\ldots x_{t-1}^{a_{t-1}}x_{t+1}^{a_{t+1}}\ldots x_s^{a_s}).$$
Since $2^{d_t} - 1-a_t >0$ and $\mathscr Z$ is a term of $\sum\limits_{i\geqslant 0}Sq^i(x_t^{a_t})Sq^{n-i}(x_1^{a_1}\ldots x_{t-1}^{a_{t-1}}x_{t+1}^{a_{t+1}}\ldots x_s^{a_s}),$ hence $\mathscr Z$ is a term of $Sq^{2^{d_t} - 1-a_t }(x_t^{a_t})Sq^{n-(2^{d_t} - 1-a_t )}(x_1^{a_1}\ldots x_{t-1}^{a_{t-1}}x_{t+1}^{a_{t+1}}\ldots x_s^{a_s}).$ On the other hand,  $2^{d_t} - 1-a_t \leq 2^{d_t} - 1,$ so there exist a number $\ell\geq 0$ such that $\alpha_{\ell}(2^{d_t} - 1-a_t) = 1.$ This implies $\alpha_{\ell}(a_t) = 0.$ Since $\binom{a_t}{2^{d_t} - 1-a_t}$ is even, $Sq^{2^{d_t} - 1-a_t}(x_t^{a_t}) = 0.$ So, one gets $\mathscr Z = 0.$ This contradicts the fact that $\mathscr Z = \prod\limits_{1\leq i\leq s}x_i^{2^{d_i}-1}\neq 0.$  Obviously, if $\mathscr Z $ is a spike, then its weight vector is weakly decreasing. 

Now, assume that $\omega = (\omega_1, \omega_2, \ldots)$ is weakly decreasing and $\omega_1\leq s.$ For $ j\geq s$, we denote by $a_j$ the non-negative integer such that
$$ \alpha_i(a_j) = \begin{cases}1, &\text{if } \omega_{i+1}\geqslant j,\\ 0, &\text{if } \omega_{i+1}< j. \end{cases}$$
Denote by $d_j$ the greatest integer such that $\omega_{d_j}\geqslant j$. Since $\omega$ is weakly decreasing, $\alpha_i(a_j) = 1$ for $1 \leq i < d_j$ and $\alpha_i(a_j) = 0$ with $i \geq d_j$. So,  $a_j = 2^{d_j}-1$. 
On the other hand, we have $\omega_{i+1} \leq  \omega_1 \leq s$ for all $i \geq 0.$ Hence, if $j > s$, then $a_j = 0$. Therefore, $\mathscr Z = x_1^{a_1}x_2^{a_2}\ldots x_s^{a_s}$ is a spike in $P_s$ and $\omega(\mathscr Z) = \omega$. The lemma is proved.
\end{proof}

In \cite{W.S2}, Singer pointed out that if $\mu(d)\leqslant s,$ then there exists uniquely a minimal spike of degree $d$ in $P_s.$ Further,  we have the following. 
\begin{theo}[see Singer~\cite{W.S2}]\label{dlsig}
Let $x$ be a monomial in $P_s,$ where $\mu(\deg(x))\leqslant s.$ Suppose that $z$ is the minimal spike with $\deg(z) = \deg(x).$ Then,  if $\omega(x) < \omega(z),$ then $x$ belongs to $\mathscr A^{+}\cdot P_s.$
\end{theo}

\subsection{Proof of Theorem \ref{dlc-1}}\label{cmdlc-1}

For convenience, we use the following notations:
\begin{eqnarray}
\mathbb{N}_s &=& \{1,2,\ldots, s\},\nonumber\\
X_{\mathbb I} &=& X_{\{i_1,i_2,\ldots,i_r\}} = \prod\limits_{i\in\mathbb{N}_s\setminus \mathbb I}x_i,\,\, \mathbb I = \{i_1,i_2,\ldots,i_r\}\subseteq  \mathbb{N}_s .\nonumber
\end{eqnarray}
Specifically, $X_{\mathbb{N}_s} = 1,\  X_{\emptyset} = x_1x_2\ldots x_s,$ and $X_{\{i\}} = x_1\ldots \hat{x}_i\ldots x_s,\,\,i = 1,2,\ldots, s.$ Let $x = x_1^{a_1}x_2^{a_2}\ldots x_s^{a_s}\in P_s$ and let $ \mathbb I_j(x) = \{i\in \mathbb{N}_{s}:\; \alpha_j(a_i) = 0\},$ for $j\geqslant 0.$ Then, a careful but straightforward computation shows that $x = \prod_{j\geq 0}X_{\mathbb I_{j}(x)}^{2^{j}}.$ Now, notice that Kameko's homomorphism $\widetilde {Sq^0_*} = (\widetilde {Sq^0_*})_{(5, 3(2^t-1) + 2^t)}: (QP_5)_{3(2^t-1) + 2^t} \longrightarrow (QP_5)_{2^{t+1} - 4}$ is an epimorphism, and so, $ (QP_5)_{3(2^t-1) + 2^t} \cong {\rm Ker}((\widetilde {Sq^0_*})_{(5, 3(2^t-1) + 2^t)})\,\bigoplus\,  (QP_5)_{2^{t+1} - 4}.$ Following our previous works \cite{P.S, P.S2}, the dimensions of $(QP_5)_{2^{t+1} - 4}$ are determined as follows. 

\begin{therm}\label{dldb}
Let $n = 2^{t+1}-4$ with $t$ a positive integer. The dimension of the $\mathbb F_2$-vector space $(QP_5)_{n}$ is determined by the following table:

\centerline{\begin{tabular}{c|cccccc}
$n = 2^{t+1}-4$&$t=1$ & $t=2$ & $t=3$ & $t=4$ & $t  =5$ & $t\geqslant 6$\cr
\hline
 $\dim(QP_5)_n$ & $1$ & $45$ & $190$ &$480$ &$650$ & $651$ \cr
\end{tabular}}
\end{therm}
When $t = 1$, we see that $2^{t+1}-4  =0 $ and $(QP_5)_{0}\cong \mathbb F_2.$ So, $\dim((QP_5)_{0}) = \dim(\mathbb F_2) = 1.$ We now determine ${\rm Ker}((\widetilde {Sq^0_*})_{(5, 3(2^t-1) + 2^t)}).$ Put
$$ \omega_{(5, t)} =  (3,3, \ldots, 3, 1),\, \mbox{and }\, \omega^*_{(5, t)} =  (3,3, \ldots, 3)\ \mbox{($t$ times of $3$)}.$$
Then, an important key in the proof of the theorem is the following lemma.

\begin{lema}\label{bdvtt}
If $x$ is an admissible monomial of degree $3(2^t-1) + 2^t$ in $P_5$ and $[x]$ belongs to  ${\rm Ker}((\widetilde {Sq^0_*})_{(5, 3(2^t-1) + 2^t)})$ then $\omega(x) =\omega_{(5, t)}.$
\end{lema}

\begin{proof}
We prove the lemma by induction on $t.$ By Sum \cite{N.S8, N.S6}, we see that the lemma holds for $t = 1.$ 

Suppose $t > 1$ and the lemma holds for $1,2,\ldots, t-1$.  Observe that the monomial $z = x_1^{2^{t+1}-1}x_2^{2^t-1}x_3^{2^t-1}$ is the minimal spike of degree $3(2^t-1) + 2^t$ in $P_5$ and  $\omega(z) =  \omega_{(5,t)}.$

 Since $3(2^t-1) + 2^t$ is odd and $[x]\neq [0],$ one gets either $\omega_1(x) = 3$ or $\omega_1(x) = 5.$ If $\omega_1(x) = 5,$ then $x = X_{\emptyset}y^2$ with $y$ a monomial of degree $2^{t+1} - 4$ in $P_5.$ Since $x$ is admissible, by Theorem \ref{dlKS}, $y$ is admissible. Hence, $(\widetilde {Sq}_*^0)_{(5, 3(2^t-1) + 2^t)}([x]) = [y]\neq [0].$ This contradicts the face that $[x]\in {\rm Ker}((\widetilde {Sq_*^0})_{(5, 3(2^t-1) + 2^t)}),$ so $\omega_1(x) = 3.$ Then, we have $x = X_{\{i, j\}}u^2$ with $1 \leqslant i < j\leqslant 5$ and $u$ an admissible monomial of degree $3(2^{t-1}-1) + 2^{t-1}$ in $P_5.$ Now, the lemma follows from the inductive hypothesis.
\end{proof}

Following Lemma \ref{bdvtt}, it may be concluded that
$$ {\rm Ker}((\widetilde {Sq^0_*})_{(5, 3(2^t-1) + 2^t)})  = (QP_5^0)_{3(2^t-1) + 2^t}\, \bigoplus\, \mathbb V,$$
where $\mathbb V = ({\rm Ker}((\widetilde {Sq^0_*})_{(5,3(2^t-1) + 2^t)})\cap (QP_5^+)_{3(2^t-1) + 2^t})$ is isomorphic to $QP_5^+(\omega_{(5, t)}).$
\medskip

{\bf The case \boldmath\mbox{$t = 1$}.} Following \cite{N.S8, N.S6}, we have $|\mathscr B_5^0(5)| = |\mathscr B_5^0(\omega_{(5, 1)})| = 45$ and $|\mathscr B_5^+(\omega_{(5, 1)})| = 0.$ This together with Theorem \ref{dldb} imply that $(QP_5)_5$ is $46$-dimensional.
\medskip

{\bf The case \boldmath\mbox{$t = 2$}.} It should be noted that for each $s\geqslant 1,$ Moetele and Mothebe indicated in \cite{M.M2} that $\dim (QP_s)_{13} = \sum\limits_{3\leqslant j\leqslant 13}\binom{d}{j}C_j,$ where $C_j,\ 3\leqslant j\leqslant 13,$ are determined by the following table:

\centerline{\begin{tabular}{c|ccccccccccccc}
$j$  &$3$ & $4$ & $5$ & $6$ & $7$ &$8$ & $9$ & $10$ & $11$ & $12$ & $13$   \cr
\hline
\ $C_j$ &$3$ & $23$ & $105$ & $268$ & $415$ &$438$ & $322$ & $164$ & $55$ & $11$ & $1$  \cr
\end{tabular}}

\medskip

This implies that $(QP_5)_{13}$ has dimension $250.$ Their methods are based on Kameko's homomorphism \cite{M.K} and a result on the admissible monomials {\cite[Theorem 2.7]{M.M2}}. For this case, we will prove it in another way. More precisely, to compute $(QP_5)_{13},$ we determine explicitly all the admissible monomials of degree $13$ in $P_5$ by using Lemma \ref{bdPS} and some homomorphisms in Subsection \ref{s2.2}. Detailed calculations are performed as follows. By Sum \cite{N.S4}, we have $|\mathscr B_4(13)| = |\mathscr B_4(\omega_{(5, 2)})| = 35.$ A simple computation shows that $ \mathscr B_5^0(13) = \mathscr B_5^0(\omega_{(5, 2)}) = \Phi^0(\mathscr B_4(13))$ is the set consisting of $145$ admissible monomials, namely $q_{2, k},$ for $1\leqslant k\leqslant 145$ (see Subsection \ref{dtcndtq-1}), i.e., $(QP_5^0)_{13}$ is $145$-dimensional. Next, we need to calculate $QP_5^+(\omega_{(5, 2)}).$ It is not difficult to show that $\Phi^+(\mathscr B_4(\omega_{(5, 2)})) \cup \{x_1x_2^2x_3^4x_4^3x_5^3,\ \ x_1^3x_2^4x_3^3x_4x_5^2, \ \ x_1^3x_2^4x_3x_4^3x_5^2, \ \ x_1^3x_2^4x_3x_4^2x_5^3\}$ is the set consisting of 60 admissible monomials, namely $b_{13,\, k},$ for $1\leqslant k\leqslant 60$ (see Subsection \ref{dtcndtq-1}). The following assertion is of key importance in the proof.

\begin{propo}\label{md13}
The set $\{[b_{13,\, k}]_{\omega_{(5,2)}},\ 1\leqslant k\leqslant 60\}$ is a basis of $\mathbb F_2$-vector space $QP_5^+(\omega_{(5, 2)}).$
\end{propo}
 
We need some lemmas for the proof of the proposition.  It is rather straightforward to prove the following directly, so we omit the details by leaving
them to the interested reader for a direct check.

\begin{lema}\label{bd13-2}
The following monomials are strictly inadmissible:
\begin{enumerate}
\item[(i)]  $x_i^2x_j^2x_kx_{\ell}x_m^3,\ x_i^2x_jx_k^2x_{\ell}x_m^3,\ x_i^2x_jx_kx_{\ell}^2x_m^3,\ i < j <k < \ell.$ Here $(i, j, k, \ell, m)$ is a permutation of $(1,\, 2,\, 3,\, 4,\, 5).$  

\item[(ii)] $\rho_i(u),\ 1\leqslant i\leqslant 5,$  where $u$ is one of the following monomials:
$$ x_1^2x_2x_3^3x_4^3,\  x_1^2x_2^3x_3x_4^3,\ x_1^2x_2^3x_3^3x_4, \ x_1^3x_2^2x_3x_4^3,\ x_1^3x_2^2x_3^3x_4,\ x_1^3x_2^3x_3^2x_4,\ x_1^3x_2^4x_3^3x_4^3.$$
\end{enumerate}
\end{lema}

The next lemma is important.

\begin{lema}\label{bd13-3}
The following monomials are strictly inadmissible:
\medskip

    \begin{tabular}{lrrrr}
    $ x_1x_2^{2}x_3^{2}x_4x_5^{7}$, & \multicolumn{1}{l}{$ x_1x_2^{2}x_3^{2}x_4^{3}x_5^{5}$,} & \multicolumn{1}{l}{$ x_1x_2^{2}x_3^{2}x_4^{5}x_5^{3}$,} & \multicolumn{1}{l}{$ x_1x_2^{2}x_3^{2}x_4^{7}x_5$,} & \multicolumn{1}{l}{$ x_1x_2^{2}x_3^{3}x_4^{2}x_5^{5}$,} \\
    $ x_1x_2^{2}x_3^{3}x_4^{6}x_5$, & \multicolumn{1}{l}{$ x_1x_2^{2}x_3^{6}x_4x_5^{3}$,} & \multicolumn{1}{l}{$ x_1x_2^{2}x_3^{6}x_4^{3}x_5$,} & \multicolumn{1}{l}{$ x_1x_2^{2}x_3^{7}x_4^{2}x_5$,} & \multicolumn{1}{l}{$ x_1x_2^{3}x_3^{2}x_4^{2}x_5^{5}$,} \\
    $ x_1x_2^{3}x_3^{2}x_4^{6}x_5$, & \multicolumn{1}{l}{$ x_1x_2^{3}x_3^{6}x_4^{2}x_5$,} & \multicolumn{1}{l}{$ x_1x_2^{6}x_3^{2}x_4x_5^{3}$,} & \multicolumn{1}{l}{$ x_1x_2^{6}x_3^{2}x_4^{3}x_5$,} & \multicolumn{1}{l}{$ x_1x_2^{6}x_3^{3}x_4^{2}x_5$,} \\
    $ x_1x_2^{7}x_3^{2}x_4^{2}x_5$, & \multicolumn{1}{l}{$ x_1^{3}x_2x_3^{2}x_4^{2}x_5^{5}$,} & \multicolumn{1}{l}{$ x_1^{3}x_2x_3^{2}x_4^{6}x_5$,} & \multicolumn{1}{l}{$ x_1^{3}x_2x_3^{6}x_4^{2}x_5$,} & \multicolumn{1}{l}{$ x_1^{3}x_2^{5}x_3^{2}x_4^{2}x_5$,} \\
    $ x_1^{7}x_2x_3^{2}x_4^{2}x_5.$ &       &       &       &  \\
    \end{tabular}%
 \end{lema}

\begin{proof}
We prove the lemma for $x = x_1x_2^{2}x_3^{2}x_4^3x_5^{5} .$  The others can be proven by a similar computation. By a direct computation using the Cartan formula, we have
$$ \begin{array}{ll}
\medskip
x &= x_1x_2x_3^{2}x_4^3x_5^{6} + x_1x_2^{2}x_3x_4^3x_5^{6} + Sq^1(x_1x_2x_3^2x_4^3x_5^{5} + x_1x_2x_3^2x_4^5x_5^{3} + x_1x_2x_3^4x_4^3x_5^{3} + x_1x_2^{2}x_3x_4^3x_5^{5}\\
\medskip
&\quad + x_1x_2^{2}x_3x_4^5x_5^{3} + x_1x_2^{4}x_3x_4^3x_5^{3}.+ x_1^2x_2x_3x_4^5x_5^{3} + x_1^4x_2x_3x_4^3x_5^{3}) \\
\medskip
&\quad + Sq^2(x_1x_2x_3x_4^3x_5^{5} + x_1x_2^{2}x_3^2x_4^3x_5^{3} + x_1^2x_2x_3^2x_4^3x_5^{3} + x_1^2x_2^{2}x_3x_4^3x_5^{3})\\
&\quad  + Sq^4(x_1x_2x_3x_4^3x_5^{3}) \ \rm mod(P_5^-(\omega_{(5, 2)})).
 \end{array}$$
This equality implies that $x$ is strictly inadmissible. 
\end{proof}

\begin{proof}[{\it Proof of Proposition \ref{md13}}]

Let $x$ be an admissible monomial of degree $13$ in $P_5$ and $\omega(x) = \omega_{(5, 2)}.$ Then, $x = X_{\{i, j\}}u^2$ with $1\leqslant i<j\leqslant 5$ and $u$ a monomial of degree $5$ in $P_5.$ Since $x$ is admissible, according to Theorem \ref{dlKS}, we get $u\in \mathscr{B}_5(\omega_{(5, 1)}).$ 

\medskip

By a direct computation, we see that for all $y\in \mathscr{B}_5(\omega_{(5, 1)}),$ such that $X_{\{i, j\}}y^2\neq b_{13,\, k},\ \forall k,\ 1\leqslant k\leqslant 60,$ there  is a monomial $w$ which is given in one of Lemmas \ref{bd13-2}, and \ref{bd13-3} such that $X_{\{i, j\}}y^2 = wz^{2^r}$ with a monomial $z\in P_5$ and $r = {\rm max}\{m\in \mathbb Z\ :\ \omega_m(w)  > 0\}.$ By Theorem \ref{dlKS},  $X_{\{i, j\}}y^2$ is inadmissible. Since $x = X_{\{i, j\}}u^2$ is admissible, one gets $x = b_{13, \, k}$ for suitable $k.$

\medskip

We now prove the set $\{[b_{13,\, k}]_{\omega_{(5,2)}}\ :\ 1\leqslant k\leqslant 60\}$  is linearly independent in $(QP_5)_{13}.$ Suppose there is a linear relation 
\begin{equation}\tag{3.2.1}\label{thtt13-1}
 \mathcal S = \sum\limits_{1\leqslant k\leqslant 60}\gamma_kb_{13,\, k}\equiv_{\omega_{(5,2)}} 0,
\end{equation}
where $\gamma_k\in\mathbb F_2.$

From a result in \cite{N.S4}, $\dim QP_4^+(\omega_{(5, 2)}) =  23,$  with the basis $\{[u_i]\ : \ 1\leqslant i\leqslant 23\},$ where

 \begin{tabular}{cllrrr}
    $u_{1}.\ x_1x_2^{2}x_3^{3}x_4^{7}$, & $u_{2}.\ x_1x_2^{2}x_3^{7}x_4^{3}$, & $u_{3}.\ x_1x_2^{3}x_3^{2}x_4^{7}$, & \multicolumn{1}{l}{$u_{4}.\ x_1x_2^{3}x_3^{3}x_4^{6}$,} & \multicolumn{1}{l}{$u_{5}.\ x_1x_2^{3}x_3^{6}x_4^{3}$,} &  \\
    $u_{6}.\ x_1x_2^{3}x_3^{7}x_4^{2}$, & $u_{7}.\ x_1x_2^{6}x_3^{3}x_4^{3}$, & $u_{8}.\ x_1x_2^{7}x_3^{2}x_4^{3}$, & \multicolumn{1}{l}{$u_{9}.\ x_1x_2^{7}x_3^{3}x_4^{2}$,} & \multicolumn{1}{l}{$u_{10}.\ x_1^{3}x_2x_3^{2}x_4^{7}$,} &  \\
    $u_{11}.\ x_1^{3}x_2x_3^{3}x_4^{6}$, & $u_{12}.\ x_1^{3}x_2x_3^{6}x_4^{3}$, & $u_{13}.\ x_1^{3}x_2x_3^{7}x_4^{2}$, & \multicolumn{1}{l}{$u_{14}.\ x_1^{3}x_2^{3}x_3x_4^{6}$,} & \multicolumn{1}{l}{$u_{15}.\ x_1^{3}x_2^{3}x_3^{3}x_4^{4}$,} &  \\
    $u_{16}.\ x_1^{3}x_2^{3}x_3^{4}x_4^{3}$, & $u_{17}.\ x_1^{3}x_2^{3}x_3^{5}x_4^{2}$, & $u_{18}.\ x_1^{3}x_2^{5}x_3^{2}x_4^{3}$, & \multicolumn{1}{l}{$u_{19}.\ x_1^{3}x_2^{5}x_3^{3}x_4^{2}$,} & \multicolumn{1}{l}{$u_{20}.\ x_1^{3}x_2^{7}x_3x_4^{2}$,} &  \\
    $u_{21}.\ x_1^{7}x_2x_3^{2}x_4^{3}$, & $u_{22}.\ x_1^{7}x_2x_3^{3}x_4^{2}$, & $u_{23}.\ x_1^{7}x_2^{3}x_3x_4^{2}.$ &       &       &  \\
\medskip
    \end{tabular}%
 
We explicitly compute $p_{(i; I)}(\mathcal S)$ in terms $u_j,\ 1\leqslant j\leqslant 23.$ By a direct computation using Lemma \ref{bdPS},Theorem \ref{dlsig}, and from the relations $p_{(i; j)}(\mathcal S)\equiv_{\omega_{(5,2)}} 0$ with either $i = 1, j = 2, 3$ or $i = 2, j = 3, 4,$ one gets
$$  \left\{\begin{array}{ll}
\gamma_k = 0,\ k\in\mathbb J,\\
\gamma_{27} = \gamma_{55}, \  \gamma_{14} = \gamma_{49} = \gamma_{58}, \  \gamma_{13} = \gamma_{20} =  \gamma_{43} = \gamma_{59}, \\
\gamma_{30} + \gamma_{35} + \gamma_{58} = 0, \  \gamma_{44} + \gamma_{55} + \gamma_{56} + \gamma_{58} = \gamma_{47} + \gamma_{53} + \gamma_{54} + \gamma_{59} = 0.
\end{array}\right.$$
Here $\mathbb J = \{1, 2, 3, 4, 5, 6, 7, 8, 9, 10, 11, 12, 15, 16, 17, 18, 19, 21, 22, 23, 24, 25, 26, 28, 29,\\ 31, 38, 39, 40, 41, 42, 45, 46, 48, 50, 51, 57, 60\}.$ Then, the relation \eqref{thtt13-1} becomes
\begin{equation}\tag{3.2.2}\label{thtt13-2} 
\mathcal S = \sum\limits_{k\,\in\, \mathbb N_{60}\setminus \mathbb J}\gamma_kb_{13,\, k} \equiv_{\omega_{(5,2)}} 0.
\end{equation}

Applying the homomorphisms $p_{(1; 4)},\, p_{(3; 4)},\, p_{(3; 5)},\, p_{(4; 5)}: P_5\to P_4$ to \eqref{thtt13-2}, we obtain $\gamma_k = 0,\ 1\leqslant k\leqslant 60.$ This finishes the proof.
 \end{proof}

Since $\dim(QP_5^0)_{13} = 145,$ from Theorem \ref{dldb} and Proposition \ref{md13}, one gets $\dim (QP_5)_{13} = 250.$

\medskip

For $t\geqslant 3, $ we have $|\mathscr B_4(3(2^t-1) + 2^t)| = 45$  (see Sum \cite{N.S4}). By a simple computation, we get
$$ \mathscr B_5^0(3(2^t-1) + 2^t) = \mathscr B_5^0(\omega_{(5, t)}) = \Phi^0(\mathscr B_4(3(2^t-1) + 2^t)) = \{q_{t,\, k}\, :\, 1\leqslant k\leqslant 195\},$$
for all $t\geqslant 3.$ Here, the monomials $q_{t,\, k},\, 1\leqslant k\leqslant 195,$ are  listed in Subsection \ref{dtcndtq-1}. 
\medskip

We now compute $QP_5^+(\omega_{(5, t)})$ for $t\geqslant 3.$ 
\medskip

{\bf The case \boldmath\mbox{$t = 3$}.} We obtain the following.

\begin{propo}\label{md29}
$QP_5^+(\omega_{(5, 3)})$ is the $\mathbb F_2$-vector space of dimension $260$ with a basis consisting of all the classes represented by the monomials $b_{29,\, k},\, 1\leqslant k\leqslant 260,$ which are determined as in Subsection \ref{dtcndtq-2}.
\end{propo}

To prove the proposition, we need some results. The following is a corollary of a result in \cite{N.S4}.

\begin{lema}\label{bd29-2}
If $\bar x$ is one of the following monomials then $\rho_i(\bar x),\, 1\leqslant i\leqslant 5$ are strictly inadmissible:
\medskip

    \begin{tabular}{lclrrr}
    $ x_1^{6}x_2x_3^{7}x_4^{7}$, & $ x_1^{6}x_2^{7}x_3x_4^{7}$, & $ x_1^{6}x_2^{7}x_3^{7}x_4$, & \multicolumn{1}{l}{$ x_1^{7}x_2^{6}x_3x_4^{7}$,} & \multicolumn{1}{l}{$ x_1^{7}x_2^{6}x_3^{7}x_4$,} & \multicolumn{1}{l}{$ x_1^{7}x_2^{7}x_3^{6}x_4$,} \\
    $ x_1^{2}x_2^{5}x_3^{7}x_4^{7}$, & $ x_1^{2}x_2^{7}x_3^{5}x_4^{7}$, & $ x_1^{2}x_2^{7}x_3^{7}x_4^{5}$, & \multicolumn{1}{l}{$ x_1^{7}x_2^{2}x_3^{5}x_4^{7}$,} & \multicolumn{1}{l}{$ x_1^{7}x_2^{2}x_3^{7}x_4^{5}$,} & \multicolumn{1}{l}{$ x_1^{7}x_2^{7}x_3^{2}x_4^{5}$,} \\
    $ x_1^{3}x_2^{4}x_3^{7}x_4^{7}$, & $ x_1^{3}x_2^{7}x_3^{4}x_4^{7}$, & $ x_1^{3}x_2^{7}x_3^{7}x_4^{4}$, & \multicolumn{1}{l}{$ x_1^{7}x_2^{3}x_3^{4}x_4^{7}$,} & \multicolumn{1}{l}{$ x_1^{7}x_2^{3}x_3^{7}x_4^{4}$,} & \multicolumn{1}{l}{$ x_1^{7}x_2^{7}x_3^{3}x_4^{4}$,} \\
    $ x_1^{3}x_2^{6}x_3^{5}x_4^{7}$, & $ x_1^{3}x_2^{6}x_3^{7}x_4^{5}$, & $ x_1^{3}x_2^{7}x_3^{6}x_4^{5}$, & \multicolumn{1}{l}{$ x_1^{6}x_2^{3}x_3^{5}x_4^{7}$,} & \multicolumn{1}{l}{$ x_1^{6}x_2^{3}x_3^{7}x_4^{5}$,} & \multicolumn{1}{l}{$ x_1^{6}x_2^{7}x_3^{3}x_4^{5}$,} \\
    $ x_1^{7}x_2^{3}x_3^{6}x_4^{5}$, & $ x_1^{7}x_2^{6}x_3^{3}x_4^{5}$, & $ x_1^{7}x_2^{7}x_3^{8}x_4^{7}.$ &       &       &  \\
    \end{tabular}%
 \end{lema}

\begin{lema}\label{bd29-3}
 The following monomials are strictly inadmissible:
\begin{enumerate}
\item[(i)]  $x_k^2x_{\ell}x_m^5x_t^6x_u^7,\, k < \ell,\, x_k^2x_{\ell}x_m^4x_t^7x_u^7,\, x_k^2x_{\ell}^4x_mx_t^7x_u^7,\, x_k^6x_{\ell}x_m^2x_t^5x_u^7,\, x_k^6x_{\ell}^3x_m^4x_tx_u^7,\, k < \ell < m.$ Here $(k, \ell, m, t, u)$ is a permutation of $(1,\, 2,\, 3,\, 4,\, 5).$  

\medskip

\item[(ii)] 
    \begin{tabular}[t]{lllr}
    $ x_1x_2^{3}x_3^{6}x_4^{6}x_5^{5}$, & $ x_1x_2^{6}x_3^{3}x_4^{6}x_5^{5}$, & $ x_1x_2^{6}x_3^{6}x_4^{3}x_5^{5}$, & \multicolumn{1}{l}{$ x_1^{3}x_2x_3^{6}x_4^{6}x_5^{5}$,} \\
    $ x_1^{3}x_2^{5}x_3^{6}x_4^{6}x_5$, & $ x_1^{3}x_2^{6}x_3x_4^{5}x_5^{6}$, & $ x_1^{3}x_2^{6}x_3x_4^{6}x_5^{5}$, & \multicolumn{1}{l}{$ x_1^{3}x_2^{6}x_3^{5}x_4x_5^{6}$,} \\
    $ x_1^{3}x_2^{6}x_3^{5}x_4^{6}x_5$, & $ x_1^{3}x_2^{6}x_3^{6}x_4x_5^{5}$, & $ x_1^{3}x_2^{6}x_3^{6}x_4^{5}x_5.$&  \\
    \end{tabular}%
  \end{enumerate}
\end{lema}

\begin{proof}
The proof of the lemma is straightforward.
\end{proof}

\begin{lema}\label{bd29-4}
If $(i, j, k, \ell, m)$ is a permutation of $(1, 2,3,4,5)$ such that $i<j < k<\ell$ then the following monomials are strictly inadmissible:

\begin{center}
    \begin{tabular}{lllll}
   $ x_ix_j^6x_k^6x_{\ell}x_m^7 $, &$x_i^6x_jx_kx_{\ell}^6x_m^7$, &  $x_i^6x_jx_k^6x_{\ell}x_m^7$, & \multicolumn{1}{l}{$x_i^6x_j^6x_kx_{\ell}x_m^7$,} & \multicolumn{1}{l}{$x_ix_j^2x_k^6x_{\ell}^5x_m^7$,} \\
$x_ix_j^6x_k^2x_{\ell}^5x_m^7$, & $x_i^3x_j^4x_kx_{\ell}^6x_m^7$,&$x_i^3x_j^4x_k^6x_{\ell}x_m^7$,&  \multicolumn{1}{l}{$x_i^3x_j^6x_k^4x_{\ell}x_m^7$,} & \multicolumn{1}{l}{$x_i^3x_j^6x_kx_{\ell}^4x_m^7$,} \\
$x_i^6x_jx_k^3x_{\ell}^5x_m^6$, & $x_i^6x_j^3x_kx_{\ell}^5x_m^6$, & $x_i^6x_j^3x_k^5x_{\ell}x_m^6$, & \multicolumn{1}{l}{$x_i^2x_j^2x_k^5x_{\ell}^5x_m^7$,} & \multicolumn{1}{l}{$x_i^2x_j^5x_k^2x_{\ell}^5x_m^7$,}   \\
   $x_i^3x_j^4x_k^3x_{\ell}^4x_m^7$, & $x_i^2x_j^3x_k^4x_{\ell}^5x_m^7$, & $x_i^2x_j^3x_k^5x_{\ell}^4x_m^7$, &
\multicolumn{1}{l}{$x_i^2x_j^4x_k^3x_{\ell}^5x_m^7$,} & \multicolumn{1}{l}{$x_i^2x_j^5x_k^3x_{\ell}^4x_m^7$,}\\
  $x_i^3x_j^2x_k^4x_{\ell}^5x_m^7$,  & $x_i^3x_j^2x_k^5x_{\ell}^4x_m^7$, & $x_i^3x_j^4x_k^2x_{\ell}^5x_m^7$, & \multicolumn{1}{l}{$x_i^2x_j^3x_k^5x_{\ell}^5x_m^6$,} & \multicolumn{1}{l}{$x_i^2x_j^5x_k^3x_{\ell}^5x_m^6$,}\\
$x_i^3x_j^2x_k^5x_{\ell}^5x_m^6$, & $x_i^3x_j^5x_k^2x_{\ell}^5x_m^6$. & $x_i^3x_j^3x_k^4x_{\ell}^5x_m^6$, & \multicolumn{1}{l}{$x_i^3x_j^3x_k^5x_{\ell}^4x_m^6$,} & \multicolumn{1}{l}{$x_i^3x_j^4x_k^3x_{\ell}^5x_m^6$,}\\
$x_i^3x_j^5x_k^3x_{\ell}^4x_m^6$. & &&&
 \end{tabular}
\end{center}
  \end{lema}

\begin{proof}
We prove the lemma for the monomial $x_ix_j^6x_k^6x_{\ell}x_m^7.$ The others can be proved by a similar computation. By using the Cartan formula, we have
$$ \begin{array}{ll}
\medskip
x_ix_j^6x_k^6x_{\ell}x_m^7 &= Sq^1(x_i^2x_j^5x_k^5x_{\ell}x_m^7) + Sq^2(x_ix_j^5x_k^5x_{\ell}x_m^7 + x_ix_j^3x_k^5x_{\ell}x_m^9) \\
&\quad+ x_ix_j^6x_k^5x_{\ell}^2x_m^7 + x_ix_j^5x_k^6x_{\ell}^2x_m^7 \ \ {\rm mod}(P_5^-(\omega^*_{(5, 3)})).
\end{array}$$
This relation implies that $x_ix_j^6x_k^6x_{\ell}x_m^7$  is strictly inadmissible. The lemma is proven.
\end{proof}

\begin{lema}\label{bd29-5}
 The following monomials are strictly inadmissible:
\medskip

    \begin{tabular}{llrrr}
    $ x_1x_2^{6}x_3^{7}x_4^{8}x_5^{7}$, & $ x_1x_2^{7}x_3^{6}x_4^{8}x_5^{7}$, & \multicolumn{1}{l}{$ x_1x_2^{7}x_3^{10}x_4^{4}x_5^{7}$,} & \multicolumn{1}{l}{$ x_1x_2^{7}x_3^{10}x_4^{7}x_5^{4}$,} & \multicolumn{1}{l}{$ x_1^{3}x_2^{3}x_3^{4}x_4^{12}x_5^{7}$,} \\
    $ x_1^{3}x_2^{3}x_3^{12}x_4^{4}x_5^{7}$, & $ x_1^{3}x_2^{3}x_3^{12}x_4^{7}x_5^{4}$, & \multicolumn{1}{l}{$ x_1^{3}x_2^{5}x_3^{6}x_4^{3}x_5^{12}$,} & \multicolumn{1}{l}{$ x_1^{3}x_2^{5}x_3^{6}x_4^{11}x_5^{4}$,} & \multicolumn{1}{l}{$ x_1^{3}x_2^{5}x_3^{8}x_4^{6}x_5^{7}$,} \\
    $ x_1^{3}x_2^{5}x_3^{8}x_4^{7}x_5^{6}$, & $ x_1^{3}x_2^{5}x_3^{9}x_4^{6}x_5^{6}$, & \multicolumn{1}{l}{$ x_1^{3}x_2^{5}x_3^{14}x_4^{3}x_5^{4}$,} & \multicolumn{1}{l}{$ x_1^{3}x_2^{12}x_3^{3}x_4^{5}x_5^{6}$,} & \multicolumn{1}{l}{$ x_1^{3}x_2^{13}x_3^{6}x_4^{3}x_5^{4}$,} \\
    $ x_1^{7}x_2x_3^{6}x_4^{8}x_5^{7}$, & $ x_1^{7}x_2x_3^{10}x_4^{4}x_5^{7}$, & \multicolumn{1}{l}{$ x_1^{7}x_2x_3^{10}x_4^{7}x_5^{4}$,} & \multicolumn{1}{l}{$ x_1^{7}x_2^{8}x_3^{3}x_4^{5}x_5^{6}$,} & \multicolumn{1}{l}{$ x_1^{7}x_2^{9}x_3^{2}x_4^{4}x_5^{7}$,} \\
    $ x_1^{7}x_2^{9}x_3^{2}x_4^{7}x_5^{4}$, & $ x_1^{7}x_2^{9}x_3^{7}x_4^{2}x_5^{4}.$ &       &       &  
    \end{tabular}
\end{lema}

\begin{proof}
We prove the lemma for the monomial $x = x_1^{3}x_2^{5}x_3^{6}x_4^{3}x_5^{12}.$  The others can be proved by a similar computation.  By a direct computation, we have
$$ \begin{array}{ll}
\medskip
 x &=  x_1^{2}x_2^{3}x_3^{5}x_4^{6}x_5^{13}+
x_1^{2}x_2^{3}x_3^{5}x_4^{12}x_5^{7}+
x_1^{2}x_2^{5}x_3^{5}x_4^{6}x_5^{11}+
x_1^{2}x_2^{5}x_3^{5}x_4^{10}x_5^{7}+
x_1^{2}x_2^{5}x_3^{9}x_4^{6}x_5^{7}\\
\medskip
&\quad + x_1^{2}x_2^{11}x_3^{4}x_4^{5}x_5^{7}+
x_1^{2}x_2^{11}x_3^{5}x_4^{4}x_5^{7}+
x_1^{2}x_2^{13}x_3^{3}x_4^{4}x_5^{7}+
x_1^{2}x_2^{13}x_3^{4}x_4^{3}x_5^{7}+
x_1^{3}x_2^{3}x_3^{4}x_4^{6}x_5^{13}\\
\medskip
&\quad + x_1^{3}x_2^{3}x_3^{4}x_4^{12}x_5^{7}+
x_1^{3}x_2^{3}x_3^{5}x_4^{6}x_5^{12}+
x_1^{3}x_2^{3}x_3^{6}x_4^{5}x_5^{12}+
x_1^{3}x_2^{3}x_3^{12}x_4^{4}x_5^{7}+
x_1^{3}x_2^{4}x_3^{3}x_4^{6}x_5^{13}\\
\medskip
&\quad + x_1^{3}x_2^{4}x_3^{3}x_4^{12}x_5^{7}+
x_1^{3}x_2^{4}x_3^{5}x_4^{6}x_5^{11}+
x_1^{3}x_2^{4}x_3^{5}x_4^{10}x_5^{7}+
x_1^{3}x_2^{4}x_3^{6}x_4^{3}x_5^{13}+
x_1^{3}x_2^{4}x_3^{6}x_4^{9}x_5^{7}\\
\medskip
&\quad + x_1^{3}x_2^{4}x_3^{12}x_4^{3}x_5^{7}+
x_1^{3}x_2^{5}x_3^{2}x_4^{5}x_5^{14}+
x_1^{3}x_2^{5}x_3^{2}x_4^{6}x_5^{13}+
x_1^{3}x_2^{5}x_3^{3}x_4^{4}x_5^{14}+
x_1^{3}x_2^{5}x_3^{5}x_4^{6}x_5^{10}\\
\medskip
&\quad + x_1^{3}x_2^{5}x_3^{6}x_4^{2}x_5^{13}+
x_1^{5}x_2^{5}x_3^{5}x_4^{5}x_5^{9}+ Sq^1(h_1) + Sq^2(h_2) + Sq^4(h_4) \ \ {\rm mod}(P_5^-(\omega_{(5, 3)})),  \ \mbox{where}\\
\medskip
h_1 &= x_1^{3}x_2^{3}x_3^{5}x_4^{6}x_5^{11}+
x_1^{3}x_2^{3}x_3^{5}x_4^{10}x_5^{7}+
x_1^{3}x_2^{3}x_3^{9}x_4^{2}x_5^{11}+
x_1^{3}x_2^{3}x_3^{9}x_4^{3}x_5^{10}+
x_1^{3}x_2^{3}x_3^{9}x_4^{6}x_5^{7}\\
\medskip
&\quad +x_1^{3}x_2^{5}x_3^{2}x_4^{5}x_5^{13}+
x_1^{3}x_2^{5}x_3^{3}x_4^{4}x_5^{13}+
x_1^{3}x_2^{5}x_3^{4}x_4^{5}x_5^{11}+
x_1^{3}x_2^{5}x_3^{4}x_4^{9}x_5^{7}+
x_1^{3}x_2^{5}x_3^{9}x_4^{4}x_5^{7}\\
\medskip
&\quad + x_1^{3}x_2^{6}x_3^{5}x_4^{5}x_5^{9}+
x_1^{3}x_2^{9}x_3^{2}x_4^{3}x_5^{11}+
x_1^{3}x_2^{11}x_3^{3}x_4^{4}x_5^{7}+
x_1^{3}x_2^{11}x_3^{4}x_4^{3}x_5^{7}+
x_1^{10}x_2^{3}x_3^{3}x_4^{3}x_5^{9}\\
\medskip
&\quad + x_1^{10}x_2^{3}x_3^{3}x_4^{5}x_5^{7}+
x_1^{10}x_2^{3}x_3^{5}x_4^{3}x_5^{7}+
x_1^{10}x_2^{5}x_3^{3}x_4^{3}x_5^{7},\\
\medskip
h_2 &=  x_1^{2}x_2^{3}x_3^{5}x_4^{6}x_5^{11}+
x_1^{2}x_2^{3}x_3^{5}x_4^{10}x_5^{7}+
x_1^{2}x_2^{3}x_3^{9}x_4^{6}x_5^{7}+
x_1^{2}x_2^{11}x_3^{3}x_4^{4}x_5^{7}+
x_1^{2}x_2^{11}x_3^{4}x_4^{3}x_5^{7}\\
\medskip
&\quad + x_1^{3}x_2^{3}x_3^{4}x_4^{6}x_5^{11}+
x_1^{3}x_2^{3}x_3^{4}x_4^{10}x_5^{7}+
x_1^{3}x_2^{3}x_3^{6}x_4^{4}x_5^{11}+
x_1^{3}x_2^{3}x_3^{6}x_4^{8}x_5^{7}+
x_1^{3}x_2^{3}x_3^{8}x_4^{6}x_5^{7}\\
\medskip
&\quad + x_1^{3}x_2^{3}x_3^{10}x_4^{4}x_5^{7}+
x_1^{3}x_2^{4}x_3^{3}x_4^{6}x_5^{11}+
x_1^{3}x_2^{4}x_3^{3}x_4^{10}x_5^{7}+
x_1^{3}x_2^{4}x_3^{6}x_4^{3}x_5^{11}+
x_1^{3}x_2^{4}x_3^{10}x_4^{3}x_5^{7}\\
\medskip
&\quad + x_1^{3}x_2^{5}x_3^{5}x_4^{5}x_5^{9}+
x_1^{3}x_2^{6}x_3^{2}x_4^{5}x_5^{11}+
x_1^{3}x_2^{6}x_3^{2}x_4^{9}x_5^{7}+
x_1^{3}x_2^{6}x_3^{3}x_4^{4}x_5^{11}+
x_1^{3}x_2^{6}x_3^{3}x_4^{8}x_5^{7}\\
\medskip
&\quad + x_1^{3}x_2^{8}x_3^{3}x_4^{6}x_5^{7}+
x_1^{3}x_2^{8}x_3^{6}x_4^{3}x_5^{7}+
x_1^{3}x_2^{10}x_3^{2}x_4^{5}x_5^{7}+
x_1^{3}x_2^{10}x_3^{3}x_4^{4}x_5^{7}+
x_1^{5}x_2^{3}x_3^{6}x_4^{2}x_5^{11}\\
\medskip
&\quad +x_1^{5}x_2^{3}x_3^{6}x_4^{3}x_5^{10}+
x_1^{5}x_2^{3}x_3^{6}x_4^{6}x_5^{7}+
x_1^{5}x_2^{10}x_3^{2}x_4^{3}x_5^{7}+
x_1^{9}x_2^{3}x_3^{3}x_4^{3}x_5^{9}+
x_1^{9}x_2^{3}x_3^{3}x_4^{5}x_5^{7}\\
\medskip
&\quad + x_1^{9}x_2^{3}x_3^{5}x_4^{3}x_5^{7}+
x_1^{9}x_2^{5}x_3^{3}x_4^{3}x_5^{7}+
x_1^{9}x_2^{6}x_3^{2}x_4^{3}x_5^{7},\\
\medskip
h_4 &=  x_1^{3}x_2^{3}x_3^{6}x_4^{2}x_5^{11}+
x_1^{3}x_2^{3}x_3^{6}x_4^{3}x_5^{10}+
x_1^{3}x_2^{3}x_3^{6}x_4^{6}x_5^{7}+
x_1^{3}x_2^{4}x_3^{5}x_4^{6}x_5^{7}+
x_1^{3}x_2^{4}x_3^{6}x_4^{5}x_5^{7}\\
\medskip
&\quad + x_1^{3}x_2^{6}x_3^{5}x_4^{4}x_5^{7}+
x_1^{3}x_2^{10}x_3^{2}x_4^{3}x_5^{7}+
x_1^{5}x_2^{3}x_3^{4}x_4^{6}x_5^{7}+
x_1^{5}x_2^{3}x_3^{6}x_4^{4}x_5^{7}+
x_1^{5}x_2^{4}x_3^{3}x_4^{6}x_5^{7}\\
\medskip
&\quad +
x_1^{5}x_2^{4}x_3^{6}x_4^{3}x_5^{7}+
x_1^{5}x_2^{6}x_3^{2}x_4^{5}x_5^{7}+
x_1^{5}x_2^{6}x_3^{3}x_4^{4}x_5^{7}.
\end{array}$$
Now, let $g = x_1x_2x_3x_4x_5^2$ and $h = x_1x_2x_3x_4x_5.$ Then, we have
$$ u:= x_1^{5}x_2^{5}x_3^{5}x_4^{5}x_5^{9} = h.g^4 = hSq^{12}(g^2).$$
On the other hand, by the formula $\chi$-trick (see \cite{W.W2}), one gets
$$ u = g^2\chi(Sq^{12})(h) + \sum_{1\leqslant i\leqslant 12}Sq^i(g^2\chi(Sq^{12-i})(h)).$$
Since $\chi(Sq^{12}) = Sq^8Sq^4 + Sq^8Sq^3Sq^1,$ by a direct computation, we coclude that 
$$g^2\chi(Sq^{12})(h) \in P_5^-(\omega_{(5, 3)}).$$
This implies that $x$ is strictly inadmissible.  The lemma follows. 
\end{proof}

\begin{proof}[{\it Proof of Proposition \ref{md29}}]
Let $x$ be an admissible monomial  in $(P_5^+)_{29}$ such that  $\omega(x) = \omega_{(5, 3)}.$ Then $x = X_{\{i, j\}}y^2$ with $y\in \mathscr{B}_5(\omega_{(5, 2)}),$ and $1\leqslant  i<j \leqslant 5.$ 

Let $u\in \mathscr{B}_5(\omega_{(5, 2)})$ such that $X_{\{i, j\}}u^2\in P_5^+.$  By a direct computation using Proposition \ref{md13}, we see that if $X_{\{i, j\}}u^2\neq b_{29,\, k},\ \forall t,\ 1\leqslant k\leqslant 260,$ then there is a monomial $w$ which is given in one of Lemmas \ref{bd29-2}, \ref{bd29-3}, \ref{bd29-4} and \ref{bd29-5} such that $X_{\{i, j\}}u^2 = wu_1^{2^{\ell}}$ with suitable monomial $u_1\in P_5,$ and $\ell = {\rm max}\{r\in\mathbb{Z}\,\,:\,\,\omega_r(w) > 0\}.$  By Theorem \ref{dlKS}, $X_{\{i, j\}}u^2$ is inadmissible. Since $x = X_{\{i, j\}}y^2$ with $y\in \mathscr{B}_5(\omega_{(5, 2)}),$ and $x$ is admissible, one can see that $x = b_{29,\, k}.$ for some $k,\ 1\leqslant k\leqslant 260.$ This implies $\mathscr B_5^+(\omega_{(5, 3)})\subseteq \{b_{29,\, k}\, : \, 1\leqslant k\leqslant 260\}.$ 

Now, we prove the set $\{[b_{29,\, k}]\ :\ 1\leqslant k\leqslant 260\}$  is linearly independent in $(QP_5)_{29}.$ Suppose there is a linear relation 
$$ \mathcal S = \sum\limits_{1\leqslant k\leqslant 260}\gamma_kb_{29,\, k}\equiv_{\omega_{(5,3)}} 0,$$
where $\gamma_k\in\mathbb F_2$ for all $k,\ 1 \leqslant k \leqslant 260.$

From Theorem \ref{dlsig}, for $(i;I) \in \mathcal N_5$, we explicitly compute $p_{(i;I)}(\mathcal S)$ in terms of a given minimal set of $\mathscr A$-generators in $P_4\ ({\rm mod}(\mathscr A^+\cdot P_{4})).$ Based on the relations $p_{(i;I)}(\mathcal S) \equiv_{\omega_{(5,3)}} 0,\ \forall (i; I)\in\mathcal N_5,\ \ell(I) > 0,$ and some other simple computations, we get $\gamma_k = 0,\, k  =1, 2, \ldots, 260.$ The proof is completed. 
\end{proof}

Recall that $\dim(QP_5^0)_{29} = 195.$ Combining this with Theorem \ref{dldb} and Proposition \ref{md29}, we obtain $\dim (QP_5)_{29} = 645.$
\medskip

{\bf The case \boldmath\mbox{$t \geqslant 4$}.} We have the following.

\begin{propo}\label{mdtq}
There exist exactly $270$ admissible monomials in $P_5^+$ such that their weight vectors are $\omega_{(5, t)},$ for every $t\geqslant 4.$ Consequently $\dim(QP_5^+(\omega_{(5, t)})) = 270$ for $t\geqslant 4.$
\end{propo}

We prove the proposition by showing that 
$$ \mathscr B_5^+(\omega_{(5, t)}) = \{b_{3(2^t-1)+2^t, k}:\ 1\leqslant k\leqslant 270\},$$ for all $t\geqslant 4.$ 
Here, the monomials $b_{3(2^t-1)+2^t, k},\ 1\leqslant k\leqslant 270$, are listed in Subsection \ref{dtcndtq-2}.

\begin{lema}\label{Bdtq-1}
If $(i, j, k, \ell, m)$ is a permutation of $(1, 2,3,4,5)$ such that $i<j < k<\ell$ then the following monomials are strictly inadmissible:
\medskip

\begin{center}
    \begin{tabular}{lrrr}
    $x_i^6x_jx_kx_{\ell}^6x_m^7,$ & \multicolumn{1}{l}{$x_ix_j^2x_k^6x_{\ell}^5x_m^7,$} & \multicolumn{1}{l}{$x_ix_j^6x_k^2x_{\ell}^5x_m^7,$} & \multicolumn{1}{l}{$x_i^6x_jx_k^2x_{\ell}^5x_m^7,$} \\
    $x_ix_j^6x_k^3x_{\ell}^4x_m^7,$ & \multicolumn{1}{l}{$x_i^3x_j^6x_kx_{\ell}^4x_m^7,$} & \multicolumn{1}{l}{$x_i^6x_jx_k^3x_{\ell}^4x_m^7,$} & \multicolumn{1}{l}{$x_i^6x_j^3x_kx_{\ell}^4x_m^7,$} \\
    $x_i^6x_jx_k^3x_{\ell}^5x_m^6,$ & \multicolumn{1}{l}{$x_i^6x_j^3x_k^5x_{\ell}x_m^6,$} & \multicolumn{1}{l}{$x_i^6x_j^3x_kx_{\ell}^5x_m^6,$} & \multicolumn{1}{l}{$x_i^2x_j^2x_k^5x_{\ell}^5x_m^7,$} \\
    $x_i^2x_j^5x_k^2x_{\ell}^5x_m^7,$ & \multicolumn{1}{l}{$x_i^2x_j^3x_k^4x_{\ell}^5x_m^7,$} & \multicolumn{1}{l}{$x_i^2x_j^3x_k^5x_{\ell}^4x_m^7,$} & \multicolumn{1}{l}{$x_i^2x_j^4x_k^3x_{\ell}^5x_m^7,$} \\
    $x_i^2x_j^5x_k^3x_{\ell}^4x_m^7,$ & \multicolumn{1}{l}{$x_i^3x_j^2x_k^4x_{\ell}^5x_m^7,$} & \multicolumn{1}{l}{$x_i^3x_j^2x_k^5x_{\ell}^4x_m^7,$} & \multicolumn{1}{l}{$x_i^3x_j^4x_k^2x_{\ell}^5x_m^7,$} \\
    $x_i^2x_j^5x_k^3x_{\ell}^5x_m^6,$ & \multicolumn{1}{l}{$x_i^3x_j^2x_k^5x_{\ell}^5x_m^6,$} & \multicolumn{1}{l}{$x_i^3x_j^4x_k^3x_{\ell}^4x_m^7,$} & \multicolumn{1}{l}{$x_ix_j^7x_k^{10}x_{\ell}^{12}x_m^{15},$} \\
    $x_i^7x_jx_k^{10}x_{\ell}^{12}x_m^{15},$ & \multicolumn{1}{l}{$x_i^3x_j^3x_k^{12}x_{\ell}^{12}x_m^{15},$} & \multicolumn{1}{l}{$x_i^3x_j^5x_k^8x_{\ell}^{14}x_m^{15},$} & \multicolumn{1}{l}{$x_i^3x_j^5x_k^{14}x_{\ell}^8x_m^{15},$} \\
    $x_i^7x_j^7x_k^8x_{\ell}^8x_m^{15}.$ &       &       &  \\
    \end{tabular}%
\end{center}
  \end{lema}

\begin{proof}

We prove the lemma for the monomial $x = x_i^7x_j^7x_k^8x_{\ell}^8x_m^{15}.$ The others can be proved by a similar computation. 

Applying the Cartan formula, we have
$$ \begin{array}{ll}
\medskip
x &= x_i^7x_j^5x_k^{10}x_{\ell}^8x_m^{15} + x_i^5x_j^{11}x_k^6x_{\ell}^8x_m^{15} + x_i^5x_j^{7}x_k^{10}x_{\ell}^8x_m^{15}\\
\medskip
&\quad + x_i^5x_j^{11}x_k^{10}x_{\ell}^4x_m^{15} + x_i^5x_j^{7}x_k^{10}x_{\ell}^8x_m^{15} + x_i^5x_j^{10}x_k^{11}x_{\ell}^4x_m^{15}\\
\medskip
&\quad + x_i^5x_j^{6}x_k^{11}x_{\ell}^8x_m^{15} + x_i^5x_j^{3}x_k^{14}x_{\ell}^8x_m^{15} + x_i^4x_j^{11}x_k^{11}x_{\ell}^4x_m^{15}\\
\medskip
&\quad + x_i^4x_j^{7}x_k^{11}x_{\ell}^8x_m^{15} + Sq^1(x_i^7x_j^{7}x_k^{11}x_{\ell}^4x_m^{15} + x_i^7x_j^{5}x_k^{13}x_{\ell}^4x_m^{15})\\
\medskip
&\quad + Sq^2(y) + Sq^4(z) + Sq^8(x_i^7x_j^{5}x_k^6x_{\ell}^4x_m^{15}) \ \ {\rm mod}(P_5^-(\omega^*_{(5, 4)})),  \ \mbox{where}\\
\medskip
y &= x_i^7x_j^7x_k^6x_{\ell}^8x_m^{15} + x_i^7x_j^3x_k^6x_{\ell}^4x_m^{23} + x_i^7x_j^7x_k^{10}x_{\ell}^4x_m^{15}\ + x_i^7x_j^6x_k^{11}x_{\ell}^4x_m^{15} + x_i^7x_j^3x_k^{14}x_{\ell}^4x_m^{15},\\
\medskip
z&= x_i^5x_j^7x_k^6x_{\ell}^8x_m^{15} + x_i^{11}x_j^5x_k^6x_{\ell}^4x_m^{15} + x_i^5x_j^7x_k^{10}x_{\ell}^4x_m^{15}\\
&\quad + x_i^4x_j^7x_k^{11}x_{\ell}^4x_m^{15} + x_i^5x_j^6x_k^{11}x_{\ell}^4x_m^{15} + x_i^5x_j^3x_k^{14}x_{\ell}^4x_m^{15}.
\end{array}$$
The above equalities imply that $x$ is strictly inadmissible.
\end{proof} 

The following lemma is proved by a direct computation.

\begin{lema}\label{Bdtq-2}
The following monomials are strictly inadmissible:
\begin{enumerate}
\item[(i)]  $x_i^6x_j^6x_kx_lx_m^7, \ x_i^2x_jx_k^4x_l^7x_m^7, \ x_i^2x_j^4x_kx_l^7x_m^7, \ x_i^2x_j^5x_kx_l^6x_m^7, \ x_i^2x_jx_k^5x_l^6x_m^7, \ x_i^3x_j^4x_kx_l^6x_m^7$, whenever $(i, j, k, l, m)$ is a permutation of $(1, 2, 3, 4, 5)$ such that $i<j<k.$

\medskip
\item[(ii)] 
    \begin{tabular}[t]{lrrr}
    $ x_1^{3}x_2^{3}x_3^{6}x_4^{4}x_5^{5}$, & \multicolumn{1}{l}{$ x_1^{3}x_2^{3}x_3^{6}x_4^{5}x_5^{4}$,} & \multicolumn{1}{l}{$ x_1^{3}x_2^{4}x_3^{3}x_4^{6}x_5^{5}$,} & \multicolumn{1}{l}{$ x_1^{3}x_2^{4}x_3^{6}x_4^{3}x_5^{5}$,} \\
    $ x_1^{3}x_2^{6}x_3^{3}x_4^{4}x_5^{5}$, & \multicolumn{1}{l}{$ x_1^{3}x_2^{6}x_3^{3}x_4^{5}x_5^{4}$,} & \multicolumn{1}{l}{$ x_1^{3}x_2^{6}x_3^{4}x_4^{3}x_5^{5}$,} & \multicolumn{1}{l}{$ x_1^{3}x_2^{6}x_3^{5}x_4^{3}x_5^{4}$,} \\
    $ x_1^{6}x_2^{3}x_3^{3}x_4^{4}x_5^{5}$, & \multicolumn{1}{l}{$ x_1^{6}x_2^{3}x_3^{3}x_4^{5}x_5^{4}$,} & \multicolumn{1}{l}{$ x_1^{6}x_2^{3}x_3^{4}x_4^{3}x_5^{5}$,} & \multicolumn{1}{l}{$ x_1^{6}x_2^{3}x_3^{5}x_4^{3}x_5^{4}$,} \\
    $ x_1^{3}x_2^{5}x_3^{2}x_4^{6}x_5^{5}$, & \multicolumn{1}{l}{$ x_1^{3}x_2^{5}x_3^{6}x_4^{2}x_5^{5}$,} & \multicolumn{1}{l}{$ x_1^{3}x_2^{6}x_3^{5}x_4^{2}x_5^{5}$,} & \multicolumn{1}{l}{$ x_1^{6}x_2^{3}x_3^{5}x_4^{2}x_5^{5}$,} \\
    $ x_1^{3}x_2^{3}x_3^{4}x_4^{6}x_5^{5}$, &   $ x_1^{3}x_2^{5}x_3^{9}x_4^{14}x_5^{14}$,     &    $ x_1^{3}x_2^{5}x_3^{14}x_4^{9}x_5^{14}.$   &  
    \end{tabular}%
   \end{enumerate}
\end{lema}

\begin{lema}\label{Bdtq-3}
The following monomials are strictly inadmissible:
\medskip

\begin{enumerate}
\item[(i)]  $\rho_i(f),\ 1\leqslant i\leqslant 5,$  where $f$ is one of the following monomials:
\medskip

    \begin{tabular}{llrr}
    $ x_1^{2}x_2^{5}x_3^{7}x_4^{7}$, & $ x_1^{2}x_2^{7}x_3^{5}x_4^{7}$, & \multicolumn{1}{l}{$ x_1^{2}x_2^{7}x_3^{7}x_4^{5}$,} & \multicolumn{1}{l}{$ x_1^{7}x_2^{2}x_3^{5}x_4^{7}$,} \\
    $ x_1^{7}x_2^{2}x_3^{7}x_4^{5}$, & $ x_1^{7}x_2^{7}x_3^{2}x_4^{5}$, & \multicolumn{1}{l}{$ x_1^{3}x_2^{4}x_3^{7}x_4^{7}$,} & \multicolumn{1}{l}{$ x_1^{3}x_2^{7}x_3^{4}x_4^{7}$,} \\
    $ x_1^{3}x_2^{7}x_3^{7}x_4^{4}$, & $ x_1^{7}x_2^{3}x_3^{4}x_4^{7}$, & \multicolumn{1}{l}{$ x_1^{7}x_2^{3}x_3^{7}x_4^{4}$,} & \multicolumn{1}{l}{$ x_1^{7}x_2^{7}x_3^{3}x_4^{4}$,} \\
    $ x_1^{3}x_2^{6}x_3^{5}x_4^{7}$, & $ x_1^{3}x_2^{6}x_3^{7}x_4^{5}$, & \multicolumn{1}{l}{$ x_1^{3}x_2^{7}x_3^{6}x_4^{5}$,} & \multicolumn{1}{l}{$ x_1^{6}x_2^{3}x_3^{5}x_4^{7}$,} \\
    $ x_1^{6}x_2^{3}x_3^{7}x_4^{5}$, & $ x_1^{6}x_2^{7}x_3^{3}x_4^{5}$, & \multicolumn{1}{l}{$ x_1^{7}x_2^{3}x_3^{6}x_4^{5}$,} & \multicolumn{1}{l}{$ x_1^{7}x_2^{6}x_3^{3}x_4^{5}$,} \\
    $ x_1^{6}x_2x_3^{7}x_4^{7}$, & $ x_1^{6}x_2^{7}x_3x_4^{7}$, & \multicolumn{1}{l}{$ x_1^{6}x_2^{7}x_3^{7}x_4$,} & \multicolumn{1}{l}{$ x_1^{7}x_2^{6}x_3x_4^{7}$,} \\
    $ x_1^{7}x_2^{6}x_3^{7}x_4$, & $ x_1^{7}x_2^{7}x_3^{6}x_4$.&       &  \\
    \end{tabular}%

\medskip
\item[(ii)]  
   \begin{tabular}[t]{llrr}
    $ x_1x_2^{15}x_3^{15}x_4^{18}x_5^{12}$, & $ x_1^{15}x_2x_3^{15}x_4^{18}x_5^{12}$, & \multicolumn{1}{l}{$ x_1^{15}x_2^{15}x_3x_4^{18}x_5^{12}$,} & \multicolumn{1}{l}{$ x_1x_2^{14}x_3^{15}x_4^{15}x_5^{16}$,} \\
    $ x_1x_2^{15}x_3^{14}x_4^{15}x_5^{16}$, & $ x_1^{15}x_2x_3^{14}x_4^{15}x_5^{16}$, & \multicolumn{1}{l}{$ x_1x_2^{15}x_3^{15}x_4^{14}x_5^{16}$,} & \multicolumn{1}{l}{$ x_1^{15}x_2x_3^{15}x_4^{14}x_5^{16}$,} \\
    $ x_1^{15}x_2^{15}x_3x_4^{14}x_5^{16}$, & $ x_1^{15}x_2^{15}x_3^{17}x_4^{2}x_5^{12}$, & \multicolumn{1}{l}{$ x_1^{3}x_2^{13}x_3^{14}x_4^{19}x_5^{12}$,} & \multicolumn{1}{l}{$ x_1^{3}x_2^{13}x_3^{15}x_4^{18}x_5^{12}$,} \\
    $ x_1^{3}x_2^{15}x_3^{13}x_4^{18}x_5^{12}$, & $ x_1^{15}x_2^{3}x_3^{13}x_4^{18}x_5^{12}$, & \multicolumn{1}{l}{$ x_1^{3}x_2^{13}x_3^{14}x_4^{15}x_5^{16}$,} & \multicolumn{1}{l}{$ x_1^{3}x_2^{13}x_3^{15}x_4^{14}x_5^{16}$,} \\
    $ x_1^{3}x_2^{15}x_3^{13}x_4^{14}x_5^{16}$, & $ x_1^{15}x_2^{3}x_3^{13}x_4^{14}x_5^{16}$, & \multicolumn{1}{l}{$ x_1^{7}x_2^{7}x_3^{8}x_4^{9}x_5^{30}$,} & \multicolumn{1}{l}{$ x_1^{7}x_2^{7}x_3^{9}x_4^{8}x_5^{30}$,} \\
    $ x_1^{7}x_2^{7}x_3^{9}x_4^{30}x_5^{8}$, & $ x_1^{7}x_2^{7}x_3^{9}x_4^{14}x_5^{24}$, & \multicolumn{1}{l}{$ x_1^{7}x_2^{11}x_3^{13}x_4^{16}x_5^{14}$,} & \multicolumn{1}{l}{$ x_1^{7}x_2^{15}x_3^{17}x_4^{10}x_5^{12}$,} \\
    $ x_1^{15}x_2^{7}x_3^{17}x_4^{10}x_5^{12}$, & $ x_1^{15}x_2^{19}x_3^{5}x_4^{10}x_5^{12}$. &       &  \\
    \end{tabular}%
  \end{enumerate}
\end{lema}

\begin{proof}
The first part of the lemma is an immediate corollary from a result in \cite{N.S4}.
We prove the second part of this lemma for the monomials $x = x_1x_2^{14}x_3^{15}x_4^{15}x_5^{16}$ and $y = x_1^{7}x_2^{7}x_3^{9}x_4^{8}x_5^{30}.$ The others can be proved by a similar computation. By a direct computation, one gets
$$ \begin{array}{ll}
x &=  x_1x_2^{10}x_3^{12}x_4^{15}x_5^{23}+
 x_1x_2^{10}x_3^{12}x_4^{23}x_5^{15}+
 x_1x_2^{10}x_3^{14}x_4^{13}x_5^{23}+
 x_1x_2^{10}x_3^{14}x_4^{21}x_5^{15}\\
\medskip
 &\quad + x_1x_2^{10}x_3^{20}x_4^{15}x_5^{15}+
 x_1x_2^{10}x_3^{22}x_4^{13}x_5^{15}+
 x_1x_2^{12}x_3^{15}x_4^{10}x_5^{23}+
 x_1x_2^{12}x_3^{15}x_4^{15}x_5^{18}\\
\medskip
 &\quad + x_1x_2^{12}x_3^{18}x_4^{15}x_5^{15}+
 x_1x_2^{12}x_3^{26}x_4^{7}x_5^{15}+
 x_1x_2^{14}x_3^{12}x_4^{15}x_5^{19}+
 x_1x_2^{14}x_3^{12}x_4^{19}x_5^{15}\\
\medskip
 &\quad + x_1x_2^{14}x_3^{14}x_4^{13}x_5^{19}+
 x_1x_2^{14}x_3^{14}x_4^{17}x_5^{15}+
 x_1x_2^{14}x_3^{15}x_4^{8}x_5^{23}\\
\medskip
&\quad + Sq^1(f_1)  + Sq^2(f_2) + Sq^4(f_4)  + Sq^8(f_8) \ \ {\rm mod}(P_5^-(\omega_{(5, 4)})),  \ \mbox{where}\\
\medskip

f_1&=  x_1^{2}x_2^{11}x_3^{15}x_4^{15}x_5^{17}+
 x_1^{2}x_2^{11}x_3^{15}x_4^{17}x_5^{15}+
 x_1^{2}x_2^{11}x_3^{17}x_4^{15}x_5^{15}+
 x_1^{2}x_2^{11}x_3^{23}x_4^{7}x_5^{17}\\
\medskip
  &\quad + x_1^{2}x_2^{11}x_3^{23}x_4^{9}x_5^{15}+
 x_1^{2}x_2^{11}x_3^{25}x_4^{7}x_5^{15}+
 x_1^{2}x_2^{13}x_3^{15}x_4^{15}x_5^{15}+
 x_1^{2}x_2^{13}x_3^{23}x_4^{7}x_5^{15},\\
\medskip
f_2&=  x_1x_2^{11}x_3^{15}x_4^{15}x_5^{17}+
 x_1x_2^{11}x_3^{15}x_4^{17}x_5^{15}+
 x_1x_2^{11}x_3^{17}x_4^{15}x_5^{15}+
 x_1x_2^{11}x_3^{23}x_4^{7}x_5^{17}\\
\medskip
 &\quad + x_1x_2^{11}x_3^{23}x_4^{9}x_5^{15}+
 x_1x_2^{11}x_3^{25}x_4^{7}x_5^{15}+
 x_1x_2^{13}x_3^{15}x_4^{15}x_5^{15}+
 x_1x_2^{13}x_3^{23}x_4^{7}x_5^{15}\\
\medskip
  &\quad + x_1x_2^{14}x_3^{18}x_4^{11}x_5^{15}+
 x_1x_2^{22}x_3^{15}x_4^{6}x_5^{15},\\
\medskip
f_4&=   x_1x_2^{14}x_3^{12}x_4^{15}x_5^{15}+
 x_1x_2^{14}x_3^{14}x_4^{13}x_5^{15}+
 x_1x_2^{14}x_3^{20}x_4^{7}x_5^{15}+
 x_1x_2^{20}x_3^{15}x_4^{6}x_5^{15},\\
\medskip
f_8&=  x_1x_2^{10}x_3^{12}x_4^{15}x_5^{15}+
 x_1x_2^{10}x_3^{14}x_4^{13}x_5^{15}+
 x_1x_2^{12}x_3^{15}x_4^{10}x_5^{15}+
 x_1x_2^{14}x_3^{15}x_4^{8}x_5^{15}.
\end{array}$$
The above relations show that $x$ is strictly inadmissible. By a similar computation, we obtain
$$ \begin{array}{ll}
y &=  x_1^{4}x_2^{11}x_3^{12}x_4^{19}x_5^{15}+
x_1^{4}x_2^{11}x_3^{19}x_4^{12}x_5^{15}+
 x_1^{4}x_2^{20}x_3^{11}x_4^{11}x_5^{15}+
 x_1^{4}x_2^{24}x_3^{7}x_4^{11}x_5^{15}\\
\medskip
 &\quad + x_1^{5}x_2^{3}x_3^{12}x_4^{26}x_5^{15}+
 x_1^{5}x_2^{3}x_3^{26}x_4^{12}x_5^{15}+
 x_1^{5}x_2^{7}x_3^{8}x_4^{14}x_5^{27}+
 x_1^{5}x_2^{7}x_3^{8}x_4^{26}x_5^{15}\\
\medskip
&\quad + x_1^{5}x_2^{7}x_3^{10}x_4^{12}x_5^{27}+
 x_1^{5}x_2^{7}x_3^{10}x_4^{24}x_5^{15}+
 x_1^{5}x_2^{10}x_3^{12}x_4^{19}x_5^{15}+
 x_1^{5}x_2^{10}x_3^{19}x_4^{12}x_5^{15}\\
\medskip
 &\quad +x_1^{5}x_2^{11}x_3^{5}x_4^{10}x_5^{30}+
 x_1^{5}x_2^{11}x_3^{6}x_4^{12}x_5^{27}+
 x_1^{5}x_2^{11}x_3^{8}x_4^{14}x_5^{23}+
 x_1^{5}x_2^{11}x_3^{8}x_4^{22}x_5^{15}\\
\medskip
 &\quad + x_1^{5}x_2^{11}x_3^{9}x_4^{6}x_5^{30}+
 x_1^{5}x_2^{11}x_3^{10}x_4^{20}x_5^{15}+
 x_1^{5}x_2^{11}x_3^{16}x_4^{14}x_5^{15}+
 x_1^{5}x_2^{11}x_3^{18}x_4^{12}x_5^{15}\\
\medskip
 &\quad + x_1^{5}x_2^{18}x_3^{11}x_4^{12}x_5^{15}+
 x_1^{7}x_2^{3}x_3^{12}x_4^{24}x_5^{15}+
 x_1^{7}x_2^{3}x_3^{24}x_4^{12}x_5^{15}+
 x_1^{7}x_2^{5}x_3^{9}x_4^{10}x_5^{30}\\
\medskip
 &\quad + x_1^{7}x_2^{5}x_3^{10}x_4^{12}x_5^{27}+
 x_1^{7}x_2^{7}x_3^{8}x_4^{12}x_5^{27}+
 x_1^{7}x_2^{7}x_3^{8}x_4^{14}x_5^{25}+
 x_1^{7}x_2^{7}x_3^{8}x_4^{24}x_5^{15} \\
\medskip
 &\quad + Sq^1(g_1)  + Sq^2(g_2) + Sq^4(g_4)  + Sq^8(g_8) \ \ {\rm mod}(P_5^-(\omega_{(5, 4)})),  \ \mbox{where}\\
\medskip
g_1 & =  x_1^{7}x_2^{5}x_3^{12}x_4^{21}x_5^{15}+
 x_1^{7}x_2^{5}x_3^{21}x_4^{12}x_5^{15}+
 x_1^{7}x_2^{7}x_3^{5}x_4^{12}x_5^{29}+
 x_1^{7}x_2^{7}x_3^{12}x_4^{19}x_5^{15}\\
\medskip
  &\quad + x_1^{7}x_2^{7}x_3^{19}x_4^{12}x_5^{15}+
 x_1^{7}x_2^{9}x_3^{8}x_4^{13}x_5^{23}+
 x_1^{7}x_2^{9}x_3^{8}x_4^{21}x_5^{15}+
 x_1^{7}x_2^{9}x_3^{16}x_4^{13}x_5^{15}\\
\medskip
  &\quad + x_1^{7}x_2^{17}x_3^{9}x_4^{12}x_5^{15}+ x_1^{7}x_2^{20}x_3^{7}x_4^{11}x_5^{15},\\
\medskip
g_2 &=  x_1^{7}x_2^{3}x_3^{5}x_4^{6}x_5^{38}+
 x_1^{7}x_2^{3}x_3^{6}x_4^{20}x_5^{23}+
 x_1^{7}x_2^{3}x_3^{12}x_4^{22}x_5^{15}+
 x_1^{7}x_2^{3}x_3^{22}x_4^{12}x_5^{15}\\
\medskip
  &\quad + x_1^{7}x_2^{6}x_3^{12}x_4^{19}x_5^{15}+
 x_1^{7}x_2^{6}x_3^{19}x_4^{12}x_5^{15}+
 x_1^{7}x_2^{7}x_3^{5}x_4^{10}x_5^{30}+
 x_1^{7}x_2^{7}x_3^{6}x_4^{12}x_5^{27}\\
\medskip
 &\quad +  x_1^{7}x_2^{7}x_3^{8}x_4^{14}x_5^{23}+
 x_1^{7}x_2^{7}x_3^{8}x_4^{22}x_5^{15}+
 x_1^{7}x_2^{7}x_3^{9}x_4^{6}x_5^{30}+
 x_1^{7}x_2^{7}x_3^{10}x_4^{20}x_5^{15}\\
\medskip
 &\quad +  x_1^{7}x_2^{7}x_3^{16}x_4^{14}x_5^{15}+
 x_1^{7}x_2^{7}x_3^{18}x_4^{12}x_5^{15}+
 x_1^{7}x_2^{18}x_3^{7}x_4^{12}x_5^{15}+
 x_1^{7}x_2^{18}x_3^{8}x_4^{11}x_5^{15},\\
\medskip
g_4&=  x_1^{4}x_2^{7}x_3^{12}x_4^{19}x_5^{15}+
 x_1^{4}x_2^{7}x_3^{19}x_4^{12}x_5^{15}+
 x_1^{4}x_2^{20}x_3^{7}x_4^{11}x_5^{15}+
 x_1^{5}x_2^{3}x_3^{12}x_4^{22}x_5^{15}\\
\medskip
 &\quad + x_1^{5}x_2^{3}x_3^{22}x_4^{12}x_5^{15}+
 x_1^{5}x_2^{6}x_3^{12}x_4^{19}x_5^{15}+
 x_1^{5}x_2^{6}x_3^{19}x_4^{12}x_5^{15}+
 x_1^{5}x_2^{7}x_3^{5}x_4^{10}x_5^{30}\\
\medskip
 &\quad + x_1^{5}x_2^{7}x_3^{6}x_4^{12}x_5^{27}+
 x_1^{5}x_2^{7}x_3^{8}x_4^{14}x_5^{23}+
 x_1^{5}x_2^{7}x_3^{8}x_4^{22}x_5^{15}+
 x_1^{5}x_2^{7}x_3^{9}x_4^{6}x_5^{30}\\
\medskip
&\quad + x_1^{5}x_2^{7}x_3^{10}x_4^{20}x_5^{15}+
 x_1^{5}x_2^{7}x_3^{16}x_4^{14}x_5^{15}+
 x_1^{5}x_2^{7}x_3^{18}x_4^{12}x_5^{15}+
 x_1^{5}x_2^{18}x_3^{7}x_4^{12}x_5^{15}\\
\medskip
 &\quad + x_1^{11}x_2^{5}x_3^{5}x_4^{6}x_5^{30}+ x_1^{11}x_2^{5}x_3^{6}x_4^{12}x_5^{23},\\ 
g_8 &=  x_1^{7}x_2^{5}x_3^{5}x_4^{6}x_5^{30}+
x_1^{7}x_2^{5}x_3^{6}x_4^{12}x_5^{23}+
 x_1^{7}x_2^{9}x_3^{10}x_4^{12}x_5^{15}+
 x_1^{7}x_2^{10}x_3^{8}x_4^{13}x_5^{15}.
 \end{array}$$
By Definition \ref{dniad}, we see that $y$ is strictly inadmissible. The lemma follows.
\end{proof}  

\newpage
\begin{lema}\label{Bdtq-4}
For any $t\geqslant 4,$ the following monomials are strictly inadmissible:

\begin{center}
\begin{tabular}{llr}
    $ x_1^{3}x_2^{2^t-4}x_3^{3}x_4^{2^t-3}x_5^{2^{t+1}-2}$, & $ x_1^{3}x_2^{2^t-3}x_3^{2^{t+1}-2}x_4^{3}x_5^{2^t-4}$, & \multicolumn{1}{l}{$ x_1^{3}x_2^{2^t-4}x_3^{3}x_4^{2^{t+1}-3}x_5^{2^t-2}$,} \\
    $ x_1^{3}x_2^{2^{t+1}-3}x_3^{2^t-2}x_4^{3}x_5^{2^t-4}$, & $ x_1^{3}x_2^{2^{t+1}-4}x_3^{3}x_4^{2^t-3}x_5^{2^t-2}$, & \multicolumn{1}{l}{$ x_1^{3}x_2^{2^t-3}x_3^{2^t-2}x_4^{3}x_5^{2^{t+1}-4}$,} \\
    $ x_1^{3}x_2^{5}x_3^{2^{t+1}-2}x_4^{2^t-5}x_5^{2^t-4}$, & $ x_1^{3}x_2^{5}x_3^{2^t-2}x_4^{2^t-5}x_5^{2^{t+1}-4}$, & \multicolumn{1}{l}{$ x_1^{3}x_2^{5}x_3^{2^t-2}x_4^{2^{t+1}-5}x_5^{2^t-4}$,} \\
    $ x_1^{3}x_2^{5}x_3^{2^t-2}x_4^{2^{t+1}-7}x_5^{2^t-2}$, & $ x_1^{3}x_2^{5}x_3^{2^{t+1}-7}x_4^{2^t-2}x_5^{2^t-2}$, & \multicolumn{1}{l}{$ x_1^{3}x_2^{2^{t+1}-3}x_3^{6}x_4^{2^t-5}x_5^{2^t-4}$,} \\
    $ x_1^{3}x_2^{2^t-3}x_3^{6}x_4^{2^t-5}x_5^{2^{t+1}-4}$, & $ x_1^{3}x_2^{2^t-3}x_3^{6}x_4^{2^{t+1}-5}x_5^{2^t-4}$, & \multicolumn{1}{l}{$ x_1^{3}x_2^{2^{t+1}-3}x_3^{7}x_4^{2^t-6}x_5^{2^t-4}$,} \\
    $ x_1^{3}x_2^{2^t-3}x_3^{7}x_4^{2^t-6}x_5^{2^{t+1}-4}$, & $ x_1^{3}x_2^{7}x_3^{2^{t+1}-5}x_4^{2^t-4}x_5^{2^t-4}$, & \multicolumn{1}{l}{$ x_1^{7}x_2^{3}x_3^{2^{t+1}-5}x_4^{2^t-4}x_5^{2^t-4}$,} \\
    $ x_1^{7}x_2^{2^{t+1}-5}x_3^{3}x_4^{2^t-4}x_5^{2^t-4}$, & $ x_1^{3}x_2^{2^t-3}x_3^{7}x_4^{2^{t+1}-6}x_5^{2^t-4}$, & \multicolumn{1}{l}{$ x_1^{3}x_2^{7}x_3^{2^{t+1}-7}x_4^{2^t-2}x_5^{2^t-4}$,} \\
    $ x_1^{7}x_2^{3}x_3^{2^{t+1}-7}x_4^{2^t-2}x_5^{2^t-4}$, & $ x_1^{3}x_2^{13}x_3^{2^{t+1}-9}x_4^{2^t-6}x_5^{2^t-4}$, & \multicolumn{1}{l}{$ x_1^{3}x_2^{13}x_3^{2^{t+1}-10}x_4^{2^t-5}x_5^{2^t-4}$,} \\
    $ x_1^{7}x_2^{2^t-5}x_3^{5}x_4^{2^t-8}x_5^{2^{t+1}-2}$, & $ x_1^{7}x_2^{2^t-5}x_3^{5}x_4^{2^{t+1}-2}x_5^{2^t-8}$, & \multicolumn{1}{l}{$ x_1^{7}x_2^{2^{t+1}-5}x_3^{5}x_4^{2^t-8}x_5^{2^t-2}$,} \\
    $ x_1^{7}x_2^{2^{t+1}-5}x_3^{5}x_4^{2^t-2}x_5^{2^t-8}$, & $ x_1^{7}x_2^{2^t-5}x_3^{5}x_4^{2^t-2}x_5^{2^{t+1}-8}$, & \multicolumn{1}{l}{$ x_1^{7}x_2^{2^t-5}x_3^{5}x_4^{2^{t+1}-8}x_5^{2^t-2}$,} \\
    $ x_1^{7}x_2^{7}x_3^{2^t-7}x_4^{2^t-6}x_5^{2^{t+1}-4}$, & $ x_1^{7}x_2^{9}x_3^{2^t-9}x_4^{2^t-6}x_5^{2^{t+1}-4}$, & \multicolumn{1}{l}{$ x_1^{7}x_2^{7}x_3^{2^t-7}x_4^{2^{t+1}-6}x_5^{2^t-4}$,} \\
    $ x_1^{7}x_2^{9}x_3^{2^t-9}x_4^{2^{t+1}-6}x_5^{2^t-4}$, & $ x_1^{7}x_2^{7}x_3^{2^{t+1}-7}x_4^{2^t-6}x_5^{2^t-4}$, & \multicolumn{1}{l}{$ x_1^{7}x_2^{11}x_3^{2^{t+1}-11}x_4^{2^t-8}x_5^{2^t-2}$,} \\
    $ x_1^{7}x_2^{11}x_3^{2^{t+1}-11}x_4^{2^t-2}x_5^{2^t-8}$, & $ x_1^{7}x_2^{9}x_3^{2^{t+1}-9}x_4^{2^t-6}x_5^{2^t-4}$. &  
    \end{tabular}%
  \end{center}
  \end{lema}

\begin{proof}
We prove the lemma for $u = x_1^{7}x_2^{9}x_3^{2^{t+1}-9}x_4^{2^t-6}x_5^{2^t-4}.$ The others can be proven by a similar computation. By a direct computation using the Cartan formula, we have
$$ \begin{array}{ll}
u&=  x_1^{4}x_2^{7}x_3^{2^{t+1}-5}x_4^{2^t-5}x_5^{2^t-4} +  x_1^{4}x_2^{11}x_3^{2^{t+1}-9}x_4^{2^t-5}x_5^{2^t-4}+
 x_1^{5}x_2^{7}x_3^{2^{t+1}-5}x_4^{2^t-6}x_5^{2^t-4}\\
\medskip
&\quad + x_1^{5}x_2^{11}x_3^{2^{t+1}-9}x_4^{2^t-6}x_5^{2^t-4} + x_1^{7}x_2^{7}x_3^{2^{t+1}-8}x_4^{2^t-5}x_5^{2^t-4}+
 x_1^{7}x_2^{7}x_3^{2^{t+1}-7}x_4^{2^t-6}x_5^{2^t-4}\\
\medskip
&\quad + x_1^{7}x_2^{8}x_3^{2^{t+1}-9}x_4^{2^t-5}x_5^{2^t-4} + Sq^1( x_1^{7}x_2^{7}x_3^{2^{t+1}-9}x_4^{2^t-5}x_5^{2^t-4}) + Sq^2( x_1^{7}x_2^{7}x_3^{2^{t+1}-9}x_4^{2^t-6}x_5^{2^t-4})\\
\medskip
 &\quad + Sq^4( x_1^{4}x_2^{7}x_3^{2^{t+1}-9}x_4^{2^t-5}x_5^{2^t-4} +  x_1^{5}x_2^{7}x_3^{2^{t+1}-9}x_4^{2^t-6}x_5^{2^t-4})  \ \ {\rm mod}(P_5^-(\omega_{(5, t)})).
\end{array}$$
This equality shows that $u$ is strictly inadmissible. The lemma is proved.
\end{proof}

The proof of the following lemma is analogous to the proof of Lemma \ref{Bdtq-4}.
\begin{lema}\label{Bdtq-5}
For any $t\geqslant 5,$ the following monomials are strictly inadmissible:

\begin{center}
   \begin{tabular}{lll}
    $ x_1^{3}x_2^{13}x_3^{2^t-10}x_4^{2^t-5}x_5^{2^{t+1}-4}$, & $ x_1^{3}x_2^{13}x_3^{2^t-10}x_4^{2^{t+1}-5}x_5^{2^t-4}$, & $ x_1^{3}x_2^{13}x_3^{2^t-9}x_4^{2^t-6}x_5^{2^{t+1}-4}$, \\
    $ x_1^{3}x_2^{13}x_3^{2^t-9}x_4^{2^{t+1}-6}x_5^{2^t-4}$, & $ x_1^{3}x_2^{15}x_3^{2^t-11}x_4^{2^t-6}x_5^{2^{t+1}-4}$, & $ x_1^{15}x_2^{3}x_3^{2^t-11}x_4^{2^t-6}x_5^{2^{t+1}-4}$, \\
    $ x_1^{3}x_2^{15}x_3^{2^t-11}x_4^{2^{t+1}-6}x_5^{2^t-4}$, & $ x_1^{15}x_2^{3}x_3^{2^t-11}x_4^{2^{t+1}-6}x_5^{2^t-4}$, & $ x_1^{7}x_2^{11}x_3^{2^t-11}x_4^{2^t-8}x_5^{2^{t+1}-2}$, \\
    $ x_1^{7}x_2^{11}x_3^{2^t-11}x_4^{2^{t+1}-2}x_5^{2^t-8}$, & $ x_1^{7}x_2^{11}x_3^{2^t-11}x_4^{2^t-2}x_5^{2^{t+1}-8}$, & $ x_1^{7}x_2^{11}x_3^{2^t-11}x_4^{2^{t+1}-8}x_5^{2^t-2}$.
    \end{tabular}%
\end{center}
  \end{lema}

\begin{proof}[{\it Proof of Proposition \ref{mdtq}}]
Let $x$ be an admissible monomial in  $(P_5^+)_{3(2^t-1)  +2^t}$ such that $\omega(x) = \omega_{(5, t)}$ with $t\geqslant 4.$ By induction on $t,$ we see that $x\neq b_{3(2^t-1)+2^t, k}$ for $k  = 1, 2, \ldots, 270,$ then there is a monomial $z$, which is given in one of Lemmas \ref{Bdtq-1}, \ref{Bdtq-2}, \ref{Bdtq-3}, \ref{Bdtq-4}, and \ref{Bdtq-5} such that $x = zy^{2^\alpha}$ for some monomial $y$ and positive integer $\alpha.$ By Theorem \ref{dlKS}, $x$ is inadmissible. Hence $\mathscr B_5^+(\omega_{(5, t)})\subseteq \{b_{3(2^t-1)+2^t, k}:\ 1\leqslant k\leqslant 270\}$ for every $t\geqslant 4.$

Now we prove that the classes $[b_{3(2^t-1)+2^t, k}]_{\omega_{(5,t)}}\neq [0]_{\omega_{(5,t)}}$ with $1\leqslant k\leqslant 270,$ and $t\geqslant 4.$ Suppose there is a linear relation 
$$ \mathcal S = \sum\limits_{1\leqslant k\leqslant 270}\gamma_kb_{3(2^t-1)+2^t, k}\equiv_{\omega_{(5,t)}} 0,$$
where $\gamma_k\in\mathbb F_2$ for all $k,\ 1\leqslant k\leqslant 270.$   

For $t\geqslant 4,$ according to Sum \cite{N.S4}, $\mathscr B_4^+(3(2^t-1)+2^t)$ is the set consisting of $33$ monomials, namely:
\medskip

\begin{tabular}{cllrr}
    $v_{1}.\ x_1x_2^{2^{t}-2}x_3^{2^{t}-1}x_4^{2^{t+1}-1}$, & $v_{2}.\ x_1x_2^{2^{t}-2}x_3^{2^{t+1}-1}x_4^{2^{t}-1}$, & $v_{3}.\ x_1x_2^{2^{t}-1}x_3^{2^{t}-2}x_4^{2^{t+1}-1}$, & \\
$v_{4}.\ x_1x_2^{2^{t}-1}x_3^{2^{t+1}-1}x_4^{2^{t}-2}$, & $v_{5}.\ x_1x_2^{2^{t+1}-1}x_3^{2^{t}-2}x_4^{2^{t}-1}$,  & $v_{6}.\ x_1x_2^{2^{t+1}-1}x_3^{2^{t}-1}x_4^{2^{t}-2}$,  &\\
$v_{7}.\ x_1^{2^{t}-1}x_2x_3^{2^{t}-2}x_4^{2^{t+1}-1}$, & $v_{8}.\ x_1^{2^{t}-1}x_2x_3^{2^{t+1}-1}x_4^{2^{t}-2}$,  & $v_{9}.\ x_1^{2^{t}-1}x_2^{2^{t+1}-1}x_3x_4^{2^{t}-2}$,  &\\
$v_{10}.\ x_1^{2^{t+1}-1}x_2x_3^{2^{t}-2}x_4^{2^{t}-1}$, & $v_{11}.\ x_1^{2^{t+1}-1}x_2x_3^{2^{t}-1}x_4^{2^{t}-2}$,  & $v_{12}.\ x_1^{2^{t+1}-1}x_2^{2^{t}-1}x_3x_4^{2^{t}-2}$,  &\\
$v_{13}.\ x_1x_2^{2^{t}-1}x_3^{2^{t}-1}x_4^{2^{t+1}-2}$, & $v_{14}.\ x_1x_2^{2^{t}-1}x_3^{2^{t+1}-2}x_4^{2^{t}-1}$,  & $v_{15}.\ x_1x_2^{2^{t+1}-2}x_3^{2^{t}-1}x_4^{2^{t}-1}$,  &\\
$v_{16}.\ x_1^{2^t-1}x_2x_3^{2^{t}-1}x_4^{2^{t+1}-2}$, & $v_{17}.\ x_1^{2^t-1}x_2x_3^{2^{t+1}-2}x_4^{2^{t}-1}$,  & $v_{18}.\ x_1^{2^{t}-1}x_2^{2^t-1}x_3x_4^{2^{t+1}-2}$,  &\\
$v_{19}.\ x_1^{3}x_2^{2^t-3}x_3^{2^{t}-2}x_4^{2^{t+1}-1}$, & $v_{20}.\ x_1^{3}x_2^{2^t-3}x_3^{2^{t+1}-1}x_4^{2^{t}-2}$,  & $v_{21}.\ x_1^{3}x_2^{2^{t+1}-1}x_3^{2^t-3}x_4^{2^{t}-2}$,  &\\
$v_{22}.\ x_1^{2^{t+1}-1}x_2^{3}x_3^{2^{t}-3}x_4^{2^{t}-2}$, & $v_{23}.\ x_1^{3}x_2^{2^t-3}x_3^{2^{t}-1}x_4^{2^{t+1}-2}$,  & $v_{24}.\ x_1^{3}x_2^{2^{t}-3}x_3^{2^{t+1}-2}x_4^{2^{t}-1}$,  &\\
$v_{25}.\ x_1^{3}x_2^{2^t-1}x_3^{2^{t}-3}x_4^{2^{t+1}-2}$, & $v_{26}.\ x_1^{2^t-1}x_2^{3}x_3^{2^{t}-3}x_4^{2^{t+1}-2}$,  & $v_{27}.\ x_1^{3}x_2^{2^{t}-1}x_3^{2^{t+1}-3}x_4^{2^{t}-2}$,  &\\
$v_{28}.\ x_1^{3}x_2^{2^{t+1}-3}x_3^{2^{t}-2}x_4^{2^{t}-1}$, & $v_{29}.\ x_1^{3}x_2^{2^{t+1}-3}x_3^{2^{t}-1}x_4^{2^{t}-2}$,  & $v_{30}.\ x_1^{2^t-1}x_2^{3}x_3^{2^{t+1}-3}x_4^{2^{t}-2}$,  &\\
$v_{31}.\ x_1^{7}x_2^{2^{t+1}-5}x_3^{2^{t}-3}x_4^{2^{t}-2}$, & $v_{32}.\ x_1^{7}x_2^{2^{t}-5}x_3^{2^{t}-3}x_4^{2^{t+1}-2}$,  & $v_{33}.\ x_1^{7}x_2^{2^t-5}x_3^{2^{t+1}-3}x_4^{2^{t}-2}$.  &
\end{tabular}

\medskip

By a direct computation using Theorem \ref{dlsig}, we express $p_{(i; I)}(\mathcal S),$ in terms of $v_j,\ 1\leqslant j\leqslant 33.$ Computing directly from the relations $$p_{(i;I)}(\mathcal S) \equiv_{\omega_{(5,t)}} 0,\ \forall (i; I)\in \mathcal N_5,\ \ell(I) = 1,$$ we obtain $\gamma_k = 0,\ k = 1, 2, \ldots, 270.$ The proposition follows.
\end{proof}

We knew that
$$ \begin{array}{ll}
(QP_5)_{3(2^t-1) + 2^t}& = {\rm Ker}((\widetilde {Sq^0_*})_{(5, 3(2^t-1) + 2^t)})\,\bigoplus\,  (QP_5)_{2^{t+1} - 4},\\
{\rm Ker}((\widetilde {Sq^0_*})_{(5, 3(2^t-1) + 2^t)})  & =  (QP_5^0)_{3(2^t-1) + 2^t}\,\bigoplus\, QP_5^+(\omega_{(5, t)}),
\end{array}$$
and $\dim(QP_5^0)_{3(2^t-1) + 2^t} = 195$ for any $t\geqslant 3.$ So, from Theorem \ref{dldb} and Proposition \ref{mdtq}, we get
$$ \dim (QP_5)_{3(2^t-1) + 2^t} = \left\{\begin{array}{ll}
945&\mbox{if\ } t = 4,\\
1115&\mbox{if\ } t = 5,\\
1116&\mbox{if\ } t \geqslant 6.
\end{array}\right.$$

The proof of Theorem \ref{dlc-1} is completed.

\section{Proof of Theorem \ref{dlc-2}}\label{s4}

In this section, we prove Theorem \ref{dlc-2} by using the results in Section \ref{s3}. Firstly, we need some notations and definitions for the proof of the theorem. 

\medskip

For $1\leqslant i\leqslant s,$ define the $\mathscr{A}$-homomorphism $\tau_i: P_s\longrightarrow P_s,$ which is determined by  $\tau_i(x_i) = x_{i+1},\;\tau_i(x_{i+1}) = x_i,\;\tau_i(x_j) = x_j$ for $j\neq i, i +1,\; 1\leqslant i < s$ and $\tau_s(x_1) = x_1 + x_2,\; \tau_s(x_j) = x_j$ for $j  > 1.$ Observe that the general linear group $GL_s$ is generated by the matrices associated with by $\tau_i,\;0\leqslant i\leqslant s$ and the symmetric group $\Sigma_s\subset GL_s$  is generated by the ones associated with $\tau_i,\;1\leqslant i <s.$ Hence, a class $[f]_{\omega}$ represented by a homogeneous polynomial $f\in P_s$ is an $GL_s$-invariant if and only if $\tau_i(f)\equiv_{\omega} f$ for $1\leqslant i\leqslant s.$ $[f]_{\omega}$ is an $\Sigma_s$-invariant if and only if $\tau_i(f)\equiv_{\omega} f$ for $1\leqslant i < s.$

For the weight vector $\omega_{(5,t)} = (3,3,\ldots, 3, 1)$ ($t$ times of $3$), we have $\deg \omega_{(5,t)} = 3(2^t-1) + 2^t.$ For any monomials $g_1,\, g_2,\, \ldots,\,  g_m$ in $P_5(\omega_{(5,t)} )$ and for a subgroup $ G\subset GL_5,$ we denote
$G(g_1,\, \ldots,\, g_m)$ the $G$-submodule of $QP_5(\omega_{(5,t)})$ generated by the set $\{[g_i]_{\omega_{(5,t)}}\, :\, 1\leqslant i\leqslant m\}.$ We see that $\omega_{(5,t)}$ is the weight vector of the mimimal spike $x_1^{2^{t+1}-1}x_2^{2^t-1}x_3^{2^t-1}\in (P_5)_{3(2^t-1) + 2^t};$ hence $[g_i]_{\omega_{(5,t)}} = [g_i]$ for all $i,\, 1\leqslant i\leqslant m.$ 

\medskip

Recall that Kameko's map $\widetilde {Sq^0_*} = (\widetilde {Sq^0_*})_{(5, 3(2^t-1) + 2^t)}: (QP_5)_{3(2^t-1) + 2^t} \longrightarrow (QP_5)_{2^{t+1} - 4}$ is an epimorphism of $GL_5$-modules. Then, as mentioned in Section \ref{s1},  we need only to determine ${\rm Ker}((\widetilde {Sq^0_*})_{(5, 3(2^t-1) + 2^t)})^{GL_5}.$ 

\medskip

From the results in Section \ref{s3},  we get
$$ \dim({\rm Ker}((\widetilde {Sq^0_*})_{(5, 3(2^t-1) + 2^t)})) =\dim(QP_5(\omega_{(5,t)})).$$ Note that $QP_5(\omega_{(5,t)})= QP_5^0(\omega_{(5,t)})\, \oplus\, QP_5^+(\omega_{(5,t)})$ and $$\dim({\rm Ker}((\widetilde {Sq^0_*})_{(5, 3(2^t-1) + 2^t)})^{GL_5}) \leq \dim(QP_5(\omega_{(5,t)})^{GL_5}).$$

Now, for $t = 1,$ $3(2^t-1) + 2^t = 5.$ From a result of Sum \cite{N.S8, N.S6}, we have $QP_5(\omega_{(5,1)})^{GL_5} = 0.$ 

We now compute $QP_5(\omega_{(5,t)})^{GL_5}$ for $t\geqslant 2.$
 
\subsection{Computation of  $QP_5(\omega_{(5,t)})^{\Sigma_5}$}

As it is known, $ \mathscr B_5^0(3(2^t-1) + 2^t) = \mathscr B_5^0(\omega_{(5, t)})$ is the set consisting of the admissible monomials $q_{t, k}$ (see Subsection \ref{dtcndtq-1}). By a simple computation, we see that the following subspaces are $\Sigma_5$-submodules of $QP_5^0(\omega_{(5,t)})$:

For $t\geqslant 2,$
$$ \begin{array}{ll}
\medskip
\Sigma_5(q_{t, 1}) &= \langle [q_{t, k}]\, :\, 1\leqslant k\leqslant 30 \rangle,\\
\Sigma_5(q_{t, 31}) &=  \langle [q_{t, k}]\, :\, 31\leqslant k\leqslant 90 \rangle.
\end{array}$$

For $t = 2,$
$$ \mathcal M_1 =  \langle [q_{2, k}]\, :\, 91\leqslant k\leqslant 145 \rangle.$$

For $t\geqslant 3,$
$$ \Sigma_5(q_{t, 91}) =  \langle [q_{t, k}]\, :\, 91\leqslant k\leqslant 110 \rangle.$$

For $t = 3,$
$$ \mathcal M_2 =  \langle [q_{3, k}]\, :\, 111\leqslant k\leqslant 195 \rangle.$$

For $t\geqslant 4,$
$$ \mathcal  M^*_{t, 1} =  \langle [q_{t, k}]\, :\, 111\leqslant k\leqslant 195 \rangle.$$

\begin{lema}\label{bdtq-1} 
Let $t$ be an integer. Then, we obtain the following:
\medskip

For $t\geqslant 2,$
$$ \begin{array}{ll}
\medskip
\Sigma_5(q_{t, 1})^{\Sigma_5} &= \langle [p_{t, 1}:= q_{t,1} + q_{t, 2} +  \cdots +q_{t,30}] \rangle,\\
\medskip
\Sigma_5(q_{t, 31})^{\Sigma_5} &= \langle [p_{t, 2}:= q_{t, 31} + q_{t, 32} + \cdots +q_{t, 90}] \rangle.\\
\end{array}$$

For $t\geqslant 3,$
$$ \Sigma_5(q_{t, 91})^{\Sigma_5} = \langle [p_{t, 3}:= q_{t, 91} + q_{t, 92} +  \cdots +q_{t, 110}] \rangle.$$
\end{lema}

\begin{proof}[{\it Outline of the proof}]
The set $B := \{[q_{t, k}]\, :\, 1\leqslant k\leqslant 30\}$ is a basis of $\Sigma_5(q_{t, 1})$ for $t\geqslant 2.$ The action of $\Sigma_5$ on $QP_5$ induces the one of it on $B.$ Furthermore, this action is transitive. Hence, if $g \equiv \sum_{k = 1}^{30}\gamma_kq_{t, k}$ with $\gamma_t\in\mathbb F_2$ and $[g]\in \Sigma_5(q_{t, 1})^{\Sigma_5},$ then the relations $\tau_i(g)\equiv g,\, i  = 1,\, 2,\, 3,\, 4,$ imply $\gamma_k = \gamma_1,\, \forall k,\, 1\leqslant k\leqslant 30.$ Hence, $\Sigma_5(q_{t, 1})^{\Sigma_5} = \langle [p_{t, 1}] \rangle$ with $p_{t, 1} = \sum_{k = 1}^{30}q_{t, k}.$  

\medskip

By a similar computation, we obtain $\Sigma_5(q_{t, 31})^{\Sigma_5} =\langle [p_{t, 2}] \rangle$ with $p_{t, 2} = \sum_{k = 31}^{90}q_{t, k},$ and $\Sigma_5(q_{t, 91})^{\Sigma_5} = \langle [p_{t, 3}] \rangle$ with $p_{t, 3} = \sum_{k = 91}^{110}q_{t, k}.$ The lemma follows.
\end{proof}

\begin{lema}\label{bdtq-2}
We have $\mathcal M_1^{\Sigma_5} =  0$ and $\dim(\mathcal M_2^{\Sigma_5}) = \dim((\mathcal M^*_{t, 1})^{\Sigma_5}) = 1$ for every $t\geqslant 4.$ 
\end{lema}

\begin{proof}

Using a result in Section \ref{s3}, we  have $\dim(\mathcal M_1) = 55.$ Suppose $h \equiv \sum_{k = 91}^{145}\gamma_kq_{2,k}$ with $\gamma_k\in\mathbb F_2$ and $[h]\in \mathcal M_1^{\Sigma_5}.$ By computing $\tau_j(h) + h$ in  terms of $q_{2, k},\, 91\leqslant k\leqslant 145$ and using the relations $\tau_j(h) + h\equiv 0,\, 1\leqslant j\leqslant 4,$ we obtain $\gamma_k = 0$ for all $k,$ $91\leqslant k\leqslant 145.$ 

\medskip

Now suppose $h \equiv \sum_{k = 111}^{195}\gamma_kq_{3,k}$  and $[h]\in \mathcal M_2^{\Sigma_5}.$ We compute  $\tau_j(h) + h$ in  terms of $q_{3, k},\, 111\leqslant k\leqslant 195.$ By computing directly from the relations $\tau_j(h) + h\equiv 0,\, j  =1, 2, 3, 4,$ one gets $\gamma_k  = 0$ for $111\leqslant k\leqslant 135$ and $\gamma_k = \gamma_{136}$ for $136\leqslant k\leqslant 195.$ This implies $\mathcal M_2^{\Sigma_5} =\langle [p_{3,4}] \rangle$ with $p_{3,4} = \sum_{i = 136}^{195}q_{3,i}.$ 

\medskip

Similarly, we have  $(\mathcal M^*_{t, 1})^{\Sigma_5} = \langle  [q^*_{t, 1}] \rangle$ with $q^*_{t, 1} = \sum_{i = 136}^{195}q_{t,i}$ for all $t\geqslant 4.$ The lemma is proved.
\end{proof}

We now denote by $b_{t, k} = b_{3(2^t-1)+2^t,\, k}$ the admissible monomials in $(P_5^+)_{3(2^t-1)+2^t}$ as given in Subsection \ref{dtcndtq-2}.
\medskip

{\bf The case \boldmath\mbox{$t  = 2$}.} By a direct computation using the results in Section \ref{s3}, we have the direct summand decomposition of the $\Sigma_5$-modules:
$$ QP_5^+(\omega_{(5,2)}) = \Sigma_5(b_{2, 1})\, \bigoplus\, \mathcal M_3,$$
where $\Sigma_5(b_{2, 1})\ = \langle[b_{2, k}]:\, 1\leqslant k\leqslant 10\rangle$ and $\mathcal M_3 = \langle[b_{2, k}]:\, 11\leqslant k\leqslant 60\rangle.$ 
\medskip

The following lemma is proved by a direct computation.
\begin{lema}\label{bdct-1}
$ \Sigma_5(b_{2, 1})^{\Sigma_5} = 0$ and $\mathcal M_3^{\Sigma_5} =  \langle[p_{2,4}:= b_{2, 31} + b_{2, 32} + \cdots + b_{2, 60}]\rangle.$
\end{lema}

From Lemmas \ref{bdtq-1} and \ref{bdct-1}, we obtain the following.

\begin{propo}\label{mdct-1}
$QP_5(\omega_{(5, 2)})^{\Sigma_5} = \langle [p_{2,1}], [p_{2,2}], [p_{2,4}]\rangle.$
\end{propo}

{\bf The case \boldmath\mbox{$t  = 3$}.}  By a simple computation, we see that
$$ \begin{array}{ll}
\Sigma_5(b_{3, 1}) &= \langle[b_{3, k}]\, :\, 1\leqslant k\leqslant 45\rangle,\\
\Sigma_5(b_{3, 46}) &= \langle [b_{3, k}]\, :\, 46\leqslant k\leqslant 65\rangle,\\
\end{array}$$
and $\mathcal  M_4 = \langle[b_{3, k}]\, :\, 66\leqslant k\leqslant 260\rangle$ are  $\Sigma_5$-submodules of $QP_5^+(\omega_{(5, 3)}).$ Hence, we have a direct summand decomposition of $\Sigma_5$-modules:
$$ QP_5^+(\omega_{(5,3)}) = \Sigma_5(b_{t, 1})\, \oplus\, \Sigma_5(b_{t, 46})   \, \oplus\,  \mathcal  M_4.$$

\begin{lema}\label{bdct-2}
We have $ \dim(\Sigma_5(b_{3,1})^{\Sigma_5}) = \dim (\Sigma_5(b_{3, 46})^{\Sigma_5}) = 1$ and $\mathcal M_4^{\Sigma_5} = 0.$
\end{lema}

\begin{proof}
We have $\dim \Sigma_5(b_{3, 1}) = 45$ and the set $\{[b_{3, k}]\, :\, 1\leqslant k\leqslant 45\}$ is a basis of $\Sigma_5(b_{3, 1}).$ Suppose $f \equiv \sum_{k = 1}^{45}\gamma_kb_{3, k}$ with $\gamma_k\in\mathbb F_2$ and $[f]\in \Sigma_5(b_{3, 1})^{\Sigma_5}.$  By computing $\tau_j(f) + f$ in terms of $b_{3, k},\ 1\leqslant k\leqslant 45$ and using the relations $\tau_j(f) + f \equiv 0,\ j = 1, 2, 3, 4,$ one gets
$$ \begin{array}{lll}
\gamma_1 &= \gamma_k,&  \text{for }\, 2\leqslant k\leqslant 10,\\
\gamma_k &= 0,& \text{for } 11\leqslant k\leqslant 45.
\end{array}$$
 Hence $\Sigma_5(b_{3, 1})^{\Sigma_5} = \langle [p_{3, 5}]\rangle$ with $p_{3,5} = \sum_{k = 1}^{10}b_{3, k}.$ By a similar computation, we get $\Sigma_5(b_{3, 46})^{\Sigma_5} = \langle [p_{3,6}]\rangle$ with $p_{3,6} = \sum_{k = 46}^{65}b_{3,k}$ and $\mathcal M_4^{\Sigma_5} = 0.$ The lemma follows. 
\end{proof} 

Based on Lemmas \ref{bdtq-1}, \ref{bdtq-2} and \ref{bdct-2},  one gets the following.

\begin{propo}\label{mdct-2}
$QP_5(\omega_{(5, 3)})^{\Sigma_5} = \langle [p_{3,1}], [p_{3, 2}],\ldots, [p_{3,6}]\rangle.$
\end{propo}

{\bf The case \boldmath\mbox{$t  \geqslant 4$}.} We set 
$$ b_{t, 1} = x_1x_2x_3^{2^t-2}x_4^{2^t-2}x_5^{2^{t+1}-1}, \ \ b_{t, 46} = x_1x_2^2x_3^{2^t-4}x_4^{2^t-1}x_5^{2^{t+1}-1}.$$
Computing from the results in Section \ref{s3}, one gets
$$ \dim(\Sigma_5(b_{t, 1})) = 45,\ \ \dim(\Sigma_5(b_{t, 46})) = 20,$$
where 
$$  \begin{array}{ll}
\Sigma_5(b_{t, 1}) &= \langle[b_{t, k}]\, :\, 1\leqslant k\leqslant 45\rangle,\\
\Sigma_5(b_{t, 46}) &= \langle [b_{t, k}]\, :\, 46\leqslant k\leqslant 65\rangle.
\end{array}$$

\begin{lema}\label{bdtq-3}
For any $t\geqslant 4,$ $ \Sigma_5(b_{t, 46})^{\Sigma_5} = \langle [q^*_{t, 2} := b_{t, 46} + b_{t, 47} + \cdots +b_{t, 65}]\rangle.$
\end{lema}

The proof of this lemma is straightforward.

\medskip

For $t\geqslant 4,$ by an easy computation, we have a direct summand decomposition of $\Sigma_5$-modules:
$$ QP_5^+(\omega_{(5,t)}) = \Sigma_5(b_{t, 1})\, \oplus\, \Sigma_5(b_{t, 46})   \, \oplus\,   \mathcal M^*_{t, 2},$$
where $\mathcal  M^*_{t,2}= \langle [b_{t, k}]\, :\, 66\leqslant k\leqslant 270\rangle.$

\begin{lema}\label{bdtq-4}
 For any integer $t\geqslant 4,$ we have
$$ \Sigma_5(b_{t, 1})^{\Sigma_5} = \langle[q^*_{t,3}],  [q^*_{t,4}] \rangle,\ \  (\mathcal  M^*_{t, 2})^{\Sigma_5} = 0,$$
with $ q^*_{t, 3}:= \sum_{k = 1}^{15}b_{t, k},\ \  q^*_{t, 4}:= \sum_{k = 16}^{45}b_{t, k}.$
\end{lema}

\begin{proof}
Suppose $g \equiv \sum_{k = 1}^{45}\gamma_kb_{t, k}$ with $\gamma_k\in\mathbb F_2$ and $[g]\in \Sigma_5(b_{t, 1})^{\Sigma_5}.$  By computing $\tau_i(g) + g$ in terms of $b_{t, k},\ 1\leqslant i\leqslant 45$ and using the relations $\tau_i(g) + g \equiv 0,\ i = 1, 2, 3, 4,$ we get
$$ \begin{array}{lll}
\gamma_1 &= \gamma_k,&  \text{for }\, 2\leqslant k\leqslant 15,\\
\gamma_{16} &= \gamma_k,& \text{for } 17\leqslant k\leqslant 45.
\end{array}$$
This implies $\Sigma_5(b_{t, 1})^{\Sigma_5} = \langle [q^*_{t, 3}],  [q^*_{t, 4}]\rangle$ with $q^*_{t,3} = \sum_{k = 1}^{15}b_{t, k}$ and $q^*_{t,4} = \sum_{k = 16}^{45}b_{t, k}.$ 

\medskip

Now suppose $g \equiv \sum_{k = 66}^{270}\gamma_kb_{t,k}$  and $[g]\in (\mathcal M^*_{t, 2})^{\Sigma_5}.$ We compute  $\tau_j(g) + g$ in  terms of $b_{t, k},\, 66\leqslant k\leqslant 270.$ By computing directly from the relations $\tau_j(g) + g\equiv 0,\, j  =1, 2, 3, 4,$ we obtain $\gamma_k  = 0$ for $66\leqslant k\leqslant 270.$ The lemma is proved.
\end{proof}

Combining this result and Lemmas \ref{bdtq-1}, \ref{bdtq-2}, \ref{bdtq-3}, \ref{bdtq-4}, one gets the following.

\begin{propo}\label{mdct-3}
For any $t\geqslant 4,$ 
$$ QP_5(\omega_{(5, t)})^{\Sigma_5} = \langle [p_{t,1}], [p_{t, 2}], [p_{t, 3}], [q^*_{t, 1}], [q^*_{t, 2}], [q^*_{t, 3}], [q^*_{t, 4}]\rangle.$$
\end{propo}

\subsection{Computation of  $QP_5(\omega_{(5,t)})^{GL_5}$}

Let $g\in (P_5)_{3(2^t-1) + 2^t}$ such that $[g]\in (QP_5)_{3(2^t-1) + 2^t}^{GL_5}.$ Then we have $[g]\in QP_5(\omega_{(5,t)})^{GL_5}.$

\medskip

For $t = 1,$ we have $3(2^t-1) + 2^t = 5$ and $(QP_5)_{5}^{GL_5} = 0$ (see Sum \cite{N.S8, N.S6}). So, the theorem is true for $t = 1.$

\medskip

For $t = 2,$ from Proposition \ref{mdct-1}, we have $g\equiv \gamma_1 p_{2, 1} + \gamma_2p_{2,2} + \gamma_3 p_{2,3}$ with $\gamma_i\in \mathbb F_2,\ i = 1,2, 3.$ By computing $\tau_5(g) + g$ in terms of the admissible monomials, we get
$$ \tau_5(g) + g\equiv  (\gamma_1+ \gamma_2)x_1^3x_2^3x_4^3x_5^4  +  \gamma_2x_1^3x_2x_3^4x_4^3x_5^2 + 
(\gamma_2 + \gamma_3)x_1x_2^3x_3^2x_4x_5^6    + \ \text{other terms} \equiv 0.$$
This relation implies $\gamma_i = 0$ for $1\leqslant i\leqslant 3.$ The theorem is proved for $t = 2.$

\medskip

For $t = 3,$ from Proposition \ref{mdct-2}, we have $g\equiv \sum_{j = 1}^6\beta_jp_{3,j}$ with $\beta_j\in \mathbb F_2, j = 1,2, \ldots, 6.$ A direct computation shows that
$$ \begin{array}{ll}
\medskip
\tau_5(g) + g &\equiv \beta_1x_1^7x_2^{15}x_5^7 + \beta_2x_2x_3^6x_4^7x_5^{15} +  (\beta_2 + \beta_3)x_1x_2^7x_4^6x_5^{15} + (\beta_2 + \beta_4)x_1x_2x_3^{14}x_4^7x_5^{6} \\
\medskip
&\quad + \beta_5x_1^3x_2^{15}x_3^{3}x_4^4x_5^{4}+ \beta_6x_1^7x_2^7x_3^{9}x_4^2x_5^{4} + \mbox{other terms\, }\equiv 0.
\end{array}$$
From this relation, we obtain $\beta_j = 0,\, 1\leqslant j\leqslant 6.$ The theorem holds for  $t = 3.$

\medskip

For $t\geqslant 4,$ by Proposition \ref{mdct-3},  we have
$$ g\equiv \lambda_1p_{t,1} + \lambda_{2}p_{t,2} + \lambda_3p_{t,3} + \sum_{j = 1}^4\lambda_{3+j}q^*_{t, j},$$
where $\lambda_k\in \mathbb F_2,\  1\leqslant k\leqslant 7.$ By computing $\tau_5(g) + g$ in terms of the admissible monomials, we get
  $$ \begin{array}{ll}
\medskip
\tau_5(g) + g &\equiv \lambda_1x_2^{3}x_3^{2^t-3}x_4^{2^t-1}x_5^{2^{t+1}-2} + (\lambda_2 + \lambda_4)x_1^{2^t-1}x_3^{3}x_4^{2^t-3}x_5^{2^{t+1}-2}\\
\medskip
&\quad  + \lambda_3x_1^{3}x_2^{2^{t+1}-3}x_3^{2^t-1}x_4^{2^t-2} + \lambda_4x_1^{7}x_2^{2^{t+1}-5}x_4^{2^t-3}x_5^{2^t-2}\\
\medskip
&\quad + (\lambda_1 + \lambda_2 + \lambda_5)x_1^{3}x_2^{5}x_3^{2^t-6}x_4^{2^{t+1}-3}x_5^{2^t-2} \\
\medskip
&\quad+ \lambda_6x_1^{3}x_2^{5}x_3^{2^{t+1}-6}x_4^{2^t-1}x_5^{2^t-4} + \lambda_7x_1^{7}x_2^{2^t-5}x_3^{5}x_4^{2^t-6}x_5^{2^{t+1}-4}+ \mbox{other terms\, }\equiv 0.
\end{array}$$
The last equality implies $\lambda_k = 0,\ k = 1, 2, \ldots, 7.$ The theorem is completely proved.

\section{Appendix}\label{s5}
In this section, we list all the admissible monomials of degree $d = 3(2^t-1)  +2^t$ in $P_5.$ 

\subsection{The admissible monomials of degree $3(2^t-1)  +2^t$ in $P_5^0$}\label{dtcndtq-1}

Recall that
$$  \mathscr B_5^0(3(2^t-1)  +2^t)  = \mathscr B_5^0(\omega_{(5,t)}) = \Phi^0(\mathscr B_4(3(2^t-1)  +2^t))$$
with $\omega_{(5,t)} = \underset{\mbox{{$t$ times of $3$}}}{\underbrace{(3, 3, \ldots, 3}}, 1).$
By Sum \cite{N.S4}, we have
$$ |\mathscr B_4(3(2^t-1)  +2^t)| = \left\{
\begin{array}{ll}
15 &\text{if}\ t = 1,\\
35 &\text{if}\ t = 2,\\
45 &\text{if}\ t \geqslant 3.
\end{array}\right.$$

\medskip

Set $u_t = |\mathscr B_5^0(3(2^t-1)  +2^t)|$ and $ \mathscr B_5^0(3(2^t-1)  +2^t) = \{q_{t, k}:\ 1\leqslant k\leqslant u_t\}.$
 We have $u_1 = 45,\ u_2 = 145,\ u_t = 195$ for $t\geqslant 3.$

\medskip
The admissible monomials $q_{t, k},\ 1\leqslant k\leqslant u_t,$ are determined as follows:

\medskip

For $t\geqslant 1,$
\medskip

\begin{center}
    \begin{tabular}{lll}
    $1.\ x_3^{2^t-1}x_4^{2^t-1}x_5^{2^{t+1}-1}$ & $2.\ x_3^{2^t-1}x_4^{2^{t+1}-1}x_5^{2^t-1}$ & $3.\ x_3^{2^{t+1}-1}x_4^{2^t-1}x_5^{2^t-1}$ \\
    $4.\ x_2^{2^t-1}x_4^{2^t-1}x_5^{2^{t+1}-1}$ & $5.\ x_2^{2^t-1}x_4^{2^{t+1}-1}x_5^{2^t-1}$ & $6.\ x_2^{2^t-1}x_3^{2^t-1}x_5^{2^{t+1}-1}$ \\
    $7.\ x_2^{2^t-1}x_3^{2^t-1}x_4^{2^{t+1}-1}$ & $8.\ x_2^{2^t-1}x_3^{2^{t+1}-1}x_5^{2^t-1}$ & $9.\ x_2^{2^t-1}x_3^{2^{t+1}-1}x_4^{2^t-1}$ \\
    $10.\ x_2^{2^{t+1}-1}x_4^{2^t-1}x_5^{2^t-1}$ & $11.\ x_2^{2^{t+1}-1}x_3^{2^t-1}x_5^{2^t-1}$ & $12.\ x_2^{2^{t+1}-1}x_3^{2^t-1}x_4^{2^t-1}$ \\
    $13.\ x_1^{2^t-1}x_4^{2^t-1}x_5^{2^{t+1}-1}$ & $14.\ x_1^{2^t-1}x_4^{2^{t+1}-1}x_5^{2^t-1}$ & $15.\ x_1^{2^t-1}x_3^{2^t-1}x_5^{2^{t+1}-1}$ \\
    $16.\ x_1^{2^t-1}x_3^{2^t-1}x_4^{2^{t+1}-1}$ & $17.\ x_1^{2^t-1}x_3^{2^{t+1}-1}x_5^{2^t-1}$ & $18.\ x_1^{2^t-1}x_3^{2^{t+1}-1}x_4^{2^t-1}$ \\
    $19.\ x_1^{2^t-1}x_2^{2^t-1}x_5^{2^{t+1}-1}$ & $20.\ x_1^{2^t-1}x_2^{2^t-1}x_4^{2^{t+1}-1}$ & $21.\ x_1^{2^t-1}x_2^{2^t-1}x_3^{2^{t+1}-1}$ \\
    $22.\ x_1^{2^t-1}x_2^{2^{t+1}-1}x_5^{2^t-1}$ & $23.\ x_1^{2^t-1}x_2^{2^{t+1}-1}x_4^{2^t-1}$ & $24.\ x_1^{2^t-1}x_2^{2^{t+1}-1}x_3^{2^t-1}$ \\
    $25.\ x_1^{2^{t+1}-1}x_4^{2^t-1}x_5^{2^t-1}$ & $26.\ x_1^{2^{t+1}-1}x_3^{2^t-1}x_5^{2^t-1}$ & $27.\ x_1^{2^{t+1}-1}x_3^{2^t-1}x_4^{2^t-1}$ \\
    $28.\ x_1^{2^{t+1}-1}x_2^{2^t-1}x_5^{2^t-1}$ & $29.\ x_1^{2^{t+1}-1}x_2^{2^t-1}x_4^{2^t-1}$ & $30.\ x_1^{2^{t+1}-1}x_2^{2^t-1}x_3^{2^t-1}$. 
    \end{tabular}%
\end{center}

\medskip

For $t  =1,$
\medskip

\begin{center}
    \begin{tabular}{lllll}
    $31.\ x_2x_3x_4x_5^{2}$ & $32.\ x_2x_3x_4^{2}x_5$ & $33.\ x_2x_3^{2}x_4x_5$ & $34.\ x_1x_3x_4x_5^{2}$ & $35.\ x_1x_3x_4^{2}x_5$ \\
    $36.\ x_1x_3^{2}x_4x_5$ & $37.\ x_1x_2x_4x_5^{2}$ & $38.\ x_1x_2x_4^{2}x_5$ & $39.\ x_1x_2x_3x_5^{2}$ & $40.\ x_1x_2x_3x_4^{2}$ \\
    $41.\ x_1x_2x_3^{2}x_5$ & $42.\ x_1x_2x_3^{2}x_4$ & $43.\ x_1x_2^{2}x_4x_5$ & $44.\ x_1x_2^{2}x_3x_5$ & $45.\ x_1x_2^{2}x_3x_4$. 
    \end{tabular}%
  \end{center}
\medskip

For $t \geqslant 2,$
\medskip

\begin{center}
   \begin{tabular}{lll}
    $31.\ x_1x_2^{2^t-2}x_3^{2^t-1}x_4^{2^{t+1}-1}$ & $32.\ x_1x_2^{2^t-2}x_3^{2^{t+1}-1}x_4^{2^t-1}$ & $33.\ x_1x_2^{2^t-1}x_3^{2^t-2}x_4^{2^{t+1}-1}$ \\
    $34.\ x_1x_2^{2^t-1}x_3^{2^{t+1}-1}x_4^{2^t-2}$ & $35.\ x_1x_2^{2^{t+1}-1}x_3^{2^t-2}x_4^{2^t-1}$ & $36.\ x_1x_2^{2^{t+1}-1}x_3^{2^t-1}x_4^{2^t-2}$ \\
    $37.\ x_1^{2^t-1}x_2x_3^{2^t-2}x_4^{2^{t+1}-1}$ & $38.\ x_1^{2^t-1}x_2x_3^{2^{t+1}-1}x_4^{2^t-2}$ & $39.\ x_1^{2^t-1}x_2^{2^{t+1}-1}x_3x_4^{2^t-2}$ \\
    $40.\ x_1^{2^{t+1}-1}x_2x_3^{2^t-2}x_4^{2^t-1}$ & $41.\ x_1^{2^{t+1}-1}x_2x_3^{2^t-1}x_4^{2^t-2}$ & $42.\ x_1^{2^{t+1}-1}x_2^{2^t-1}x_3x_4^{2^t-2}$ \\
    $43.\ x_2^{2^t-2}x_3^{2^t-1}x_4^{2^{t+1}-1}$ & $44.\ x_2^{2^t-2}x_3^{2^{t+1}-1}x_4^{2^t-1}$ & $45.\ x_2^{2^t-1}x_3^{2^t-2}x_4^{2^{t+1}-1}$ \\
    $46.\ x_2^{2^t-1}x_3^{2^{t+1}-1}x_4^{2^t-2}$ & $47.\ x_2^{2^{t+1}-1}x_3^{2^t-2}x_4^{2^t-1}$ & $48.\ x_2^{2^{t+1}-1}x_3^{2^t-1}x_4^{2^t-2}$ \\
    $49.\ x_1^{2^t-2}x_3^{2^t-1}x_4^{2^{t+1}-1}$ & $50.\ x_1^{2^t-2}x_3^{2^{t+1}-1}x_4^{2^t-1}$ & $51.\ x_1^{2^t-2}x_2^{2^t-1}x_4^{2^{t+1}-1}$ \\
    $52.\ x_1^{2^t-2}x_2^{2^t-1}x_3^{2^{t+1}-1}$ & $53.\ x_1^{2^t-2}x_2^{2^{t+1}-1}x_4^{2^t-1}$ & $54.\ x_1^{2^t-2}x_2^{2^{t+1}-1}x_3^{2^t-1}$ \\
    
 \end{tabular}%
\end{center}

\newpage
\begin{center}
   \begin{tabular}{lll}
$55.\ x_1^{2^t-1}x_3^{2^t-2}x_4^{2^{t+1}-1}$ & $56.\ x_1^{2^t-1}x_3^{2^{t+1}-1}x_4^{2^t-2}$ & $57.\ x_1^{2^t-1}x_2^{2^t-2}x_4^{2^{t+1}-1}$ \\
    $58.\ x_1^{2^t-1}x_2^{2^t-2}x_3^{2^{t+1}-1}$ & $59.\ x_1^{2^t-1}x_2^{2^{t+1}-1}x_4^{2^t-2}$ & $60.\ x_1^{2^t-1}x_2^{2^{t+1}-1}x_3^{2^t-2}$ \\
    $61.\ x_1^{2^{t+1}-1}x_3^{2^t-2}x_4^{2^t-1}$ & $62.\ x_1^{2^{t+1}-1}x_3^{2^t-1}x_4^{2^t-2}$ & $63.\ x_1^{2^{t+1}-1}x_2^{2^t-2}x_4^{2^t-1}$ \\
    $64.\ x_1^{2^{t+1}-1}x_2^{2^t-2}x_3^{2^t-1}$ & $65.\ x_1^{2^{t+1}-1}x_2^{2^t-1}x_4^{2^t-2}$ & $66.\ x_1^{2^{t+1}-1}x_2^{2^t-1}x_3^{2^t-2}$ \\
    $67.\ x_2x_3^{2^t-2}x_4^{2^{t+1}-1}$ & $68.\ x_2x_3^{2^{t+1}-1}x_4^{2^t-2}$ & $69.\ x_2^{2^{t+1}-1}x_3x_4^{2^t-2}$ \\
    $70.\ x_1x_3^{2^t-2}x_4^{2^{t+1}-1}$ & $71.\ x_1x_3^{2^{t+1}-1}x_4^{2^t-2}$ & $72.\ x_1x_2^{2^t-2}x_4^{2^{t+1}-1}$ \\
    $73.\ x_1x_2^{2^t-2}x_3^{2^{t+1}-1}$ & $74.\ x_1x_2^{2^{t+1}-1}x_4^{2^t-2}$ & $75.\ x_1x_2^{2^{t+1}-1}x_3^{2^t-2}$ \\
    $76.\ x_1^{2^{t+1}-1}x_3x_4^{2^t-2}$ & $77.\ x_1^{2^{t+1}-1}x_2x_4^{2^t-2}$ & $78.\ x_1^{2^{t+1}-1}x_2x_3^{2^t-2}$ \\
    $79.\ x_2x_3^{2^t-2}x_4^{2^t-1}$ & $80.\ x_2x_3^{2^t-1}x_4^{2^t-2}$ & $81.\ x_2^{2^t-1}x_3x_4^{2^t-2}$ \\
    $82.\ x_1x_3^{2^t-2}x_4^{2^t-1}$ & $83.\ x_1x_3^{2^t-1}x_4^{2^t-2}$ & $84.\ x_1x_2^{2^t-2}x_4^{2^t-1}$ \\
    $85.\ x_1x_2^{2^t-2}x_3^{2^t-1}$ & $86.\ x_1x_2^{2^t-1}x_4^{2^t-2}$ & $87.\ x_1x_2^{2^t-1}x_3^{2^t-2}$ \\
    $88.\ x_1^{2^t-1}x_3x_4^{2^t-2}$ & $89.\ x_1^{2^t-1}x_2x_4^{2^t-2}$ & $90.\ x_1^{2^t-1}x_2x_3^{2^t-2}$. 
    \end{tabular}%
  \end{center}
\medskip

For $t = 2,$
\medskip

\begin{center}
    \begin{tabular}{lllll}
    $91.\ x_2x_3^{3}x_4^{3}x_5^{6}$ & $92.\ x_2x_3^{3}x_4^{6}x_5^{3}$ & $93.\ x_2x_3^{6}x_4^{3}x_5^{3}$ & $94.\ x_2^{3}x_3x_4^{3}x_5^{6}$ & $95.\ x_2^{3}x_3x_4^{6}x_5^{3}$ \\
    $96.\ x_2^{3}x_3^{3}x_4x_5^{6}$ & $97.\ x_1x_3^{3}x_4^{3}x_5^{6}$ & $98.\ x_1x_3^{3}x_4^{6}x_5^{3}$ & $99.\ x_1x_3^{6}x_4^{3}x_5^{3}$ & $100.\ x_1x_2^{3}x_4^{3}x_5^{6}$ \\
    $101.\ x_1x_2^{3}x_4^{6}x_5^{3}$ & $102.\ x_1x_2^{3}x_3^{3}x_5^{6}$ & $103.\ x_1x_2^{3}x_3^{3}x_4^{6}$ & $104.\ x_1x_2^{3}x_3^{6}x_5^{3}$ & $105.\ x_1x_2^{3}x_3^{6}x_4^{3}$ \\
    $106.\ x_1x_2^{6}x_4^{3}x_5^{3}$ & $107.\ x_1x_2^{6}x_3^{3}x_5^{3}$ & $108.\ x_1x_2^{6}x_3^{3}x_4^{3}$ & $109.\ x_1^{3}x_3x_4^{3}x_5^{6}$ & $110.\ x_1^{3}x_3x_4^{6}x_5^{3}$ \\
    $111.\ x_1^{3}x_3^{3}x_4x_5^{6}$ & $112.\ x_1^{3}x_2x_4^{3}x_5^{6}$ & $113.\ x_1^{3}x_2x_4^{6}x_5^{3}$ & $114.\ x_1^{3}x_2x_3^{3}x_5^{6}$ & $115.\ x_1^{3}x_2x_3^{3}x_4^{6}$ \\
    $116.\ x_1^{3}x_2x_3^{6}x_5^{3}$ & $117.\ x_1^{3}x_2x_3^{6}x_4^{3}$ & $118.\ x_1^{3}x_2^{3}x_4x_5^{6}$ & $119.\ x_1^{3}x_2^{3}x_3x_5^{6}$ & $120.\ x_1^{3}x_2^{3}x_3x_4^{6}$ \\
    $121.\ x_2^{3}x_3^{3}x_4^{5}x_5^{2}$ & $122.\ x_2^{3}x_3^{5}x_4^{2}x_5^{3}$ & $123.\ x_2^{3}x_3^{5}x_4^{3}x_5^{2}$ & $124.\ x_1^{3}x_3^{3}x_4^{5}x_5^{2}$ & $125.\ x_1^{3}x_3^{5}x_4^{2}x_5^{3}$ \\
    $126.\ x_1^{3}x_3^{5}x_4^{3}x_5^{2}$ & $127.\ x_1^{3}x_2^{3}x_4^{5}x_5^{2}$ & $128.\ x_1^{3}x_2^{3}x_3^{5}x_5^{2}$ & $129.\ x_1^{3}x_2^{3}x_3^{5}x_4^{2}$ & $130.\ x_1^{3}x_2^{5}x_4^{2}x_5^{3}$ \\
\end{tabular}%
  \end{center}
\medskip

For $t \geqslant 3,$
\medskip

\begin{center}
 \begin{tabular}{llr}
    $91.\ x_2^{3}x_3^{2^t-3}x_4^{2^t-2}x_5^{2^{t+1}-1}$ & $92.\ x_2^{3}x_3^{2^t-3}x_4^{2^{t+1}-1}x_5^{2^t-2}$ & \multicolumn{1}{l}{$93.\ x_2^{3}x_3^{2^{t+1}-1}x_4^{2^t-3}x_5^{2^t-2}$} \\
    $94.\ x_2^{2^{t+1}-1}x_3^{3}x_4^{2^t-3}x_5^{2^t-2}$ & $95.\ x_1^{3}x_3^{2^t-3}x_4^{2^t-2}x_5^{2^{t+1}-1}$ & \multicolumn{1}{l}{$96.\ x_1^{3}x_3^{2^t-3}x_4^{2^{t+1}-1}x_5^{2^t-2}$} \\
    $97.\ x_1^{3}x_3^{2^{t+1}-1}x_4^{2^t-3}x_5^{2^t-2}$ & $98.\ x_1^{3}x_2^{2^t-3}x_4^{2^t-2}x_5^{2^{t+1}-1}$ & \multicolumn{1}{l}{$99.\ x_1^{3}x_2^{2^t-3}x_4^{2^{t+1}-1}x_5^{2^t-2}$} \\
    $100.\ x_1^{3}x_2^{2^t-3}x_3^{2^t-2}x_5^{2^{t+1}-1}$ & $101.\ x_1^{3}x_2^{2^t-3}x_3^{2^t-2}x_4^{2^{t+1}-1}$ & \multicolumn{1}{l}{$102.\ x_1^{3}x_2^{2^t-3}x_3^{2^{t+1}-1}x_5^{2^t-2}$} \\
    $103.\ x_1^{3}x_2^{2^t-3}x_3^{2^{t+1}-1}x_4^{2^t-2}$ & $104.\ x_1^{3}x_2^{2^{t+1}-1}x_4^{2^t-3}x_5^{2^t-2}$ & \multicolumn{1}{l}{$105.\ x_1^{3}x_2^{2^{t+1}-1}x_3^{2^t-3}x_5^{2^t-2}$} \\
    $106.\ x_1^{3}x_2^{2^{t+1}-1}x_3^{2^t-3}x_4^{2^t-2}$ & $107.\ x_1^{2^{t+1}-1}x_3^{3}x_4^{2^t-3}x_5^{2^t-2}$ & \multicolumn{1}{l}{$108.\ x_1^{2^{t+1}-1}x_2^{3}x_4^{2^t-3}x_5^{2^t-2}$} \\
    $109.\ x_1^{2^{t+1}-1}x_2^{3}x_3^{2^t-3}x_5^{2^t-2}$ & $110.\ x_1^{2^{t+1}-1}x_2^{3}x_3^{2^t-3}x_4^{2^t-2}$. &  
    \end{tabular}%
 \end{center}
\medskip

For $t  = 3,$
\medskip

\begin{center}
    \begin{tabular}{lclll}
    $111.\ x_2x_3^{7}x_4^{7}x_5^{14}$ & $112.\ x_2^{7}x_3x_4^{7}x_5^{14}$ & $113.\ x_1x_3^{7}x_4^{7}x_5^{14}$ & $114.\ x_1x_2^{7}x_4^{7}x_5^{14}$ & $115.\ x_1x_2^{7}x_3^{7}x_5^{14}$ \\
    $116.\ x_1x_2^{7}x_3^{7}x_4^{14}$ & $117.\ x_1^{7}x_3x_4^{7}x_5^{14}$ & $118.\ x_1^{7}x_2x_4^{7}x_5^{14}$ & $119.\ x_1^{7}x_2x_3^{7}x_5^{14}$ & $120.\ x_1^{7}x_2x_3^{7}x_4^{14}$ \\
    $121.\ x_2^{3}x_3^{5}x_4^{14}x_5^{7}$ & $122.\ x_2^{3}x_3^{7}x_4^{5}x_5^{14}$ & $123.\ x_1^{3}x_3^{5}x_4^{14}x_5^{7}$ & $124.\ x_1^{3}x_3^{7}x_4^{5}x_5^{14}$ & $125.\ x_1^{3}x_2^{5}x_4^{14}x_5^{7}$ \\
    $126.\ x_1^{3}x_2^{5}x_3^{14}x_5^{7}$ & $127.\ x_1^{3}x_2^{5}x_3^{14}x_4^{7}$ & $128.\ x_1^{3}x_2^{7}x_4^{5}x_5^{14}$ & $129.\ x_1^{3}x_2^{7}x_3^{5}x_5^{14}$ & $130.\ x_1^{3}x_2^{7}x_3^{5}x_4^{14}$ \\
   $131.\ x_2^{3}x_3^{7}x_4^{13}x_5^{6}$ & $132.\ x_1^{3}x_3^{7}x_4^{13}x_5^{6}$ & $133.\ x_1^{3}x_2^{7}x_4^{13}x_5^{6}$ & $134.\ x_1^{3}x_2^{7}x_3^{13}x_5^{6}$ & $135.\ x_1^{3}x_2^{7}x_3^{13}x_4^{6}$\\    
$136.\ x_2x_3^{7}x_4^{14}x_5^{7}$ & $137.\ x_2x_3^{14}x_4^{7}x_5^{7}$ & $138.\ x_2^{7}x_3x_4^{14}x_5^{7}$ & $139.\ x_2^{7}x_3^{7}x_4x_5^{14}$ & $140.\ x_1x_3^{7}x_4^{14}x_5^{7}$ \\
    $141.\ x_1x_3^{14}x_4^{7}x_5^{7}$ & $142.\ x_1x_2^{7}x_4^{14}x_5^{7}$ & $143.\ x_1x_2^{7}x_3^{14}x_5^{7}$ & $144.\ x_1x_2^{7}x_3^{14}x_4^{7}$ & $145.\ x_1x_2^{14}x_4^{7}x_5^{7}$ \\
  \end{tabular}%
  \end{center}

\newpage
\begin{center}
    \begin{tabular}{lclll}
    $146.\ x_1x_2^{14}x_3^{7}x_5^{7}$ & $147.\ x_1x_2^{14}x_3^{7}x_4^{7}$ & $148.\ x_1^{7}x_3x_4^{14}x_5^{7}$ & $149.\ x_1^{7}x_3^{7}x_4x_5^{14}$ & $150.\ x_1^{7}x_2x_4^{14}x_5^{7}$ \\
    $151.\ x_1^{7}x_2x_3^{14}x_5^{7}$ & $152.\ x_1^{7}x_2x_3^{14}x_4^{7}$ & $153.\ x_1^{7}x_2^{7}x_4x_5^{14}$ & $154.\ x_1^{7}x_2^{7}x_3x_5^{14}$ & $155.\ x_1^{7}x_2^{7}x_3x_4^{14}$ \\
    $156.\ x_2^{3}x_3^{5}x_4^{7}x_5^{14}$ & $157.\ x_2^{7}x_3^{3}x_4^{5}x_5^{14}$ & $158.\ x_1^{3}x_3^{5}x_4^{7}x_5^{14}$ & $159.\ x_1^{3}x_2^{5}x_4^{7}x_5^{14}$ & $160.\ x_1^{3}x_2^{5}x_3^{7}x_5^{14}$ \\
    $161.\ x_1^{3}x_2^{5}x_3^{7}x_4^{14}$ & $162.\ x_1^{7}x_3^{3}x_4^{5}x_5^{14}$ & $163.\ x_1^{7}x_2^{3}x_4^{5}x_5^{14}$ & $164.\ x_1^{7}x_2^{3}x_3^{5}x_5^{14}$ & $165.\ x_1^{7}x_2^{3}x_3^{5}x_4^{14}$ \\
    $166.\ x_2^{3}x_3^{13}x_4^{6}x_5^{7}$ & $167.\ x_2^{3}x_3^{13}x_4^{7}x_5^{6}$ & $168.\ x_2^{7}x_3^{3}x_4^{13}x_5^{6}$ & $169.\ x_1^{3}x_3^{13}x_4^{6}x_5^{7}$ & $170.\ x_1^{3}x_3^{13}x_4^{7}x_5^{6}$ \\
    $171.\ x_1^{3}x_2^{13}x_4^{6}x_5^{7}$ & $172.\ x_1^{3}x_2^{13}x_4^{7}x_5^{6}$ & $173.\ x_1^{3}x_2^{13}x_3^{6}x_5^{7}$ & $174.\ x_1^{3}x_2^{13}x_3^{6}x_4^{7}$ & $175.\ x_1^{3}x_2^{13}x_3^{7}x_5^{6}$ \\
    $176.\ x_1^{3}x_2^{13}x_3^{7}x_4^{6}$ & $177.\ x_1^{7}x_3^{3}x_4^{13}x_5^{6}$ & $178.\ x_1^{7}x_2^{3}x_4^{13}x_5^{6}$ & $179.\ x_1^{7}x_2^{3}x_3^{13}x_5^{6}$ & $180.\ x_1^{7}x_2^{3}x_3^{13}x_4^{6}$ \\
    $181.\ x_2^{7}x_3^{11}x_4^{5}x_5^{6}$ & $182.\ x_1^{7}x_3^{11}x_4^{5}x_5^{6}$ & $183.\ x_1^{7}x_2^{11}x_4^{5}x_5^{6}$ & $184.\ x_1^{7}x_2^{11}x_3^{5}x_5^{6}$ & $185.\ x_1^{7}x_2^{11}x_3^{5}x_4^{6}$ \\
\end{tabular}%
  \end{center}
\medskip

For $t  \geqslant 4,$
\medskip

\begin{center}
    \begin{tabular}{lrr}
    $111.\ x_2^{3}x_3^{2^t-3}x_4^{2^t-1}x_5^{2^{t+1}-2}$ & \multicolumn{1}{l}{$112.\ x_2^{3}x_3^{2^t-3}x_4^{2^{t+1}-2}x_5^{2^t-1}$} & \multicolumn{1}{l}{$113.\ x_2^{3}x_3^{2^t-1}x_4^{2^t-3}x_5^{2^{t+1}-2}$} \\
    $114.\ x_2^{2^t-1}x_3^{3}x_4^{2^t-3}x_5^{2^{t+1}-2}$ & \multicolumn{1}{l}{$115.\ x_1^{3}x_3^{2^t-3}x_4^{2^t-1}x_5^{2^{t+1}-2}$} & \multicolumn{1}{l}{$116.\ x_1^{3}x_3^{2^t-3}x_4^{2^{t+1}-2}x_5^{2^t-1}$} \\
    $117.\ x_1^{3}x_3^{2^t-1}x_4^{2^t-3}x_5^{2^{t+1}-2}$ & \multicolumn{1}{l}{$118.\ x_1^{3}x_2^{2^t-3}x_4^{2^t-1}x_5^{2^{t+1}-2}$} & \multicolumn{1}{l}{$119.\ x_1^{3}x_2^{2^t-3}x_4^{2^{t+1}-2}x_5^{2^t-1}$} \\
    $120.\ x_1^{3}x_2^{2^t-3}x_3^{2^t-1}x_5^{2^{t+1}-2}$ & \multicolumn{1}{l}{$121.\ x_1^{3}x_2^{2^t-3}x_3^{2^t-1}x_4^{2^{t+1}-2}$} & \multicolumn{1}{l}{$122.\ x_1^{3}x_2^{2^t-3}x_3^{2^{t+1}-2}x_5^{2^t-1}$} \\
    $123.\ x_1^{3}x_2^{2^t-3}x_3^{2^{t+1}-2}x_4^{2^t-1}$ & \multicolumn{1}{l}{$124.\ x_1^{3}x_2^{2^t-1}x_4^{2^t-3}x_5^{2^{t+1}-2}$} & \multicolumn{1}{l}{$125.\ x_1^{3}x_2^{2^t-1}x_3^{2^t-3}x_5^{2^{t+1}-2}$} \\
    $126.\ x_1^{3}x_2^{2^t-1}x_3^{2^t-3}x_4^{2^{t+1}-2}$ & \multicolumn{1}{l}{$127.\ x_1^{2^t-1}x_3^{3}x_4^{2^t-3}x_5^{2^{t+1}-2}$} & \multicolumn{1}{l}{$128.\ x_1^{2^t-1}x_2^{3}x_4^{2^t-3}x_5^{2^{t+1}-2}$} \\
    $129.\ x_1^{2^t-1}x_2^{3}x_3^{2^t-3}x_5^{2^{t+1}-2}$ & \multicolumn{1}{l}{$130.\ x_1^{2^t-1}x_2^{3}x_3^{2^t-3}x_4^{2^{t+1}-2}$} & \multicolumn{1}{l}{$131.\ x_2^{7}x_3^{2^t-5}x_4^{2^{t+1}-3}x_5^{2^t-2}$} \\
    $132.\ x_1^{7}x_3^{2^t-5}x_4^{2^{t+1}-3}x_5^{2^t-2}$ & \multicolumn{1}{l}{$133.\ x_1^{7}x_2^{2^t-5}x_4^{2^{t+1}-3}x_5^{2^t-2}$} & \multicolumn{1}{l}{$134.\ x_1^{7}x_2^{2^t-5}x_3^{2^{t+1}-3}x_5^{2^t-2}$} \\
\end{tabular}%
   \end{center}
\medskip

\newpage
\subsection{The admissible monomials of degree $3(2^t-1)  +2^t$ in $P_5^+$}\label{dtcndtq-2}

As it is known,
$$ \mathscr B_5(3(2^t-1)  +2^t) = \mathscr B_5^0(3(2^t-1)  +2^t) \bigcup \varphi(\mathscr B_5(2^{t+1} - 4)) \bigcup B,$$ 
where $B = (\mathscr B_5^+(3(2^t-1)  +2^t) \bigcap {\rm Ker}((\widetilde {Sq}_*^0)_{(5, 3(2^t-1) + 2^t)})) = \mathscr B^+_5(\omega_{(5,t)}),$  
and\\ $ \varphi: P_5\to P_5,\ \varphi(y) = x_1x_2x_3x_4x_5y^2, \ \forall y\in P_5.$ We have\\
 $ |\mathscr B_5^0(3(2^t-1)  +2^t)| = \left\{\begin{array}{ll}
45 &\text{if}\ t = 1,\\
145 &\text{if}\ t = 2,\\
195 &\text{if}\ t \geqslant 3,
\end{array}\right.$
\ and\
$ |\mathscr B_5(2^{t+1} - 4)| = \left\{\begin{array}{ll}
1 &\text{if}\ t = 1,\\
45 &\text{if}\ t = 2,\\
190&\text{if}\ t = 3,\\
480 &\text{if}\ t = 4,\\
650 &\text{if}\ t = 5,\\
651 &\text{if}\ t \geqslant 6.
\end{array}\right.$

Set $v_t = |\mathscr B^+_5(\omega_{(5,t)})|$ and $\mathscr B^+_5(\omega_{(5,t)}) = \{b_{3(2^t-1)  +2^t, k}:\ 1\leqslant k\leqslant v_t\}.$
 Then, $v_1 = 1,\ v_2 = 60,\ v_3 = 260,\ v_t = 270$ for $t\geqslant 4.$
\medskip

For $t  = 1,$ according to Sum \cite{N.S8, N.S6}, $\mathscr B^+_5(\omega_{(5,1)}) = \big\{b_{5, 1} = x_1x_2x_3x_4x_5\big\}.$
\medskip

For $t\geqslant 2,$ the admissible monomials $b_{3(2^t-1)  +2^t, k},\ 1\leqslant k\leqslant v_t,$ are determined as follows:
\medskip

For $t = 2,$
\begin{center}
   \begin{tabular}{llll}
    $1.\ x_1x_2x_3^{2}x_4^{2}x_5^{7}$ & $2.\ x_1x_2x_3^{2}x_4^{7}x_5^{2}$ & $3.\ x_1x_2x_3^{7}x_4^{2}x_5^{2}$ & $4.\ x_1x_2^{2}x_3x_4^{2}x_5^{7}$ \\
    $5.\ x_1x_2^{2}x_3x_4^{7}x_5^{2}$ & $6.\ x_1x_2^{2}x_3^{7}x_4x_5^{2}$ & $7.\ x_1x_2^{7}x_3x_4^{2}x_5^{2}$ & $8.\ x_1x_2^{7}x_3^{2}x_4x_5^{2}$ \\
    $9.\ x_1^{7}x_2x_3x_4^{2}x_5^{2}$ & $10.\ x_1^{7}x_2x_3^{2}x_4x_5^{2}$ & $11.\ x_1x_2x_3^{2}x_4^{3}x_5^{6}$ & $12.\ x_1x_2x_3^{2}x_4^{6}x_5^{3}$ \\
    $13.\ x_1x_2x_3^{3}x_4^{2}x_5^{6}$ & $14.\ x_1x_2x_3^{3}x_4^{6}x_5^{2}$ & $15.\ x_1x_2x_3^{6}x_4^{2}x_5^{3}$ & $16.\ x_1x_2x_3^{6}x_4^{3}x_5^{2}$ \\
    $17.\ x_1x_2^{2}x_3x_4^{3}x_5^{6}$ & $18.\ x_1x_2^{2}x_3x_4^{6}x_5^{3}$ & $19.\ x_1x_2^{2}x_3^{3}x_4x_5^{6}$ & $20.\ x_1x_2^{3}x_3x_4^{2}x_5^{6}$ \\
    $21.\ x_1x_2^{3}x_3x_4^{6}x_5^{2}$ & $22.\ x_1x_2^{3}x_3^{2}x_4x_5^{6}$ & $23.\ x_1x_2^{3}x_3^{6}x_4x_5^{2}$ & $24.\ x_1x_2^{6}x_3x_4^{2}x_5^{3}$ \\
    $25.\ x_1x_2^{6}x_3x_4^{3}x_5^{2}$ & $26.\ x_1x_2^{6}x_3^{3}x_4x_5^{2}$ & $27.\ x_1^{3}x_2x_3x_4^{2}x_5^{6}$ & $28.\ x_1^{3}x_2x_3x_4^{6}x_5^{2}$ \\
    $29.\ x_1^{3}x_2x_3^{2}x_4x_5^{6}$ & $30.\ x_1^{3}x_2x_3^{6}x_4x_5^{2}$ & $31.\ x_1x_2^{2}x_3^{3}x_4^{5}x_5^{2}$ & $32.\ x_1x_2^{2}x_3^{5}x_4^{2}x_5^{3}$ \\
\end{tabular}%
  \end{center}%
\medskip

For $t = 3,$
\medskip

\begin{center}
    \begin{tabular}{lllll}
    $1.\ x_1x_2x_3^{6}x_4^{6}x_5^{15}$ & $2.\ x_1x_2x_3^{6}x_4^{15}x_5^{6}$ & $3.\ x_1x_2x_3^{15}x_4^{6}x_5^{6}$ & $4.\ x_1x_2^{15}x_3x_4^{6}x_5^{6}$ &  \\
    $5.\ x_1^{3}x_2^{3}x_3^{4}x_4^{4}x_5^{15}$ & $6.\ x_1^{3}x_2^{3}x_3^{4}x_4^{15}x_5^{4}$ & $7.\ x_1^{3}x_2^{3}x_3^{15}x_4^{4}x_5^{4}$ & $8.\ x_1^{3}x_2^{15}x_3^{3}x_4^{4}x_5^{4}$ &  \\
    $9.\ x_1^{15}x_2x_3x_4^{6}x_5^{6}$ & $10.\ x_1^{15}x_2^{3}x_3^{3}x_4^{4}x_5^{4}$ & $11.\ x_1x_2^{2}x_3^{5}x_4^{6}x_5^{15}$ & $12.\ x_1x_2^{2}x_3^{5}x_4^{15}x_5^{6}$ &  \\ 
$13.\ x_1x_2^{2}x_3^{15}x_4^{5}x_5^{6}$ & $14.\ x_1x_2^{3}x_3^{4}x_4^{6}x_5^{15}$ & $15.\ x_1x_2^{3}x_3^{4}x_4^{15}x_5^{6}$ & $16.\ x_1x_2^{3}x_3^{6}x_4^{4}x_5^{15}$ &  \\
 $17.\ x_1x_2^{3}x_3^{6}x_4^{15}x_5^{4}$ & $18.\ x_1x_2^{3}x_3^{15}x_4^{4}x_5^{6}$ & $19.\ x_1x_2^{3}x_3^{15}x_4^{6}x_5^{4}$ & $20.\ x_1x_2^{6}x_3x_4^{6}x_5^{15}$ &  \\
    $21.\ x_1x_2^{6}x_3x_4^{15}x_5^{6}$ & $22.\ x_1x_2^{6}x_3^{15}x_4x_5^{6}$ & $23.\ x_1x_2^{15}x_3^{2}x_4^{5}x_5^{6}$ & $24.\ x_1x_2^{15}x_3^{3}x_4^{4}x_5^{6}$ &  \\
    $25.\ x_1x_2^{15}x_3^{3}x_4^{6}x_5^{4}$ & $26.\ x_1x_2^{15}x_3^{6}x_4x_5^{6}$ & $27.\ x_1^{3}x_2x_3^{4}x_4^{6}x_5^{15}$ & $28.\ x_1^{3}x_2x_3^{4}x_4^{15}x_5^{6}$ &  \\
    $29.\ x_1^{3}x_2x_3^{6}x_4^{4}x_5^{15}$ & $30.\ x_1^{3}x_2x_3^{6}x_4^{15}x_5^{4}$ & $31.\ x_1^{3}x_2x_3^{15}x_4^{4}x_5^{6}$ & $32.\ x_1^{3}x_2x_3^{15}x_4^{6}x_5^{4}$ &  \\
\end{tabular}%
 \end{center}

\newpage
\begin{center}
    \begin{tabular}{lllll}
    $33.\ x_1^{3}x_2^{5}x_3^{2}x_4^{4}x_5^{15}$ & $34.\ x_1^{3}x_2^{5}x_3^{2}x_4^{15}x_5^{4}$ & $35.\ x_1^{3}x_2^{5}x_3^{15}x_4^{2}x_5^{4}$ & $36.\ x_1^{3}x_2^{15}x_3x_4^{4}x_5^{6}$ &  \\
 $37.\ x_1^{3}x_2^{15}x_3x_4^{6}x_5^{4}$ & $38.\ x_1^{3}x_2^{15}x_3^{5}x_4^{2}x_5^{4}$ & $39.\ x_1^{15}x_2x_3^{2}x_4^{5}x_5^{6}$ & $40.\ x_1^{15}x_2x_3^{3}x_4^{4}x_5^{6}$ &  \\
$41.\ x_1^{15}x_2x_3^{3}x_4^{6}x_5^{4}$ & $42.\ x_1^{15}x_2x_3^{6}x_4x_5^{6}$ & $43.\ x_1^{15}x_2^{3}x_3x_4^{4}x_5^{6}$ & $44.\ x_1^{15}x_2^{3}x_3x_4^{6}x_5^{4}$ &  \\
 $45.\ x_1^{15}x_2^{3}x_3^{5}x_4^{2}x_5^{4}$ & $46.\ x_1x_2^{2}x_3^{4}x_4^{7}x_5^{15}$ & $47.\ x_1x_2^{2}x_3^{4}x_4^{15}x_5^{7}$ & $48.\ x_1x_2^{2}x_3^{7}x_4^{4}x_5^{15}$ &  \\
    $49.\ x_1x_2^{2}x_3^{7}x_4^{15}x_5^{4}$ & $50.\ x_1x_2^{2}x_3^{15}x_4^{4}x_5^{7}$ & $51.\ x_1x_2^{2}x_3^{15}x_4^{7}x_5^{4}$ & $52.\ x_1x_2^{7}x_3^{2}x_4^{4}x_5^{15}$ &  \\
$53.\ x_1x_2^{7}x_3^{2}x_4^{15}x_5^{4}$ & $54.\ x_1x_2^{7}x_3^{15}x_4^{2}x_5^{4}$ & $55.\ x_1x_2^{15}x_3^{2}x_4^{4}x_5^{7}$ & $56.\ x_1x_2^{15}x_3^{2}x_4^{7}x_5^{4}$ &  \\
 $57.\ x_1x_2^{15}x_3^{7}x_4^{2}x_5^{4}$ & $58.\ x_1^{7}x_2x_3^{2}x_4^{4}x_5^{15}$ & $59.\ x_1^{7}x_2x_3^{2}x_4^{15}x_5^{4}$ & $60.\ x_1^{7}x_2x_3^{15}x_4^{2}x_5^{4}$ &  \\
    $61.\ x_1^{7}x_2^{15}x_3x_4^{2}x_5^{4}$ & $62.\ x_1^{15}x_2x_3^{2}x_4^{4}x_5^{7}$ & $63.\ x_1^{15}x_2x_3^{2}x_4^{7}x_5^{4}$ & $64.\ x_1^{15}x_2x_3^{7}x_4^{2}x_5^{4}$ &  \\
\end{tabular}%
 \end{center}

\newpage
\begin{center}
    \begin{tabular}{lllll}
    $225.\ x_1^{7}x_2x_3^{7}x_4^{2}x_5^{12}$ & $226.\ x_1^{7}x_2x_3^{7}x_4^{6}x_5^{8}$ & $227.\ x_1^{7}x_2x_3^{7}x_4^{8}x_5^{6}$ & $228.\ x_1^{7}x_2x_3^{7}x_4^{10}x_5^{4}$ &  \\
    $229.\ x_1^{7}x_2x_3^{10}x_4^{5}x_5^{6}$ & $230.\ x_1^{7}x_2x_3^{11}x_4^{4}x_5^{6}$ & $231.\ x_1^{7}x_2x_3^{11}x_4^{6}x_5^{4}$ & $232.\ x_1^{7}x_2x_3^{14}x_4x_5^{6}$ &  \\
    $233.\ x_1^{7}x_2^{3}x_3x_4^{4}x_5^{14}$ & $234.\ x_1^{7}x_2^{3}x_3x_4^{6}x_5^{12}$ & $235.\ x_1^{7}x_2^{3}x_3x_4^{12}x_5^{6}$ & $236.\ x_1^{7}x_2^{3}x_3x_4^{14}x_5^{4}$ &  \\
    $237.\ x_1^{7}x_2^{3}x_3^{3}x_4^{4}x_5^{12}$ & $238.\ x_1^{7}x_2^{3}x_3^{3}x_4^{12}x_5^{4}$ & $239.\ x_1^{7}x_2^{3}x_3^{5}x_4^{2}x_5^{12}$ & $240.\ x_1^{7}x_2^{3}x_3^{5}x_4^{6}x_5^{8}$ &  \\
    $241.\ x_1^{7}x_2^{3}x_3^{5}x_4^{8}x_5^{6}$ & $242.\ x_1^{7}x_2^{3}x_3^{5}x_4^{10}x_5^{4}$ & $243.\ x_1^{7}x_2^{3}x_3^{8}x_4^{5}x_5^{6}$ & $244.\ x_1^{7}x_2^{3}x_3^{9}x_4^{4}x_5^{6}$ &  \\
   $245.\ x_1^{7}x_2^{3}x_3^{9}x_4^{6}x_5^{4}$ & $246.\ x_1^{7}x_2^{3}x_3^{11}x_4^{4}x_5^{4}$ & $247.\ x_1^{7}x_2^{3}x_3^{13}x_4^{2}x_5^{4}$ & $248.\ x_1^{7}x_2^{7}x_3x_4^{2}x_5^{12}$ &  \\
    $249.\ x_1^{7}x_2^{7}x_3x_4^{6}x_5^{8}$ & $250.\ x_1^{7}x_2^{7}x_3x_4^{8}x_5^{6}$ & $251.\ x_1^{7}x_2^{7}x_3x_4^{10}x_5^{4}$ & $252.\ x_1^{7}x_2^{7}x_3^{8}x_4x_5^{6}$ &  \\
    $253.\ x_1^{7}x_2^{7}x_3^{9}x_4^{2}x_5^{4}$ & $254.\ x_1^{7}x_2^{9}x_3^{2}x_4^{5}x_5^{6}$ & $255.\ x_1^{7}x_2^{9}x_3^{3}x_4^{4}x_5^{6}$ & $256.\ x_1^{7}x_2^{9}x_3^{3}x_4^{6}x_5^{4}$ &  \\
    $257.\ x_1^{7}x_2^{11}x_3x_4^{4}x_5^{6}$ & $258.\ x_1^{7}x_2^{11}x_3x_4^{6}x_5^{4}$ & $259.\ x_1^{7}x_2^{11}x_3^{3}x_4^{4}x_5^{4}$ & $260.\ x_1^{7}x_2^{11}x_3^{5}x_4^{2}x_5^{4}$.&  \\
    \end{tabular}%
 \end{center}
\medskip

For $t\geqslant 4,$
\medskip

\begin{center}
    \begin{tabular}{lrr}
    $1.\ x_1x_2x_3^{2^t-2}x_4^{2^t-2}x_5^{2^{t+1}-1}$ & \multicolumn{1}{l}{$2.\ x_1x_2x_3^{2^t-2}x_4^{2^{t+1}-1}x_5^{2^t-2}$} & \multicolumn{1}{l}{$3.\ x_1x_2x_3^{2^{t+1}-1}x_4^{2^t-2}x_5^{2^t-2}$} \\
    $4.\ x_1x_2^{2^t-2}x_3x_4^{2^t-2}x_5^{2^{t+1}-1}$ & \multicolumn{1}{l}{$5.\ x_1x_2^{2^t-2}x_3x_4^{2^{t+1}-1}x_5^{2^t-2}$} & \multicolumn{1}{l}{$6.\ x_1x_2^{2^t-2}x_3^{2^{t+1}-1}x_4x_5^{2^t-2}$} \\
    $7.\ x_1x_2^{2^{t+1}-1}x_3x_4^{2^t-2}x_5^{2^t-2}$ & \multicolumn{1}{l}{$8.\ x_1x_2^{2^{t+1}-1}x_3^{2^t-2}x_4x_5^{2^t-2}$} & \multicolumn{1}{l}{$9.\ x_1^{2^{t+1}-1}x_2x_3x_4^{2^t-2}x_5^{2^t-2}$} \\
    $10.\ x_1^{2^{t+1}-1}x_2x_3^{2^t-2}x_4x_5^{2^t-2}$ & \multicolumn{1}{l}{$11.\ x_1^{3}x_2^{5}x_3^{2^t-6}x_4^{2^t-4}x_5^{2^{t+1}-1}$} & \multicolumn{1}{l}{$12.\ x_1^{3}x_2^{5}x_3^{2^t-6}x_4^{2^{t+1}-1}x_5^{2^t-4}$} \\
    $13.\ x_1^{3}x_2^{5}x_3^{2^{t+1}-1}x_4^{2^t-6}x_5^{2^t-4}$ & \multicolumn{1}{l}{$14.\ x_1^{3}x_2^{2^{t+1}-1}x_3^{5}x_4^{2^t-6}x_5^{2^t-4}$} & \multicolumn{1}{l}{$15.\ x_1^{2^{t+1}-1}x_2^{3}x_3^{5}x_4^{2^t-6}x_5^{2^t-4}$} \\
    $16.\ x_1x_2^{2}x_3^{2^t-3}x_4^{2^t-2}x_5^{2^{t+1}-1}$ & \multicolumn{1}{l}{$17.\ x_1x_2^{2}x_3^{2^t-3}x_4^{2^{t+1}-1}x_5^{2^t-2}$} & \multicolumn{1}{l}{$18.\ x_1x_2^{2}x_3^{2^{t+1}-1}x_4^{2^t-3}x_5^{2^t-2}$} \\
    $19.\ x_1x_2^{2^{t+1}-1}x_3^{2}x_4^{2^t-3}x_5^{2^t-2}$ & \multicolumn{1}{l}{$20.\ x_1^{2^{t+1}-1}x_2x_3^{2}x_4^{2^t-3}x_5^{2^t-2}$} & \multicolumn{1}{l}{$21.\ x_1x_2^{3}x_3^{2^t-4}x_4^{2^t-2}x_5^{2^{t+1}-1}$} \\
    $22.\ x_1x_2^{3}x_3^{2^t-4}x_4^{2^{t+1}-1}x_5^{2^t-2}$ & \multicolumn{1}{l}{$23.\ x_1x_2^{3}x_3^{2^t-2}x_4^{2^t-4}x_5^{2^{t+1}-1}$} & \multicolumn{1}{l}{$24.\ x_1x_2^{3}x_3^{2^t-2}x_4^{2^{t+1}-1}x_5^{2^t-4}$} \\
\end{tabular}%
 \end{center}

\newpage
\begin{center}
    \begin{tabular}{lrr}
  $97.\ x_1x_2^{2}x_3^{2^t-1}x_4^{2^{t+1}-4}x_5^{2^t-1}$ & \multicolumn{1}{l}{$98.\ x_1x_2^{2}x_3^{2^{t+1}-4}x_4^{2^t-1}x_5^{2^t-1}$} & \multicolumn{1}{l}{$99.\ x_1x_2^{2^t-1}x_3^{2}x_4^{2^t-1}x_5^{2^{t+1}-4}$} \\
    $100.\ x_1x_2^{2^t-1}x_3^{2}x_4^{2^{t+1}-4}x_5^{2^t-1}$ & \multicolumn{1}{l}{$101.\ x_1x_2^{2^t-1}x_3^{2^t-1}x_4^{2}x_5^{2^{t+1}-4}$} & \multicolumn{1}{l}{$102.\ x_1^{2^t-1}x_2x_3^{2}x_4^{2^t-1}x_5^{2^{t+1}-4}$} \\
$103.\ x_1^{2^t-1}x_2x_3^{2}x_4^{2^{t+1}-4}x_5^{2^t-1}$ & \multicolumn{1}{l}{$104.\ x_1^{2^t-1}x_2x_3^{2^t-1}x_4^{2}x_5^{2^{t+1}-4}$} & \multicolumn{1}{l}{$105.\ x_1^{2^t-1}x_2^{2^t-1}x_3x_4^{2}x_5^{2^{t+1}-4}$} \\
    $106.\ x_1x_2^{3}x_3^{2^t-4}x_4^{2^t-1}x_5^{2^{t+1}-2}$ & \multicolumn{1}{l}{$107.\ x_1x_2^{3}x_3^{2^t-4}x_4^{2^{t+1}-2}x_5^{2^t-1}$} & \multicolumn{1}{l}{$108.\ x_1x_2^{3}x_3^{2^t-1}x_4^{2^t-4}x_5^{2^{t+1}-2}$} \\
    $109.\ x_1x_2^{3}x_3^{2^t-1}x_4^{2^{t+1}-2}x_5^{2^t-4}$ & \multicolumn{1}{l}{$110.\ x_1x_2^{3}x_3^{2^{t+1}-2}x_4^{2^t-4}x_5^{2^t-1}$} & \multicolumn{1}{l}{$111.\ x_1x_2^{3}x_3^{2^{t+1}-2}x_4^{2^t-1}x_5^{2^t-4}$} \\
$112.\ x_1x_2^{2^t-1}x_3^{3}x_4^{2^t-4}x_5^{2^{t+1}-2}$ & \multicolumn{1}{l}{$113.\ x_1x_2^{2^t-1}x_3^{3}x_4^{2^{t+1}-2}x_5^{2^t-4}$} & \multicolumn{1}{l}{$114.\ x_1^{3}x_2x_3^{2^t-4}x_4^{2^t-1}x_5^{2^{t+1}-2}$} \\
    $115.\ x_1^{3}x_2x_3^{2^t-4}x_4^{2^{t+1}-2}x_5^{2^t-1}$ & \multicolumn{1}{l}{$116.\ x_1^{3}x_2x_3^{2^t-1}x_4^{2^t-4}x_5^{2^{t+1}-2}$} & \multicolumn{1}{l}{$117.\ x_1^{3}x_2x_3^{2^t-1}x_4^{2^{t+1}-2}x_5^{2^t-4}$} \\
  $118.\ x_1^{3}x_2x_3^{2^{t+1}-2}x_4^{2^t-4}x_5^{2^t-1}$ & \multicolumn{1}{l}{$119.\ x_1^{3}x_2x_3^{2^{t+1}-2}x_4^{2^t-1}x_5^{2^t-4}$} & \multicolumn{1}{l}{$120.\ x_1^{3}x_2^{2^t-1}x_3x_4^{2^t-4}x_5^{2^{t+1}-2}$} \\
\end{tabular}%
 \end{center}

\newpage
\begin{center}
    \begin{tabular}{lrr}
    $226.\ x_1^{3}x_2^{4}x_3^{2^t-5}x_4^{2^{t+1}-3}x_5^{2^t-2}$ & \multicolumn{1}{l}{$227.\ x_1^{3}x_2^{4}x_3^{2^{t+1}-5}x_4^{2^t-3}x_5^{2^t-2}$} & \multicolumn{1}{l}{$228.\ x_1^{3}x_2^{5}x_3^{2^t-6}x_4^{2^t-3}x_5^{2^{t+1}-2}$} \\
$229.\ x_1^{3}x_2^{5}x_3^{2^t-6}x_4^{2^{t+1}-3}x_5^{2^t-2}$ & \multicolumn{1}{l}{$230.\ x_1^{3}x_2^{5}x_3^{2^t-6}x_4^{2^t-1}x_5^{2^{t+1}-4}$} & \multicolumn{1}{l}{$231.\ x_1^{3}x_2^{5}x_3^{2^t-6}x_4^{2^{t+1}-4}x_5^{2^t-1}$} \\
    $232.\ x_1^{3}x_2^{5}x_3^{2^t-1}x_4^{2^t-6}x_5^{2^{t+1}-4}$ & \multicolumn{1}{l}{$233.\ x_1^{3}x_2^{2^t-1}x_3^{5}x_4^{2^t-6}x_5^{2^{t+1}-4}$} & \multicolumn{1}{l}{$234.\ x_1^{2^t-1}x_2^{3}x_3^{5}x_4^{2^t-6}x_5^{2^{t+1}-4}$} \\
   $235.\ x_1^{3}x_2^{5}x_3^{2^t-5}x_4^{2^t-4}x_5^{2^{t+1}-2}$ & \multicolumn{1}{l}{$236.\ x_1^{3}x_2^{5}x_3^{2^t-5}x_4^{2^{t+1}-2}x_5^{2^t-4}$} & \multicolumn{1}{l}{$237.\ x_1^{3}x_2^{5}x_3^{2^t-5}x_4^{2^t-2}x_5^{2^{t+1}-4}$} \\
    $238.\ x_1^{3}x_2^{5}x_3^{2^t-5}x_4^{2^{t+1}-4}x_5^{2^t-2}$ & \multicolumn{1}{l}{$239.\ x_1^{3}x_2^{5}x_3^{2^{t+1}-5}x_4^{2^t-4}x_5^{2^t-2}$} & \multicolumn{1}{l}{$240.\ x_1^{3}x_2^{5}x_3^{2^{t+1}-5}x_4^{2^t-2}x_5^{2^t-4}$} \\
    $241.\ x_1^{3}x_2^{5}x_3^{2^t-1}x_4^{2^{t+1}-6}x_5^{2^t-4}$ & \multicolumn{1}{l}{$242.\ x_1^{3}x_2^{5}x_3^{2^{t+1}-6}x_4^{2^t-4}x_5^{2^t-1}$} & \multicolumn{1}{l}{$243.\ x_1^{3}x_2^{5}x_3^{2^{t+1}-6}x_4^{2^t-1}x_5^{2^t-4}$} \\
    $244.\ x_1^{3}x_2^{2^t-1}x_3^{5}x_4^{2^{t+1}-6}x_5^{2^t-4}$ & \multicolumn{1}{l}{$245.\ x_1^{2^t-1}x_2^{3}x_3^{5}x_4^{2^{t+1}-6}x_5^{2^t-4}$} & \multicolumn{1}{l}{$246.\ x_1^{3}x_2^{5}x_3^{2^{t+1}-6}x_4^{2^t-3}x_5^{2^t-2}$} \\
    $247.\ x_1^{3}x_2^{7}x_3^{2^t-7}x_4^{2^t-4}x_5^{2^{t+1}-2}$ & \multicolumn{1}{l}{$248.\ x_1^{3}x_2^{7}x_3^{2^t-7}x_4^{2^{t+1}-2}x_5^{2^t-4}$} & \multicolumn{1}{l}{$249.\ x_1^{7}x_2^{3}x_3^{2^t-7}x_4^{2^t-4}x_5^{2^{t+1}-2}$} \\
\end{tabular}%
 \end{center}
\medskip

For $t = 4,$
$$ 269.\ x_1^7x_2^{11}x_3^{13}x_4^{14}x_5^{16}\ \ \ \ \ 270.\ x_1^7x_2^{11}x_3^{13}x_4^{18}x_5^{12}.$$
\medskip

For $t \geq 5,$
$$ 269.\ x_1^{7}x_2^{11}x_3^{2^t-11}x_4^{2^t-6}x_5^{2^{t+1}-4}\ \ \ \ \ 270.\ x_1^{7}x_2^{11}x_3^{2^t-11}x_4^{2^{t+1}-6}x_5^{2^{t}-4}.$$


\begin{thebibliography}{25}

\bibitem{Adams}
 J. F. Adams, \emph{On the structure and applications of the Steenrod algebra}, Comment. Math. Helv., 72 (1958), 180-214.

\bibitem{J.B}
J. M. Boardman, 
\emph{Modular representations on the homology of power of real projective space,} in: M. C. Tangora (Ed.), Algebraic Topology, Oaxtepec, 1991, in: Contemp. Math., 146 (1993), 49-70.

\bibitem{W.B}
W. Browder, \emph{The Kervaire invariant of framed manifolds and its generalization}, Ann. of Math. (2) 90 (1969), 157-186.

\bibitem{B.H.H}
R. R. Bruner, L. M. H\`a and N. H. V. H\uhorn ng,
\emph{On behavior of the algebraic transfer,} Trans. Amer. Math. Soc., 357 (2005), 437-487.

\bibitem{C.H}
M. C. Crabb and J. R. Hubbuck,
\emph{Representations of the homology of BV and the Steenrod algebra II,} Algebra Topology: new trend in localization and periodicity, Progr. Math., 136 (1996), 143-154.

\bibitem{T.C}
T. W. Chen, 
\emph{Determination of  $\mbox{Ext}^{5,*}_{\mathscr{A}}(\mathbb{Z}/2, \mathbb{Z}/2),$} Topol. Appl., 158 (2011), 660-689.

\bibitem{Chon-Ha1}
P. H. Ch\ohorn n and L. M. H\`a, 
\emph{On May spectral sequences and the algebraic transfer,} Manuscripta Math., 138 (2012), 141-160.

\bibitem{Chon-Ha2}
P. H. Ch\ohorn n and L. M. H\`a, 
\emph{On the May spectral sequence and the algebraic transfer II,} Topol. Appl., 178 (2014), 372-383.

\bibitem{Ha}
L. M. H\`a, 
\emph{Sub-Hopf algebras of the Steenrod algebra and the Singer transfer}, in: Proceedings of the School and Conference in Algebraic Topology, H\`a N\d \ocircumflex i 2004, in: Geom. Topol. Publ., 11 (2007), 101-124.

\bibitem{Hill}
M. A. Hill, M. J. Hopkins and D. C. Ravenel, 
\emph{On the non-existence of elements of kervaire invariant one}, Ann. of Math. (2) 184 (2016), 1-262.

\bibitem{V.H2}
N. H. V. H\uhorn ng, 
\emph{The cohomology of the Steenrod algebra and representations of the general linear groups,} Trans. Amer. Math. Soc., 357 (2005), 4065-4089.

\bibitem{M.K}
 M. Kameko,
\emph{Products of projective spaces as Steenrod modules,} PhD. Thesis, The Johns Hopkins University, ProQuest LLC, Ann Arbor, MI, 1990. 

\bibitem{L.M}
 W.H. Lin and M. Mahowald, \emph{The Adams spectral sequence for Minami's Theorem}, in: Proceedings of the Northwestern Homotopy Theory Conference (Providence, Rhode Island) (H. R. Miller and S. B. Priddy, eds.), Contemp. Math. vol. 220, 1983, pp. 143-177.

\bibitem{W.L}
 W.H. Lin, 
\emph{$\mbox{Ext}^{4,*}_{\mathscr{A}}(\mathbb{Z}/2, \mathbb{Z}/2)$ and $\mbox{Ext}^{5,*}_{\mathscr{A}}(\mathbb{Z}/2, \mathbb{Z}/2),$} Topology Appl., 155 (2008), 459-496.

\bibitem{Mahowald}
M. Mahowald, \emph{A new infinite family in $_{2}\pi^{S}_*$}, Topology 16 (1977),  249-256.

\bibitem{N.M}
N. Minami, \emph{The Adams spectral sequence and the triple transfer}, Amer. J. Math. 117 (1995), 965-985.

\bibitem{N.M2}
 N. Minami, \emph{The iterated transfer analogue of the new doomsday conjecture}, Trans. Amer. Math. Soc. 351 (1999), 2325-2351.

\bibitem{S.M}
S.A. Mitchell, \emph{Splitting $B(\mathbb Z/p)^{n}$ and $BT^{n}$ Via Modular Representation Theory}, Math. Z. 189 (1985), 1-9.

\bibitem{M.M1}
 M. F. Mothebe,
\emph{Generators of the polynomial algebra $\mathbb F_2[x_1, x_2, \ldots, x_n]$ as a module over the Steenrod algebra}, PhD. thesis, The University of   Manchester, 1997.

\bibitem{M.M2}
 M. F. Mothebe,
 \emph{The admissible monomial basis for the polynomial algebra in degree thirteen}, East-West J. of Mathematics, 18 (2016), 151-170.

\bibitem{T.N} 
T. N. Nam, 
\emph{$\mathcal{A}$-g\'en\'erateurs g\'en\'eriques pour l'alg\`ebre polynomiale,} Adv. Math., 186 (2004), 334-362.

\bibitem{T.N2} 
T.N. Nam, 
\emph{Transfert alg\'ebrique et action du groupe lin\'eaire sur les puissances divis\' ees modulo 2}, Ann. Inst. Fourier (Grenoble) 58 (2008), 1785-1837.

\bibitem{F.P}
 F. P. Peterson,
\emph{Generators of  $H^*(\mathbb{R}P^{\infty}\times \mathbb{R}P^{\infty})$ as a module over the Steenrod algebra,} Abstracts Amer. Math. Soc., No. 833, April 1987.

\bibitem{P.S}
\DJ. V. Ph\'uc and N. Sum, 
\emph{On the generators of the polynomial algebra as a module over the Steenrod algebra,} C.R.Math. Acad. Sci. Paris, 353 (2015), 1035-1040.

\bibitem{P.S2}
\DJ. V. Ph\'uc and N. Sum, 
\emph{On a minimal set of generators for the polynomial algebra of five variables as a module over the Steenrod algebra},  Acta Math. Vietnam., 42 (2017), 149-162.

\bibitem{D.P1}
 \DJ. V. Ph\'uc, 
\emph{The hit problem for the polynomial algebra of five variables in degree seventeen and its application},  East-West J. Math., 18 (2016), 27-46.

\bibitem{D.P4}
 \DJ. V. Ph\'uc, 
\emph{$\mathcal A$-generators for the polynomial algebra of five variables in degree $5(2^t-1) + 6.2^t$}, Commun. Korean Math. Soc., 35 (2020), 371-399.

\bibitem{D.P5}
\DJ.V. Ph\'uc, 
\emph{On Peterson's open problem and representations of the general linear groups}, J. Korean Math. Soc.  58 (2021), 643-702.

\bibitem{D.P6}
\DJ.V. Ph\'uc, 
\emph{On generators of the unstable $\mathscr A$-module $H^{*}((K(\mathbb F_2, 1))^{\times 5})$ in a generic degree and applications}, Preprint (2021), submitted for publication, available online at  \url{https://www.researchgate.net/publication/348564137}. 

\bibitem{D.P7}
\DJ.V. Ph\'uc, 
\emph{The hit problem of five variables in the generic degree $5(2^{s}-1) + 42.2^{s}$ and its application},  Preprint (2021) , submitted for publication, available online at \url{https://www.researchgate.net/publication/350430054}.

\bibitem{D.P8}
\DJ.V. Ph\'uc, \emph{On the hit problem for the algebra $H^{*}(B\mathbb F_2^{s}, \mathbb F_2)$ as a module over the mod two Steenrod algebra and its applications},  Preprint (2021) , submitted for publication, available online at \url{https://www.researchgate.net/publication/350236648}.

\bibitem{D.P9}
\DJ.V. Ph\'uc, \emph{On the dimension of $H^{*}((\mathbb Z_2)^{\times t}, \mathbb Z_2)$ as a module over Steenrod ring},  Preprint (2021) ,submitted for publication, available online at \url{https://www.researchgate.net/publication/350063962}.

\bibitem{D.P10}
\DJ.V. Ph\'uc, \emph{Structure of the space of $GL_4(\mathbb Z_2)$-coinvariants $\mathbb Z_2\otimes_{GL_4(\mathbb Z_2)} PH_*(\mathbb Z_2^4, \mathbb Z_2)$ in some generic degrees and its application},  Preprint (2021) , submitted for publication, available online at \url{https://www.researchgate.net/publication/350592289}.

\bibitem{D.P11}
\DJ.V. Ph\'uc, \emph{A note on the modular representation on the $\mathbb Z/2$-homology groups of the fourth power of real projective space and its application},  Preprint (2021) , submitted for publication, available online at \url{https://www.researchgate.net/publication/350988977}.

\bibitem{D.P12}
\DJ.V. Ph\'uc, \emph{On the lambda algebra and Singer's cohomological transfer}, Preprint (2021) , submitted for publication, available online at \url{https://www.researchgate.net/publication/352017781}.


\bibitem{S.P}
S. Priddy, 
\emph{On characterizing summands in the classifying space of a group,} I, Amer. Jour. Math., 112 (1990), 737-748.

\bibitem{R.S}
J. Repka and P. Selick, 
\emph{On the subalgebra of $H_*((\mathbb{R}P^{\infty})^n; \mathbb{F}_2$ ) annihilated by Steenrod operations,}  J. Pure Appl. Algebra, 127 (1998), 273-288.

\bibitem{J.S}
J. H. Silvermn, 
\emph{Hit polynomials and the canonical antiautomorphism of the Steenrod algebra,} Proc. Amer. Math. Soc., 123 (1995), 627-637.

\bibitem{S.S}
J. H. Silverman and W. M. Singer, 
\emph{On the action of Steenrod squares on polynomial algebras II}, J. Pure Appl. Algebra, 98 (1995), 95-103.

\bibitem{W.S1}
W. M. Singer, 
\emph{The transfer in homological algebra,} Math. Z., 202 (1989), 493-523.

\bibitem{W.S2}
W. M. Singer, 
\emph{On the action of the Steenrod squares on polynomial algebras,}  Proc. Amer. Math. Soc., 111 (1991), 577-583.

\bibitem{SE}
 N. E. Steenrod and D. B. A. Epstein, 
\emph{Cohomology operations,}Annals of Mathematics Studies 50, Princeton University Press, Princeton N.J, 1962.

\bibitem{N.S1}
N. Sum, 
\emph{The negative answer to Kameko's conjecture on the hit problem,} Adv. Math., 225 (2010), 2365-2390.

\bibitem{N.S2}  
N. Sum, 
\emph{On the hit problem for the polynomial algebra}, C. R. Math. Acad. Sci. Paris, 351 (2013) 565-568.

\bibitem{N.S8}
N. Sum, 
\emph{On the Peterson hit problem of five variables and its applications to the fifth Singer transfer,} East-West J. of Mathematics, 16 (2014), 47-62.

\bibitem{N.S4}
N. Sum, 
\emph{On the Peterson hit problem ,} Adv. Math., 274 (2015), 432-489.

\bibitem{N.S5}
N. Sum, 
\emph{The hit problem and the algebraic transfer in some degrees}, East-West J. of Mathematics, 20 (2018), 158-179.

\bibitem{N.S7}
N. Sum, 
\emph{On the determination of the Singer transfer,} Vietnam Journal of Science, Technology and Engineering, 60 (2018), 3-16.

\bibitem{N.S6}
N. Sum, 
\emph{The squaring operation and the Singer algebraic transfer}, Vietnam J. Math. (2020), DOI: 10.1007/s10013-020-00423-1.

\bibitem{Tangora}
M. C. Tangora, 
\emph{On the cohomology of the Steenrod algebra},  Math. Z., 116 (1970), 18-64.

\bibitem{W.W}
G. Walker and R. M. W. Wood, 
\emph{Young tableaux and the Steenrod algebra}, in: Proceedings of the International School and Conference in Algebraic Topology, H\`a N\d \ocircumflex i 2004, Geom. Topol. Monogr., Geom. Topol. Publ., Coventry, 11 (2007), 379-397.

\bibitem{W.W1}
G. Walker and R. M. W. Wood, 
\textit{Weyl modules and the mod 2 Steenrod algebra}, J. Algebra, 311 (2007), 840-858.

\bibitem{W.W2}
G. Walker and R. M. W. Wood, 
\textit{Polynomials and the mod 2 Steenrod Algebra: Volume 1, The Peterson hit problem}; In: London Math. Soc.  Lecture Note Ser., Cambridge Univ. Press, January 11, 2018.

\bibitem{R.W}
 R. M. W. Wood, 
\emph{Steenrod squares of polynomials and the Peterson conjecture,} Math. Proc. Cambriges Phil. Soc., 105 (1989), 307-309.

\bibitem{R.W2}
 R. M. W. Wood, 
\emph{Problems in the Steenrod algebra}, Bull. London Math. Soc., 30 (1998), 449-517.
\end{thebibliography}
\end{document}